%% file: unlenr.tex
\documentclass[11pt]{amsart}

\usepackage[a4paper,margin=2.5cm]{geometry}

\usepackage{graphicx}             
\usepackage{xcolor}               
\usepackage{microtype,stmaryrd}

\usepackage{bbm}
\usepackage{mathrsfs}
\usepackage[font=sf, labelfont={sf,bf}, margin=1cm]{caption}
\usepackage{dsfont,amsmath,amsfonts,amsthm,amssymb}
\usepackage[utf8]{inputenc}
\usepackage{color}
\usepackage{verbatim}
\usepackage{epsfig}
\usepackage{graphics}
\usepackage{mathtools}
\usepackage{array}
\usepackage{enumerate}
\usepackage[ruled]{algorithm}
\usepackage{algorithmic}
\usepackage{bbm}
\usepackage{wrapfig}
\usepackage{caption}
\usepackage{subcaption}
\usepackage[all, knot, cmtip]{xy} 
\usepackage{wrapfig}
\usepackage[pdfpagemode=UseNone,bookmarksopen=false,colorlinks=true,urlcolor=blue,citecolor=blue,citebordercolor=blue,linkcolor=blue]{hyperref}

\include{macros}

\title{Random enriched trees with applications to random graphs}

\author{Benedikt Stufler}

\thanks{The author is supported by the German Research Foundation DFG, STU 679/1-1}

\address[Benedikt Stufler]{Unit\'e de Math\'ematiques Pures et Appliqu\'ees,
	\'Ecole Normale Sup\'erieure de Lyon }
\email{benedikt.stufler@ens-lyon.fr}

\begin{document}

\begin{abstract}
	We establish limit theorems that describe the asymptotic local and global geometric behaviour of random  enriched trees considered up to symmetry. We apply these general results to random unlabelled weighted rooted graphs and uniform random unlabelled $k$-trees that are rooted at a $k$-clique of distinguishable vertices. For both models we establish a Gromov--Hausdorff scaling limit, a Benjamini--Schramm limit, and a local weak limit that describes the asymptotic shape near the fixed root. 
\end{abstract}

\maketitle


\section{Introduction}
\label{sec:intro}

The study of  large random discrete structures lies at the intersection of probability theory and combinatorics.  A combinatorial approach often involves using the framework of combinatorial classes to express the quantities under consideration in terms of coefficients of power series, and applying analytic tools such as singularity analysis or saddle-point methods to obtain very precise limits and concentration results \cite{MR2484382, MR2483235, MR1871555, MR2735332, MR3068033, 2016arXiv160504206C}. From a probabilistic viewpoint, the focus is on  establishing graph limits describing the asymptotic shape, either locally in so called local weak limits  \cite{MR2013797, 2014arXiv1412.6911S, MR3083919, MR2243873, MR1873300,   MR3256879,MR3183575}, or globally in Gromov--Hausdorff scaling limits \cite{MR3342658, MR3414449, 2015arXiv150306738A,  MR2336042, MR2778796, MR3112934, MR3070569}, and more recently, also on an intermediate scale in local Gromov--Hausdorff scaling limits \cite{2016arXiv160801129B, 2015arXiv151101028W}.

In this context, many of the objects under consideration  such as graphs and planar maps are endowed with an operation of the symmetric group, and it is natural  to study the corresponding unlabelled objects, that is, the orbits under this group action as representatives of objects considered up to symmetry.  For some types of planar maps this is not a particularly interesting endeavour, as their study may often be reduced to half-edge rooted maps which admit only trivial symmetries. On the other hand, a  variety of  discrete structures such as graph classes with constraints exhibit a highly complex and interesting symmetric structure. For example, the precise asymptotic number of random labelled planar graphs has been obtained roughly a decade ago in the celebrated work by Gim\'enez and Noy~\cite{MR2476775}, but the asymptotic number of unlabelled planar graphs is  still unknown and obtaining it is surely one of the central contemporary problems in enumerative combinatorics. 

The study of  objects considered up to symmetry poses a particular challenge.  Probabilistic approaches in the past treat  models of random unordered trees~\cite{2015arXiv150207180P, 2014arXiv1412.6333S, MR2829313, MR3050512, 2016arXiv160408287W}. Combinatorial results on more complex structures  were obtained for example by Bodirsky, Fusy and Kang \cite{MR2350456}  and Kraus~\cite{MR2735357} for models of unlabelled outerplanar graph, by Drmota, Fusy, Kang, Kraus and Ru{\'e}~\cite{MR2873207} for so called families of subcritical  classes of unlabelled graphs, and Drmota and Jin~\cite{MR3509385} and Gainer-Dewar, Gessel and Ira \cite{MR3213312} for unlabelled $k$-dimensional trees. 

In the present work, we obtain probabilistic limits for a large family of random combinatorial objects considered up to symmetry, including random unlabelled rooted graphs that are sampled according to weights on their $2$-connected components, and random front-rooted unlabelled $k$-dimensional trees. Rather than treating each model individually, we take a unified approach and establish a set of limit theorems that apply to the abstract family of random unlabelled $\cR$-enriched trees, with the class $\cR$ ranging over all combinatorial classes. The concept of enriched trees goes back to Labelle \cite{MR642392}. Roughly speaking, given a class $\cR$ of combinatorial objects, an $\cR$-enriched tree is a rooted tree together with a function that assigns to each vertex an $\cR$-structure {\em on} its offspring. The model we consider is an unlabelled $\cR$-enriched tree with $n$ vertices considered up to symmetry, that we sample with probability proportional to a weight formed by taking the product of weights assigned to its local $\cR$-structures. The limits are formed as $n$ becomes large, possibly along a shifted  sublattice of the integers. Of course it also makes sense to study random labelled pendants of enriched trees, and  this endeavour is undertaken in~\cite{enr}.

Recall that a symmetry may be defined as a combinatorial structure together with an automorphism. Our approach uses an encoding of symmetries of an $\cR$-enriched tree by a $\Sym(\cR)$-enriched plane tree, that is, a plane tree where each vertex is endowed with a local $\cR$-symmetry. We construct two infinite limit objects in terms of random trees enriched with local symmetries, and establish  weak limits that describe the asymptotic behaviour of the $o(\sqrt{n})$-vicinity of the  root vertex of the random enriched tree, and the $o(\sqrt{n})$-vicinity near a uniformly at random selected node. For the latter, some inspiration was taken by Aldous' approach  \cite{MR1102319} on asymptotic fringe distributions. In order to study global geometric properties, we define metric spaces based on random unlabelled enriched trees that are patched together from a random cover by small  spaces. Using a size-biased $\Sym(\cR)$-enriched tree that is strongly related to the local weak limit, we study the asymptotic global metric structure as the number of points of this model of random metric spaces becomes large, resulting in a Gromov--Hausdorff scaling limit.

In order to illustrate the scope of our results, we provide applications to specific models of random unlabelled graphs. The first model considered is that of random unlabelled rooted connected graphs sampled with probability proportional to a product of weights assigned to their $2$-connected components. For this model, we obtain a local weak limit that describes the asymptotic vicinity near the fixed root, a Benjamini--Schramm limit that describes the asymptotic shape near a random vertex, and a Gromov--Hausdorff scaling limit. Moreover, we obtain sharp tail bounds for the diameter. In the two local limits, we even obtain total variational convergence of arbitrary $o(\sqrt{n})$-sized neighbourhoods of the fixed root and random root, which is best-possible as the convergence fails for neighbourhoods whose radius is comparable to $\sqrt{n}$. The setting we consider explicitly includes uniform random unlabelled rooted graphs from so called subcritical graph classes introduced in \cite{MR2873207}, such as series-parallel graphs, outerplanar graphs, and cacti graphs. As for extremal graph parameters, our results also establish the correct order of the diameter. The maximum degree and largest $2$-connected component are shown to have typically order $O(\log n)$. In \cite{MR2873207} additive parameters of these graphs such as the degree distribution were studied using analytic methods. The two local limits add a probabilistic interpretation to the limit degree distributions obtained in \cite{MR2873207} for the degree of a random vertex and of the fixed root. Furthermore, general results by Kurauskas \cite[Thm. 2.1]{2015arXiv150408103K}) and Lyons \cite[Thm. 3.2]{MR2160416} for Benjamini--Schramm convergent sequences of random graphs may be applied to deduce laws of large numbers for subgraph count asymptotics and spanning tree count asymptotics in terms of the Benjamini--Schramm limit.

Our general results also apply to random unlabelled $k$-trees that are rooted at a front of distinguishable vertices.  A $k$-tree consists either of a complete graph with $k$ vertices, or is obtained from a smaller $k$-tree by adding a vertex and connecting it with $k$ distinct vertices of the smaller $k$-tree.  Such objects are interesting from a combinatorial point of view, as their enumeration problem has a long history, see \cite{MR0237382,MR3007180,MR1888841,MR0299535,MR2577935,MR0234868,MR0357214}. They are also interesting from an algorithmic point of view, as many NP-hard problems on graphs have polynomial algorithms when restricted to $k$-trees \cite{MR985145,MR2484635}. Employing recent results for limits of random unlabelled Gibbs partitions~\cite{2016arXiv161001401S}, we obtain a local weak limit for unlabelled front-rooted $k$-trees that describes the total variational asymptotic behaviour of arbitrary $o(\sqrt{n})$-neighbourhoods of the root-front. We also obtain a Benjamini--Schramm limit describing the asymptotic shape of $o(\sqrt{n})$-neighbourhoods of a uniformly at random selected vertex. Furthermore, we obtain a Gromov--Hausdorff scaling limit. For all three limits, a concentration result is required that relates the distances of certain points with respect to the $k$-tree metric and to a tree-metric in the underlying representation by trees endowed with local symmetries. We obtain this intermediate result by locating a hidden Markov chain and applying a  large deviation inequality by  Lezaud~\cite{MR1627795} for functions on  non-reversible Markov processes.

As a third application, our results also yield local weak limits, Benjamini--Schramm limits and scaling limits for a family of random unordered trees drawn according to weights assigned to the vertex degrees.

\subsection*{Plan of the paper}
Section~\ref{sec:intro} gives an informal description of the topic and main applications. Section~\ref{sec:prel} fixes notation related to graphs, trees and $k$-trees, and recalls necessary background on local weak convergence, Gromov--Hausdorff convergence and further properties of large critical Galton--Watson trees. Section~\ref{sec:combspec} is a concise introduction to the combinatorial framework of species of structures and Boltzmann distributions, with a focus on the decomposition of symmetries. Section~\ref{sec:gibbs} discusses a limit theorem for unlabelled Gibbs partitions, that we are going to use in our applications. Section~\ref{sec:tools} explicitly states some probabilistic and combinatorial tools related to random walks and Markov chains. Section~\ref{sec:limits} presents the contributions of the present paper in detail. Specifically, Subsection~\ref{sec:intro1} introduce the model of unlabelled $\cR$-enriched trees under consideration, and discusses combinatorial bijections that show how this generalizes various models of random graphs considered up to symmetry, in particular unlabelled block-weighted rooted graphs and unlabelled front-rooted $k$-trees. Subsection~\ref{sec:locunl} builds the framework regarding the local weak limits of unlabelled enriched trees with respect to the fixed root vertex and with respect to a randomly selected point.  Subsection~\ref{sec:partA} introduces a general model of random metric spaces based on random unlabelled enriched trees, and establishes a scaling limit and sharp diameter tail-bounds. Subsection~\ref{sec:unlablocal} presents our applications to random weighed unlabelled connected rooted graphs, in particular a scaling limit with respect to the first-passage percolation metric, sharp diameter tail-bounds, a local weak limit and a Benjamini--Schramm limit. Subsection~\ref{sec:apktr} discusses applications to random unlabelled front-rooted $k$-dimensional trees, for which a scaling limit, a local weak limit and a Benjamini--Schramm limit are established. Subsection~\ref{sec:polya} presents further applications to a family of simply generated unlabelled unrooted trees. In Section~\ref{sec:proo} we collect all proofs.

\section{Preliminaries}
\label{sec:prel}

\subsection{Notation}
\label{sec:not}
Throughout, we set
\[
\ndN=\{1,2,\ldots\}, \qquad \ndN_0 = \{0\} \cup \ndN, \qquad [n]=\{1,2,\ldots, n\}, \qquad n \in \ndN_0.
\]
The set of non-negative real numbers is denoted by $\ndR_{\ge 0}$. We usually assume that all considered random variables are defined on a common probability space whose measure we denote by $\mathbb{P}$, and let $\mathbb{L}_p$ denote the corresponding space of $p$-integrable real-valued functions.  All unspecified limits are taken as $n$ becomes large, possibly taking only values in a subset of the natural numbers.  We write $\convdis$ and $\convp$ for convergence in distribution and probability, and $\eqdist$ for equality in distribution. An event holds with high probability, if its probability tends to $1$ as $n$ tends to infinity..
We let $O_p(1)$ denote an unspecified random variable $X_n$ of a stochastically bounded sequence $(X_n)_n$, and write $o_p(1)$ for a random variable $X_n$ with $X_n \convp 0$. We write $\cL(X)$ to denote the law of a random variable $X$. The total variation distance of measures and random variables is denoted by $d_{\textsc{TV}}$. Given a power series $f(z)$, we let $[z^n]f(z)$ denote the coefficient of $z^n$ in $f(z)$.

\subsection{Graphs, trees and k-trees}
\label{sec:graph}

We are going to consider \emph{simple} graphs, that have no loops or parallel edges. The vertices that are adjacent to a vertex $v$ in a graph $G$ are its \emph{neighbourhood}. The cardinality of its neighbourhood is called the \emph{degree} of the vertex $v$, and denoted by $d_G(v)$. A graph is termed \emph{connected}, if any two vertices may be joined by a \emph{path}. More generally, for $k\ge 1$ we say a graph $G$ is $k$-connected, if it has at least $k+1$ vertices and deleting any $k-1$ of them does not disconnect the graph. A {\em cutvertex} is a vertex whose removal disconnects the graph. Hence $2$-connected graphs are graphs without cutvertices and size at least three. 

A {\em graph isomorphism} between graphs $G$ and $H$ is a bijection between their vertex sets such any two vertices in $G$ are joined by an edge if and only if their images in $H$ are. In this case we say the two graphs are \emph{isomorphic}. We say a graph is \emph{rooted}, if one of its vertices is distinguished. Graph isomorphisms between rooted graphs are required to map the roots to each other. More generally, we may consider graphs with an ordered number of distinguishable root-vertices, that must be respected by graph isomorphisms. A graph considered up to isomorphism is an {\em unlabelled graph}. That is, any two unlabelled graphs are distinct if they are not isomorphic. Formally, unlabelled graphs are defined as isomorphism classes of graphs. Unlabelled rooted graphs are defined analogously.

A \emph{tree} is a graph that is connected and does not contain circles. In a rooted tree, we say the vertices lying on the path between a vertex $v$ and the root are the \emph{ancestors} of $v$. The vertices that are joined to $v$ by an edge but are not ancestors are its \emph{offspring} or \emph{sons}. The collection of the sons of a vertex is its \emph{offspring set}. The cardinality of the offspring set of a vertex $v$ in a rooted tree $A$ is its \emph{outdegree} and denoted by $d_A^+(v)$. Unlabelled rooted trees are also termed~\emph{Pólya}-trees, in honour of George Pólya.

The complete graph with $n$ vertices is denoted by $K_n$. That is, in $K_n$ any two distinct vertices are connected. A subgraph of a graph is termed an $n$-\emph{clique}, if its isomorphic to $K_n$. A \emph{k-tree} is a graph that may be constructed by starting with a $k$-clique, and adding in each step a new vertex that gets connected with $k$ arbitrarily chosen distinct vertices of the previously constructed graph. The $k$-cliques of a $k$-tree are also called its \emph{fronts}, and the $k+1$-cliques its \emph{hedra}. In the present work, we are only considering $k$-trees that are rooted at a front of distinguishable vertices.

A \emph{block} $B$ of a graph $G$ is a subgraph that is inclusion maximal with the property of being either an isolated vertex, a $2$-clique, or $2$-connected. Connected graphs have a tree-like block-structure, whose details are explicitly given for example in Diestel's book \cite[Ch. 3.1]{MR2744811}. We mention a few properties, that we are going to use. Any two blocks of $G$ overlap in at most one vertex. The cutvertices of $G$ are precisely the vertices that belong to more than one block. 

Any connected graph $C$ is naturally equipped with the \emph{graph-metric} on its vertex set, that assigns to any two vertices the minimum number of edges required to join them. We usually denote by $d_C(\cdot, \cdot)$. \label{mark:dc} Given a vertex $v \in C$ and an integer $k \ge 0$, the \emph{$k$-neighbourhood} $V_k(C,v)$ is the subgraph induced by all vertices with distance at most $k$ from $v$. \label{mark:nv} We regard $V_k(C,v)$ as rooted at the vertex $v$. The \emph{diameter} $\Di(C)$ is the supremum of all distances between pairs of vertices. For a rooted graph $C^\bullet$, we may also consider its \emph{height} $\He(C^\bullet)$, which is the supremum of all distances of vertices from the root of $C^\bullet$. Given a vertex $v$, we let $\he_{C^\bullet}(v)$ denote its distance from the root.

Another metric on $C$ is the \emph{block-metric} $d_{\textsc{block}}$. \label{mark:db} The block-distance between any two vertices of $C$ is given by the minimum number of blocks required to cover any  joining path. By standard properties of the block-structure of connected graphs, the choice of the joining path does not matter. For any vertex $v \in C$ and any integer $k \ge 0$ we let $U_k(C,v)$ denote the \emph{$k$-block-neighbourhood}, that is, the subgraph induced by all vertices with block-distance at most $k$. \label{mark:nu} We regard $U_k(C,v)$ as rooted at the vertex $v$.

\subsection{Local weak convergence}
\label{sec:lowe}
Let $G^\bullet = (G, v_G)$ and $
H^\bullet = (H, v_H)$ be two connected, rooted, and locally finite graphs. We may consider the distance
\[
d(G^\bullet, H^\bullet) =  2^{-\sup \{k \in \ndN_0 \,\mid\, V_k(G^\bullet) \simeq V_k(H^\bullet) \}}\]
with  $V_k(G^\bullet) \simeq V_k(H^\bullet)$ denoting isomorphism of rooted graphs. This defines a premetric on the collection of all rooted locally finite connected graphs.  Two such graphs have distance zero, if and only if they are isomorphic. Hence this defines a metric on the collection $\mathbb{B}$ of all unlabelled, connected, rooted, locally finite graphs. The space $(\mathbb{B}, d_{\textsc{BS}})$ is known to be complete and separable, that is, a Polish space. 

A random rooted graph $\mG^\bullet \in \mathbb{B}$ is the the {\em local weak limit} of a sequence $\mG_n^\bullet = (\mG_n, v_n)$, $n\in \ndN$  of random elements of $\mathbb{B}$, if it is the weak limit with respect to this metric. That is, if \[\lim_{n \to \infty} \Ex{f(G_n^\bullet)} = \Ex{f(G^\bullet)}\] for any bounded continuous function $f: \mathbb{B} \to \ndR$. This is equivalent to stating
\[
\lim_{n \to \infty} \Pr{V_k(\mG_n^\bullet) \simeq G^\bullet} = \Pr{V_k(\mG^\bullet) \simeq G^\bullet}.
\]
for any rooted graph $G^\bullet$. If the conditional distribution of $v_n$ given the graph $\mG_n$ is uniform on the vertex set $V(\mG_n)$, then the limit $G^\bullet$ is often also called the \emph{Benjamini--Schramm limit} of the sequence $(\mG_n)_n$. This kind of convergence is often yields laws of large numbers for additive graph parameters.

\subsection{Gromov--Hausdorff convergence}
\label{sec:ghc}

Let $X^\bullet =  (X,d_X, x_0)$ and $Y^\bullet = (Y,d_Y,y_0)$ denote pointed compact metric spaces. A {\em correspondence} between $X^\bullet$ and $Y^\bullet$ is a subset $R \subset X \times Y$ containing $(x_0,y_0)$ such that for any $x \in X$ there is a point $y \in Y$ with $(x,y) \in R$, and conversely for any $y \in Y$ there is a point $x \in X$ with $(x,y) \in R$. The {\em distortion} of the correspondence is defined as the supremum \[\text{dis}(R) = \sup \{|d_X(x_1, x_2) - d_Y(y_1, y_2) \mid (x_1, y_1), (x_2, y_2) \in R \}.\] The {\em (pointed) Gromov--Hausdorff distance} between the pointed spaces $X^\bullet$ and $Y^\bullet$ may be defined by \[d_{\textsc{GH}}(X,Y) = \frac{1}{2} \inf_R \text{dis}(R)\] with the index $R$ ranging over all correspondences between $X^\bullet$ and $Y^\bullet$. The factor $1/2$ is only required in order to stay consistent with an alternative definition of the Gromov--Hausdorff distance via the Hausdorff distance of embeddings of $X^\bullet$ and $Y^\bullet$ into common metric spaces, see \cite[Prop.\ 3.6]{MR3025391} and \cite[Thm.\ 7.3.25]{MR1835418}. This distance satisfies the axioms of a premetric on the collection of all compact rooted metric spaces. Two such spaces have distance zero from each other, if and only if they are isometric. That is, if there is a distance preserving bijection between the two that also preserves the root vertices. Hence this yields a metric on the collection $\ndK^\bullet$ of isometry classes of pointed compact metric spaces. The space $(\ndK^\bullet, d_{\text{GH}})$ is known to be Polish (complete and separable), see \cite[Thm.\ 3.5]{MR3025391} and \cite[Thm. 7.3.30 and 7.4.15]{MR1835418}.

\subsection{Large critical Galton--Watson trees}
\label{sec:gwt}

In this section we let $\cT_n$ denote a Galton--Watson tree conditioned on having $n$ vertices, such that offspring distribution $\xi$ has average value $\Ex{\xi} =1$ and finite non-zero variance $\sigma^2$.

\subsubsection{Convergence toward the CRT}
The (Brownian) continuum random tree (CRT) is a random metric space constructed by Aldous in his pioneering papers \cite{MR1085326, MR1166406, MR1207226}. Its construction is as follows. To any continuous function $f: [0,1] \to [0, \infty[$  satisfying $f(0) = f(1) = 0$ we may associate a premetric $d$ on the unit interval $[0,1]$ given by
\[
d(u,v) = f(u) + f(v) - 2 \inf_{u \le s \le v} f(s)
\]
for $u \le v$. The corresponding quotient space $(\cT_f, d_{\cT_f}) = ([0,1]/ \mathord{\sim}, \bar{d})$, in which points with distance zero from each other are identified, is considered as rooted at the coset $\bar{0}$ of the point zero. This pointed metric space is an $\ndR$-tree, see  \cite{MR2147221, MR3025391} for the definition of $\ndR$-trees and further details. The CRT may be defined as the random pointed metric space $(\CRT, d_{\CRT}, \bar{0})$ corresponding to Brownian excursion $\me = (\me_t)_{0 \le t \le 1}$ of duration one.

The famous invariance principle,
\begin{align}
	(\cT_n, \frac{\sigma}{2} n^{-1/2} d_{\cT_n}, \emptyset) \convdis (\CRT, d_{\CRT}, \bar{0})
\end{align}
in the Gromov--Hausdorff sense, is due to Aldous \cite{MR1207226} and there exist various extensions, see for example Duquesne \cite{MR1964956}, Duquesne and Le Gall \cite{MR2147221}, Haas and Miermont \cite{MR3050512}. 

\subsubsection{Tail-bounds for the height and level width}
Addario-Berry, Devroye and Janson {\cite[Thm. 1.2]{MR3077536}} showed that there are constants $C,c>0$ such that for all $n$ and $h \ge 0$
\begin{align}
	\label{eq:gwttail}
	\Pr{\He(\cT_n) \ge h)} \le C \exp(-c h^2/n).
\end{align}	
A corresponding left-tail upper bound of the form
\begin{align}
	\label{eq:gwtlefttail}
\Pr{\He(\cT_n) \le h} \le C \exp(-c(n-2)/h^2)
\end{align}
for all $n$ and $h \ge 0$ is given in \cite[p. 6]{MR3077536}. The first moment of the number $L_k(\cT_n)$ of all vertices $v$ with height $h_{\cT_n}(v)= k$ admits a bound of the form
\begin{align}
	\label{eq:ineqstr1}
	\Ex{L_k(\cT_n)} \le C k \exp(-ck^2/n).
\end{align}
for all $n$ and $k \ge 1$. See \cite[Thm. 1.5]{MR3077536}. 

\section{Combinatorial species and weighted Boltzmann distributions}

\label{sec:combspec}
In order to study combinatorial objects up to symmetry, it is convenient to use the  language of combinatorial species developed by Joyal \cite{MR633783}.  It provides a clean and powerful framework in which complex combinatorial bijection may be stated using simple algebraic terms. In order to make the present work accessible to a large audience, we recall the notions and results required to state and prove our main results. The theory admits an elegant and concise description using the language of category theory, but we will avoid this terminology and assume no knowledge by the reader in this regard. The exposition of the combinatorial and algebraic aspects in the following subsections follows mainly Joyal \cite{MR633783} and Bergeron, Labelle and Leroux \cite{MR1629341}. The probabilistic aspects in the Boltzmann sampling framework is based on a  recent work by Bodirsky, Fusy, Kang and Vigerske \cite[Prop. 38]{MR2810913}.


\subsection{Weighted combinatorial species}
\label{sec:preop}
A {\em combinatorial species} $\cF$ is a functor or "rule" that produces for each finite set $U$ a finite set $\cF[U]$ of $\cF$-objects and for each bijection $\sigma: U \to V$ a bijective map $\cF[\sigma]: \cF[U] \to \cF[V]$ such that the following properties hold.
\begin{enumerate}[1)]
	\item $\cF$ preserves identity maps, that is for any finite set $U$ it holds that \[\cF[\text{id}_U] = \text{id}_{\cF[U]}.\]
	\item $\cF$ preserves composition of maps, that is,  for any bijections of finite sets $\sigma:U \to V$ and $\sigma': V \to W$ we require that \[\cF[\sigma' \sigma] = \cF[\sigma']\cF[\sigma].\]
\end{enumerate}
The idea behind this is that finite combinatorial objects are composed out of atoms, and relabelling this atoms yields structurally equivalent objects.


We say a combinatorial species $\cF$ maps any finite set $U$ of {\em labels} to the finite set $\cF[U]$ of {\em $\cF$-objects}  and any bijection $\sigma: U \to V$ to the {\em transport function} $\cF[\sigma]$. For any two $\cF$-objects $F_U \in \cF[U]$ and $F_V \in \cF[V]$ that satisfy $\cF[\sigma](F_U) = F_V$, we say $F_U$ and $F_V$ are {\em isomorphic} and $\sigma$ is an {\em isomorphism} between them. The object $F_U$ has {\em size} $|F_U| = |U|$ and $U$ is its \emph{underlying set}. An {\em unlabelled $\cF$-object} or {\em isomorphism type} is an isomorphism class of $\cF$-objects. That is, a maximal collection of pairwise isomorphic objects.  By abuse of notation, we treat unlabelled objects as if they were regular objects. 



An  {\em $\ndR_{\ge 0}$-weighted species} $\cF^\omega$ consists of a species $\cF$ and a {\em weighting} $\omega$ that produces for any finite set $U$ a map \[\omega_U: \cF[U] \to \ndR_{\ge 0}\] such that $\omega_U = \omega_V \circ \cF[\sigma]$ for any bijection $\sigma: U \to V$. Any object $F \in \cF[U]$ has {\em weight} $\omega_U(F)$. By abuse of notation we will often drop the index and write $\omega(F)$ instead of $\omega_U(F)$. Isomorphic structures have the same weight, hence we may define the weight of an unlabelled $\cF$-object  to be the weight of any representative. The inventory $|\tilde{\cF}[n]|_\omega$ is defined as the sum of weights of all unlabelled $\cF$-objects of size $n$. Any species may be considered as a weighted species by assigning weight $1$ to each structure, and in this case the inventory counts the number of $\cF$-objects. If we do not specify any weighting for a species, then we assume that it is equipped with this canonical weighting.

To any weighted species $\cF^\omega$ we may associate its  {\em ordinary generating series}  \[ \tilde{\cF}^\omega(z) = \sum_{n\ge 0} |\tilde{\cF}[n]|_\omega z^n = \sum_{F \text{ unlabelled $\cF$-object}} \omega(F) z^{|F|}.\] We may form the species $\Sym(\cF)$ of {\em $\cF$-symmetries} by letting $\Sym(\cF)[U]$ be the set of all pairs $(F, \sigma)$ with $F \in \cF[U]$ an $\cF$-structure and $\sigma$ an {\em automorphism} of $F$, that is, a bijection $\sigma: U \to U$ with $\cF[\sigma](F) = F$. For any bijection $\gamma: U \to V$ the corresponding transport function $\Sym(\cF)[\gamma]$ maps to a symmetry $(F, \sigma) \in \Sym(\cF)[U]$ to the symmetry $(\cF[\gamma](F), \gamma \sigma \gamma^{-1})$ in $\Sym(\cF)[V]$.

There is a canonical weighting on $\Sym(\cF)$ with weights in the power series ring $\ndR[[s_1, s_2, \ldots]]$. By abuse of notation, we also denote this weighting by $\omega$. It is given by \[\omega(F, \sigma) = \omega(F) s_1^{\sigma_1} s_2^{\sigma_2}\cdots\] with $\sigma_i$ denoting the number of cycles of length $i$ of the permutation $\sigma$. Here we count fixpoints as $1$-cycles. The {\em cycle index sum} $Z_{\cF^\omega}$ of $\cF^\omega$ may be defined is defined by
\[
Z_{\cF^\omega} = \sum_{n \ge 0} \sum_{(F,\sigma) \in \Sym(\cF)[n]} \omega(F,\sigma) / n! \in \ndR[[s_1, s_2, \ldots]].
\]
The generating series are related by  
\begin{align}
	\label{eq:relcyc}
 \tilde{\cF}^\omega(z) = Z_{\cF^\omega}(z, z^2, z^3, \ldots).
\end{align}
 See for example Chapter 2.3 in the book by Bergeron, Labelle and Leroux \cite{MR1629341}. 
The main point of considering symmetries is the following.
\begin{lemma}[Number of symmetries]
	\label{le:relcyc}
For each unlabelled $\cF$-object $s$ with size $n$ there are precisely $n!$ symmetries $(F, \sigma) \in \Sym(\cF)[n]$ such that the isomorphism type $t(F)$ of the $\cF$-object $F$ is equal to $s$.
\end{lemma}
 This follows from basic properties of  group operations, see for example Joyal \cite[Sec. 3]{MR633783}. In particular, if we draw a symmetry $(\mF, \sigma)$ from the set $\Sym(\cF)[n]$ at random with probability proportional to its weight, then 
\begin{align}
\Pr{t(\mF) = t} = \omega(t) / \sum_{s \in \tilde{\cF}[n]} \omega(s)
\end{align}
for any unlabelled $\cF$-object $t$ of size $n$.

We say that two species $\cF$ and $\cG$ are {\em isomorphic}, denoted by $\cF \simeq \cG$, if there is a family $(\alpha_U)_U$ of bijections $\alpha_U: \cF[U] \to \cG[U]$, with the index $U$ ranging over all finite sets, such that the following diagram commutes for any bijection $\sigma: U \to V$ of finite sets.
\[
\xymatrix{ \cF[U] \ar[d]^{\alpha_U} \ar[r]^{\cF[\sigma]} &\cF[V]\ar[d]^{\alpha_V}\\
	\cG[U] \ar[r]^{\cG[\sigma]} 		    &\cG[V]}
\]
The family $(\alpha_U)_U$ is then termed a {\em species isomorphism} from $\cF$ to $\cG$.

Two weighted species $\cF^{\omega}$ and $\cG^\nu$ are called isomorphic, if there exists a species isomorphism $(\alpha_U)_U$ from $\cF$ to $\cG$ that preserves the weights, that is, with $\nu(\alpha_U(F)) = \omega(F)$ for each finite set $U$ and $\cF$-object $F \in \cF[U]$. In this case the cycle index sums and hence also the other two generating series of $\cF^\omega$ and $\cG^\nu$ coincide.

There are some natural examples of species that we are going to encounter frequently. The species $\Set$ with $\Set[U] = \{U\}$ has only one structure of each size and its cycle index sum is given by
\[
	Z_\Set(s_1, s_2, \ldots) = \exp(\sum_{i \ge 1} s_i / i).
\]
The species $\Seq$ of linear orders assigns to each finite set $U$ the set $\Seq[U]$ of tuples $(u_1, \ldots, u_t)$ of distinct elements with  $U=\{u_1, \ldots, u_t\}$. Its cycle index sum is given by
\[
	Z_\Seq(s_1, s_2, \ldots) = 1/(1 -s_1).
\]
The species $\cX$ is given by $\cX[U] = \emptyset$ if $|U| \ne 1$ and $\cX[U] = \{U\}$ if $U$ is a singleton. The species $0$ is given by $0[U] = \emptyset$ for all $U$, and the species $1$ by $1[\emptyset] = \{\emptyset\}$ and $1[U] = \emptyset$ for all finite non-empty sets $U$.

\subsection{Operations on species}
\label{sec:opspe}
Species may be combined in several ways to form new species. 


\subsubsection{Products}
The {\em product} $\cF \cdot \cG$ of two species $\cF$ and $\cG$ is the species given by \[(\cF \cdot \cG)[U] = \bigsqcup_{(U_1, U_2)} \cF[U_1] \times \cG[U_2]\] with the index ranging over all ordered 2-partitions of $U$, that is, ordered pairs of (possibly empty) disjoint sets whose union equals $U$. The transport of the product along a bijection is defined componentwise. Given weightings $\omega$ on $\cF$ and $\nu$ on $\cG$, there is a canonical weighting on the product given by \[\mu(F,G) = \omega(F)\nu(G).\] This defines the product of weighted species
\[(\cF \cdot \cG)^\mu = \cF^\omega \cdot \cG^\nu.\]
The corresponding cycle index sum  satisfies \[Z_{(\cF \cdot \cG)^\mu} = Z_{\cF^\omega} Z_{\cG^\nu}.\] 
We also define the powers of a species by
\[
	(\cF^\omega)^i = \cF^\omega \cdot \ldots \cdot \cF^\omega
\]
with $i$ factors in total, and define $(\cF^\omega)^0 = 0$ to be the empty species having no objects at all.

\subsubsection{Sums} Let $(\cF_i)_{i \in I}$ be a family of species such that for any finite set $U$ only finitely many indices $i$ with $\cF_i[U] \ne \emptyset$ exist. Then the {\em sum} $\sum_{i \in I} \cF_i$ is a species defined by \[(\sum_{i \in I} \cF_i)[U] = \bigsqcup_{i \in I} \cF_i[U].\] Given weightings $\omega_i$ on $\cF_i$, there is a canonical weighting $\mu$ on the sum given by \[\mu(F) = \omega_i(F)\] for any $i$ and $F \in \cF_i[U]$. This defines the sum of the weighted species
\[
(\sum_{i \in I} \cF_i)^\mu = \sum_{i \in I} \cF_i^{\omega_i}.
\]
The corresponding cycle index sum is given by \[Z_{\sum_i \cF_i^{\omega_i}} = \sum_i Z_{\cF_i^{\omega_i}}.\]

\subsubsection{Derived species}
Given a species $\cF$, the corresponding {\em derived} species $\cF'$ is given by
\[
\cF'[U] = \cF[U \cup \{*_U\}]
\]
with $*_U$ referring to an arbitrary fixed element not contained in the set~$U$. (For example, we could set $*_U = \{U\}$.) Any weighting $\omega$ on $\cF$ may also be viewed as a weighting on $\cF'$, by letting the weight of a derived object $F \in \cF'[U]$ be given by $\omega_{U \cup \{*_U\}}(F)$. The transport along a bijection $\sigma: U \to V$ is done by applying the transport $\cF[\sigma']$ of the bijection $\sigma': U \cup \{*_U\} \to V \cup \{*_V\}$ with $\sigma'|U = \sigma$. The cycle-index sum of the weighted derived species $(\cF^\omega)'$ is satisfies \[Z_{(\cF^\omega)'} = \frac{\partial}{\partial s_1} Z_{\cF^\omega}.\] 
By abuse of notation, we are often going to drop the index and just refer to the $*$-atom.

\subsubsection{Pointing} For any species $\cF$ we may form the {\em pointed} species $\cF^\bullet$. It is given by the product of species \[\cF^\bullet = \cX \cdot \cF'\] with $\cX$ denoting the species consisting of single object of size $1$. In other words, an $\cF^\bullet$-object is pair $(m, v)$ of an $\cF$-object $m$ and a distinguished label $v$ which we call the root of the object. Any weighting $\omega$ on $\cF$ may also be considered as a weighting on $\cF^\bullet$, by letting the weight of $(m,v)$ be given by $\omega(m)$. This choice of weighting is consistent with the natural weighting given by the product and derivation operation $\cX \cdot \cF'$, if we assign weight $1$ to the unique object of $\cX$. The corresponding cycle index sum is consequently given by
\[
Z_{(\cF^\bullet)^\omega} = s_1 \frac{\partial}{\partial s_1} Z_{\cF^\omega}.
\]
\subsubsection{Substitution} Given species $\cF$ and $\cG$ with  $\cG[\emptyset] = \emptyset$, we may form the composition $\cF \circ \cG$ as the species with object sets \[(\cF \circ \cG)[U] = \bigcup_\pi \left ( \{\pi \} \times \cF[\pi] \times \prod_{Q \in \pi} \cG[Q] \right ),\] with the index $\pi$ ranging over all unordered partitions of the set $U$. Here the transport $(\cF \circ \cG)[\sigma]$ along a bijection $\sigma: U \to V$ is done as follows. For any object $(\pi, F, (G_Q)_{Q \in \pi})$ in $(\cF \circ \cG)[U]$ define the partition
\[
\hat{\pi} = \{\sigma(Q) \mid Q \in \pi \},
\]
and let
\[
\hat{\sigma}: \pi \to \hat{\pi}
\]
denote the induced bijection between the partitions. Then set
\[
(\cF \circ \cG)[\sigma](\pi, F, (G_Q)_{Q \in \pi}) = (\hat{\pi}, \cF[\hat{\sigma}](F), (\cG[\sigma|_Q](g_Q))_{ \sigma(Q) \in \hat{\pi}}).
\]
That is, the transport along the induced bijection of partitions gets applied to the $\cF$-object and the transports along the restrictions $\sigma|_Q$, $Q \in \pi$ get applied to the $\cG$-objects. Often, we are going to write $\cF(\cG)$ instead of $\cF \circ \cG$. Given a weighting $\omega$ on $\cF$ and a weighting $\nu$ on $\cG$, there is a canonical weighting $\mu$ on the composition given by \[\mu(\pi, F, (G_Q)_{Q \in \pi}) = \omega(F) \prod_{Q \in \pi} \nu(Q).\] This defines the composition of weighted species
\[
(\cF \circ \cG)^\mu = \cF^\omega \circ \cG^\nu.
\]
The corresponding cycle index sum is given by
\begin{align}
\label{eq:cyccomp}
Z_{(\cF \circ \cG)^\mu}(s_1, s_2, \ldots) = Z_{\cF^\omega}(Z_{\cG^\nu}(s_1, s_2, \ldots), Z_{\cG^{\nu^2}}(s_2, s_4, \ldots), Z_{\cG^{\nu^3}}(s_3, s_6, \ldots), \ldots).
\end{align}
Here $\nu^i$ denotes the weighting with $(\nu^i)(G) = \nu(G)^i$ for all $\cG$-structures $G$.
\subsubsection{Restriction}
For any subset $\Omega \subset \ndN_0$ we may restrict a weighted species $\cF^\omega$ to objects whose size lies in $\Omega$ and denote the result by $\cF^\omega_{\Omega}$. 

\subsubsection{Relations between the different operations}
The interplay of the operations discussed in this section is described by a variety of natural isomorphisms. The two most important are the \emph{product rule} and the \emph{chain rule}, which we are going to use often.
\begin{proposition}[\cite{MR633783}]
	\label{pro:chainprod}
	Let $\cF^\omega$ and $\cG^\nu$ be weighted species.
	\begin{enumerate}
		\item There is a canonical choice for an isomorphism
	\[
	(\cF^\omega \cdot \cG^\nu)' \simeq (\cF^\omega)'\cdot \cG^\nu + \cF^\omega \cdot(\cG^\nu)'.
	\]
		\item Suppose that $\cG[\emptyset]=\emptyset$. Then there is also a canonical isomorphism
	\[
		(\cF^\omega \circ \cG^\nu)' \simeq ((\cF^\omega)' \circ \cG^\nu) \cdot (\cG^\nu)'.
	\]
	\end{enumerate}
\end{proposition}
We may easily verify the product rule, as the $*$-label in $(\cF^\omega \cdot \cG^\nu)'$ may either belong the $\cF$-structure, accounting for the summand $(\cF^\omega)' \cdot \cG^\nu$, or to the $\cG$-structure, accounting for the second summand. The chain rule also has an intuitive explanation. The idea is that the partition class or $\cG$-structure containing the $*$-label in an $(\cF^\omega \circ \cG^\nu)'$-structure distinguishes an atom of the $\cF$-structure. Hence the $(\cF^\omega \circ \cG^\nu)'$-structure consists of a $(\cF^\omega)' \circ \cG^\nu)$ composite structure, where all atoms of the $\cF$-structure receive a regular $\cG$-structure, except for a marked $*$-atom, to which we assign a derived $\cG$-structure, which accounts for the extra factor $(\cG^\nu)'$.


\subsection{Symmetries of the substitution operation}
\label{sec:symmetry}
We are going to need detailed information on the structure of the symmetries of the composition $\cF \circ \cG$. The exposition of this section follows Joyal \cite[Sec. 3.2]{MR633783} and Bergeron, Labelle and Leroux \cite[Sec. 4.3]{MR1629341}. We are going to discuss the following result.

\begin{lemma}[Parametrization of the symmetries of the substitution]
	\label{le:symconstr}
	Up to isomorphism of symmetries, any $\cF \circ \cG$-symmetry may be constructed as described below from an $\cF$-symmetry $(F, \sigma)$ together with a family of $\cG$-symmetries $(G_\tau, \sigma_\tau)_\tau$ with the index $\tau$ ranging over all cycles of the permutation $\sigma$.
\end{lemma}

There is much more to this result, as it lies at the heart of the proof of Equation~\eqref{eq:cyccomp}. We refer the inclined reader to the mentioned literature for details. For the purpose of the present paper, it is sufficient to understand how the symmetry and in particular its cycles get assembled. We are going to use this later in order to define random unlabelled structures based on a tree-like decomposition of symmetries.

The method of construction referred to in Lemma~\ref{le:symconstr} is a bit involved, hence let us first recall what an $\cF \circ \cG$-symmetry is by definition. Let $U$ be a finite set. Any element of $\Sym(\cF \circ \cG)[U]$ consists of the following objects: a partition $\pi$ of the set $U$, an $\cF$-structure $F \in \cF[\pi]$, a family of $\cG$-structures $(G_Q)_{Q \in \pi}$ with $G_Q \in \cG[Q]$ and a permutation $\sigma: U \to U$. The permutation $\sigma$ is required to permute the partition classes and induce an automorphism \[\bar{\sigma}: \pi \to \pi, \,\, Q \mapsto \sigma(Q)\] of the $\cF$-object $F$. Moreover, for any partition class $Q \in \pi$  the restriction $\sigma|_Q: Q \to \sigma(Q)$ is required to be an isomorphism from $G_Q$ to $G_{\sigma(Q)}$.

Note that for any cycle $\tau = (Q_1, \ldots, Q_\ell)$ of $\bar{\sigma}$, it follows that  $\sigma^{\ell}|_{Q_1}: Q_1 \to Q_1$ is an automorphism of $G_{Q_1}$. Hence $(G_\tau, \sigma_\tau) := (G_{Q_1}, \sigma^{\ell}|_{Q_1})$ is a $\cG$-symmetry. The symmetry $S$  together with the bijections $\gamma_i: = \sigma|_{Q_i}$  for $1 \le i \le \ell-1$ already contain all information about the $\cG$-objects $G_{Q_1}, \ldots, G_{Q_\ell}$ and the restriction $\sigma|_{Q_1 \cup \ldots \cup Q_\ell}$. Indeed, it holds that $G_{Q_{i+1}} = \cG[\gamma_i \cdots \gamma_1](G_\tau)$ for all $1 \le i \le \ell -1$, hence we may reconstruct the $\cG$-objects. The bijections $\gamma_i$ contain all information about $\sigma_{Q_1 \cup \ldots \cup Q_{\ell -1}}$, and it holds that $\sigma|_{Q_\ell} = \gamma_{\ell -1}^{-1} \cdots \gamma_{1}^{-1} \sigma_\tau$. In particular, any $k$-cycle $(a_1, \ldots, a_k)$ of the permutation $\sigma_\tau$ corresponds to the $k\ell$-cycle \[(a_1, \sigma(a_1), \ldots, \sigma^{\ell-1}(a_1), a_2, \sigma(a_2), \ldots, \sigma^{\ell-1}(a_2), \ldots, a_k, \sigma(a_k), \ldots, \sigma^{\ell-1}(a_k))\] of the permutation $\sigma|_{Q_1 \cup \ldots \cup Q_\ell}$. 

\begin{figure}[t]
	\centering
	\begin{minipage}{1.0\textwidth}
		\centering
		\includegraphics[width=0.7\textwidth]{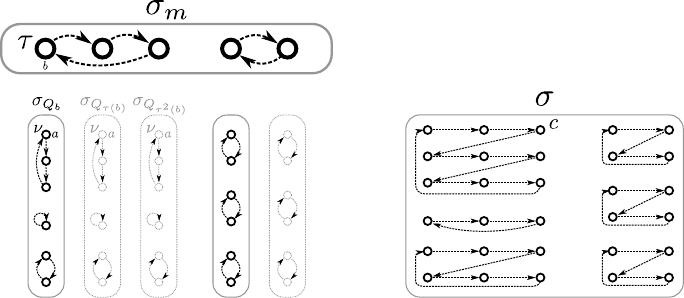}
		\caption{Composition of cycles.}
		\label{fi:symmetries}
	\end{minipage}
\end{figure}

This implies that any $\cF \circ \cG$-symmetry is isomorphic to a symmetry $( (\pi, F, (G_Q)_{Q \in \pi}), \sigma)$ constructed in the following way, which is illustrated in Figure~\ref{fi:symmetries}. Start with choosing an $\cF$-symmetry $(m, \sigma_m)$. For any cycle $\tau$ of the permutation $\sigma_m$ choose a $\cG$-symmetry $(G_\tau, \sigma_\tau)$ and let $Q_\tau$ denote its set of labels. For every atom $e$ of the cycle $\tau$ set $Q_e := Q_\tau \times \{e\}$ and $(G_{Q_e}, \sigma_{Q_e}) := \Sym(\cG)[f_e](G_\tau, \sigma_\tau)$ with $f_e: Q_\tau \to Q_e$ the canonical bijection. For any label $e$ of the $\cF$-structure $m$ set $f(e) := Q_e$ and let $\pi$ denote the set of all sets $Q_e$. Thus $F := \cF[f](m)$ is an $\cF$-structure with label set $\pi$ and $C := (\pi, F, (G_{Q})_{Q \in \pi})$ is an $\cF \circ \cG$-structure. Let $\tau$ be a cycle of $\sigma_m$ and $\nu$ a cycle of $\sigma_\tau$. Choose an atom $b$ of $\tau$ and an atom $a$ of $\nu$. Let $\ell$ denote the length of $\tau$ and $k$ the length of $\nu$. Form the composed cycle $c$ by
\[
((a,b), \ldots, (a, \tau^{\ell-1}(b)), (\nu(a),b), \ldots, (\nu(a), \tau^{\ell-1}(b)), \ldots, (\nu^{k-1}(a), b), \ldots, (\nu^{k-1}(a), \tau^{\ell-1}(b))).
\]
Then the product $\sigma$ of all composed cycles is an automorphism of the $\cF \circ \cG$-structure $C$. The composed cycles are pairwise disjoint, hence it does not matter in which order we take the product. Note that $\sigma$ does not depend on the choice of the $a$'s but  different choices of the $b$'s result in a different automorphism $\sigma$. More precisely, if for a given cycle $\tau$ of $\sigma_m$ we choose $\tau(b)$ instead of $b$, then the resulting automorphism is given by the conjugation $(\text{id}, \tau) \sigma (\text{id}, \tau)^{-1}$ instead of $\sigma$. But $(\text{id}, \tau)$ is an automorphism of the $\cF \circ \cG$-structure $C$, hence the resulting symmetry $(C, (\text{id}, \tau) \sigma (\text{id}, \tau)^{-1})$ is isomorphic to $(C, \sigma)$. This implies that the isomorphism type of $(C, \sigma)$ does not depend on the choices of the $a$'s and $b$'s. Fixing any canonical way of making these choices yields the construction of Lemma~\ref{le:symconstr}.

\subsection{Weighted Boltzmann distributions and samplers}
\label{sec:webo}
Boltzmann distributions crop up in the study of local limit of random discrete structures and in the limit of convergent unlabelled Gibbs partitions. A {\em Boltzmann sampler} is a possibly recursive procedure involving random choices that generates a structure according a Boltzmann distribution. For example, a subcritical or critical Galton--Watson tree may be considered as a Boltzmann distributed plane tree. A recursive sampler in this setting is a procedure that draws the offspring of the root and then calls itself for each offspring vertex. 

\subsubsection{Boltzmann distributions}
\label{sec:bodistr}
Let $\cF^\omega$ be a weighted species. For any $y\ge 0$ satisfying $0 < \tilde{\cF}^\omega(y) < \infty$, the corresponding {\em Boltzmann distribution for unlabelled $\cF$-objects} is given by
\begin{align}
\mathbb{P}_{\tilde{\cF}^\omega, y}(\tilde{F}) = \tilde{\cF}^\omega(y)^{-1} \omega(\tilde{F}) y^{|\tilde{F}|}, \quad \tilde{F} \text{ an unlabelled $\cF$-object}.
\end{align}
Given a sequence  $\mathbf{y} = (y_j)_{j \in \ndN}$ of non-negative parameters $y_j$ satisfying $0 < Z_{\cF^\omega}(\mathbf{y}) < \infty$, the corresponding {\em P\'olya-Boltzmann distribution} is given by
\begin{align}
\label{eq:defwebopo}
\mathbb{P}_{\Sym(\cF)^\omega, \mathbf{y}}(F, \sigma) = Z_{\cF^\omega}(\mathbf{y})^{-1} \omega(F) \frac{y_1^{\sigma_1} y_2^{\sigma_2} \cdots }{|(F, \sigma)|!}, \quad (F, \sigma) \in  \sum_{m \ge 0} \Sym(\cF)[m],
\end{align}
with $\sigma_i$ denoting the number of $i$-cycles of a permutation $\sigma$.
By Lemma~\ref{le:relcyc} and Equation~\eqref{eq:relcyc}, the Boltzmann distribution for unlabelled objects may be considered as the marginal distribution of the $\cF$-object in special cases of the P\'olya-Boltzmann distribution. That is, the $\cF$-object of a  $\mathbb{P}_{\Sym(\cF)^\omega, (y, y^2, y^3, \ldots)}$-distributed $\cF$-symmetry follows a $\mathbb{P}_{\tilde{\cF}^\omega, y}$-distribution.

\subsubsection{Boltzmann samplers}
\label{sec:WeBoSa}
The following lemma allows us to construct Boltzmann distributed random variables in the unlabelled setting for the sum, product and composition of species. The results of this section have been established by Bodirsky, Fusy, Kang and Vigerske \cite[Prop. 38]{MR2810913} for species without weights, and the generalization to the weighted setting is straight-forward.

\begin{lemma}[Weighted P\'olya-Boltzmann samplers]
	\label{le:pobole}
	\hspace{1pt}
		\begin{enumerate}
			\item Let $\cF^\omega$ and $\cG^\nu$ be weighted species, and let $X$ and $Y$ be independent random variables with distributions $\mathcal{L}(X) = \mathbb{P}_{\Sym(\cF)^\omega, \mathbf{y}}$ and $\mathcal{L}(Y) = \mathbb{P}_{\Sym(\cG)^\omega, \mathbf{y}}$. Then $(X,Y)$ may be interpreted as an $(\cF \cdot \cG)$-symmetry over the set $[|X|] \sqcup [|Y|]$. If $\alpha$ denotes a uniformly at random drawn bijection from this set to $[|X| + |Y|]$, then
			\[
			\cL \left ( (\cF \cdot \cG)[\alpha](X,Y) \right ) = \mathbb{P}_{\Sym(\cF \cdot \cG)^\omega, \mathbf{y}}.
			\]
			\item Let $(\cF_i^{\omega_i})_{i \in I}$ be a family of weighted species, and $(X_i)_{i \in I}$ a family of independent random variables with distributions $\mathbb{P}_{\Sym(\cF_i)^{\omega_i}, \mathbf{y}}$ such that  $\sum_i Z_{\cF_i^{\omega_i}}(\mathbf{y}) < \infty$. If $K \in I$ gets drawn at random with probability proportional to $Z_{\cF_K^{\omega_K}}(\mathbf{y})$, that is \[\Pr{K=k} = Z_{\cF_k^{\omega_k}}(\mathbf{y}) / \sum_i Z_{\cF_i^{\omega_i}}(\mathbf{y}),\] then 
			\[
			\cL(X_K) = \mathbb{P}_{\Sym(\sum_i \cF_i^{\omega_i}), \mathbf{y}}.
			\] 
			\item Let $\cF^\omega$ and $\cG^\nu$ be species such that $\cG^\nu[\emptyset] = \emptyset$ and let $\mathbf{y} = (y_j)_{j \in \ndN}$ a family of non-negative parameters with $0< Z_{\cG^\nu}(\mathbf{y}) < \infty$ and $0 < Z_{\cF^\omega \circ \cG^\nu}(\mathbf{y}) < \infty$. For each $k$ set $\mathbf{y}^k = (y_k, y_{2k}, y_{3k}, \ldots)$. Let $X=(F,\sigma)$ be a $\mathbb{P}_{\Sym(\cF^\omega), (Z_{\cG^{\nu^1}}(\mathbf{y}), Z_{\cG^{\nu^2}}(\mathbf{y^2}), \ldots)}$-distributed random $\cF$-symmetry and let $(Y_{i,k})_{i,k \in \ndN}$ an independent family (that is also independent of $X$) of random $\cG$-symmetries such that $Y_{i,k}$ follows a $\mathbb{P}_{\Sym(\cG^\nu), \mathbf{y}^k}$-distribution for all $k,i$. We may canonically assign to each cycle $\tau$ of the random permutation $\sigma$ a unique element $Y_\tau$ of the list $(Y_{i, |\tau|})_{i \ge 1}$. For example, we could do this by ordering for each $\ell$ the cycles of $\sigma$ having length $\ell$ according to their smallest atom, and assign $Y_{i,\ell}$ to the $i$th cycle in the ordering. Then $(X, (Y_\tau)_\tau)$ corresponds according to Lemma~\ref{le:symconstr} to an $\cF \circ \cG$-symmetry over some set $M$. Draw a bijection $\alpha: M \to [|M|]$ uniformly at random. Then
			\[
			\cL( (\cF \circ \cG)[\alpha](X, (Y_\tau)_\tau)) = \mathbb{P}_{\Sym(\cF^\omega \circ \cG^\nu), \mathbf{y}}.
			\]
		\end{enumerate}
\end{lemma}

\subsection{Combinatorial specifications and recursive Boltzmann samplers}
\label{sec:cospec}
\subsubsection{Motivation}
A recursive procedure is a list of instructions that are to be followed step by step and may contain references to the procedure itself. For example, a Galton--Watson tree may be described by the procedure that starts with a root vertex and attaches a random number of independent calls of itself. 

Often one encounter species admitting a recursive isomorphism such as $\cF^\omega \simeq \cX + (\cF^\omega)^2$. If this decomposition satisfies a certain property {\em (R)}, then for any admissible parameter we may apply the rules from Section~\ref{sec:WeBoSa} for the sum, product and composition of species in order to construct a recursive (P\'olya-)Boltzmann sampler for $\cF$. That is, a recursive procedure that terminates almost surely and samples objects according to the (P\'olya)-Boltzmann distribution. 

For the given example $\cF$, such a recursive Boltzmann sampler would first, by the sum rule, make a coin flip in order to determine whether it terminates with a single vertex (a Boltzmann sampler for $\cX^\kappa$), or creates, by the product rule, an ordered pair of independent calls of itself. In other words, its a Galton--Watson tree. As property {\em (R)} guarantees that this process terminates almost surely, we also know that this Galton--Watson tree must be critical or subcritical. It is clear that not any recursive decomposition can have this desired property. For example, the species $1$ which consists of a single object with size zero admits an isomorphism $1 \simeq 1 \cdot 1$, but applying the product rule yields a recursive procedure which never terminates.

Precisely stating property {\em (R)} requires us to introduce the complex concepts of weighted multi-sort species and samplers, as well as related operations such as combinatorial composition and partial derivatives in this context. 

\subsubsection{Combinatorial specifications}
A {\em  $2$-sort species} $\cH$ is a functor that maps any pair $U= (U_1, U_2)$ of finite sets to a finite set $\cH[U] = \cH[U_1, U_2]$ and any pair $\sigma = (\sigma_1, \sigma_2)$ of bijections $\sigma_i: U_i \to V_i$ to a bijection $\cH[\sigma]: \cH[U] \to \cH[V]$ in such a way, that identity maps and composition of maps are preserved. 
A weighted $2$-sort species $\cH^\omega$ additionally carries a weighting $\omega$ given by family of maps
\[
	\omega_{U_1, U_2}: \cH[U_1, U_2] \to \ndR_{\ge 0}
\]
for all pairs $(U_1, U_2)$. The weighting is required to be assign the same weight to isomorphic structures. That is, the diagram
\[
\xymatrix{ \cH[U_1,U_2] \ar[d]_{\cH[\sigma_1,\sigma_2]} \ar[r]^-{\omega_{U_1, U_2}} &\ndR_{\ge 0} \\
	\cH[V_1,V_2] \ar[ur]_{\omega_{V_1,V_2}}	    &}
\]
must commute for all bijections $\sigma_1: U_1 \to V_1$ and $\sigma_2: U_2 \to V_2$.

The operations of sum, product and composition extend naturally to the multi-sort-context Let $\cH$ and $\cK$ be 2-sort species and $U = (U_1,U_2)$ a pair of finite sets. The {\em sum} is defined by
\[
(\cH + \cK)[U] = \cH[U] \sqcup \cK[U].
\]
We write $U = V + W$ if $U_i = V_i \cup W_i$ and $V_i \cap W_i = \emptyset$ for all $i$. The {\em product} is defined by
\[
(\cH \cdot \cK)[U] = \bigsqcup_{V + W = U} \cH[V] \times \cK[W].
\]
The {\em partial derivatives} are given by
\[
\partial_1 \cH[U] = H[U_1 \cup \{*_{U_1}\}, U_2] \quad \text{and} \quad \partial_2 \cH[U] = H[U_1, U_2 \cup \{*_{U_2}\}].
\] 
In order state Joyal's implicit species theorem we also require the substitution operation for multi-sort species; this will allow us to  define species ``recursively'' up to (canonical) isomorphism. Let $\cF_1$ and $\cF_2$ be (1-sort) species and $M$ a finite set. A structure of the {\em composition} $\cH(\cF_1, \cF_2)$ over the set $M$ is a quadruple $(\pi, \chi, \alpha, \beta)$ such that: 
\begin{enumerate}
	\item $\pi$ is partition of the set $M$.
	\item $\chi: \pi \to \{1, 2\}$ is a function assigning to each class a sort.
	\item $\alpha$ a function that assigns to each class $Q \in \pi$ a $\cF_{\chi(Q)}$ object $\alpha(Q) \in \cF_{\chi(Q)}[Q]$.
	\item $\beta$ a $\cH$-structure over the pair $(\chi^{-1}(1), \chi^{-1}(2))$.
\end{enumerate}
This construction is {\em functorial}: any pair of isomorphisms $\alpha_1$, $\alpha_2$ with $\alpha_i: \cF_i \, \simeq \, \cG_i$ {\em induces} an isomorphism $\cH[\alpha_1, \alpha_2]: \cH(\cF_1, \cF_2) \, \simeq \, \cH(\cG_1, \cG_2)$.

Let $\cH$ be a $2$-sort species and recall that $\cX$ denotes the species with a unique object of size one. A solution of the system $\cY = \cH(\cX, \cY)$ is pair $(\cA, \alpha)$ of a species $\cA$ with $\cA[0] = 0$ and an isomorphism $\alpha: \cA  \,\simeq\, \cH(\cX, \cA)$. An isomorphism of two solutions $(\cA, \alpha)$ and $(\cB, \beta)$ is an isomorphism of species $u:\cA \,\simeq\, \cB$ such that the following diagram commutes:
\[
\xymatrix{ \cA \ar[d]^{u} \ar[r]^-{\alpha} &\cH(\cX,\cA)\ar[d]^{\cH(\text{id},u)}\\
	\cB \ar[r]^-{\beta} 		    &\cH(\cX,\cB)}
\]
We may now state Joyal's implicit species theorem. Recall that $0$ denotes the empty species with $0[U] = \emptyset$ for all finite sets $U$.
\begin{theorem}[\cite{MR633783}, Théorème 6]
	\label{te:implicitspecies}
	Let $\cH$ be a 2-sort species satisfying $\cH(0,0)=0$. If $(\partial_2 \cH)(0,0)=0$, then the system $\cY = \cH(\cX, \cY)$ has up to isomorphism only one solution. Moreover, between any two given solutions there is exactly one isomorphism.
\end{theorem}
\noindent We say that an isomorphism $\cF \simeq \cH(\cX, \cF)$ is a {\em combinatorial specification} for the species $\cF$ if the 2-sort species $\cH$ satisfies the requirements of Theorem~\ref{te:implicitspecies}.

\subsubsection{Recursive Boltzmann samplers}
\label{sec:recur}
Given a combinatorial specification $\cF \simeq \cH(\cX, \cF)$ we may apply the rules of Lemma~\ref{le:pobole} recursively to construct a recursive Boltzmann sampler that is guaranteed to terminated almost surely.  A justification of this fact is given by Bodirsky, Fusy, Kang and Vigerske~\cite[Thm. 40]{MR2810913} for species without weights, and the generalization to the weighted setting is straight-forward. Let us demonstrate this with an example.  Let $\cF$ is the species of binary unordered rooted trees, where each tree receives weight $1$. Any such tree is either a single root vertex, or root vertex with two binary trees dangling from it. This yields an isomorphism
$
	\cF \simeq \cX  + \cX \cdot \cF^2,
$
where we set $\cF^2 = \cF \cdot \cF$, and the two summands correspond to the two described cases. This may be reformulated by
$
	\cF \simeq \cX \cdot \cG(\cF) = \cH(\cX, \cF)
$
with $\cG = 1 + \cX^2$. It holds that $(\partial_2 \cH)(0,0) = 0 \cdot \cG'(0) = 0$, hence for each $y_1, y_2, \ldots \ge 0$ with $0 < Z_{\cF}(y_1,y_2, \ldots) < \infty$ we may apply the rules of Lemma~\ref{le:pobole} to obtain a recursive procedure that terminates almost surely and whose output follows a $\mathbb{P}_{Z_\cF, (y_i)_i}$-distribution. Briefly described, the procedure starts with a root-vertex and then draws a $\cG$-symmetry according to a Boltzmann distribution. There are three different outcomes, either the symmetry has size $0$, in which the sampler stops, or it has size $2$, in which case the sampler calls itself recursively to determine the two subtrees dangling from the root-vertex, with the parameters depending on whether the $\cG$-symmetry consists of two fix-points or a single $2$-cycle.

\section{Unlabelled Gibbs partitions and subexponential sequences}
\label{sec:gibbs}
The term Gibbs partitions was coined by Pitman~\cite{MR2245368} in his comprehensive survey on combinatorial stochastic processes. It describes a model of random partitions of sets, where the collection of classes as well as each individual partition class are endowed with a weighted structure.  

Many structures such as classes of graphs may also be viewed up to symmetry. The symmetric group acts in a canonical way on the collection of composite structures over a fixed set, and its orbits may be identified with the unlabelled objects. Sampling such an isomorphism class with probability proportional to its weight is the natural unlabelled version of the Gibbs partition model.

Let $\cF^\omega$ and $\cG^\nu$ be weighted species such that $\cG[\emptyset] = \emptyset$, and such that the ordinary generating series $\widetilde{\cF^\omega \circ \cG^\nu}(z)$ is not a polynomial.  An \emph{unlabelled Gibbs partitions} is a random composite structure \[\mS_n = (\pi_n, \mF_n, (\mG_Q)_{Q \in \pi_n})
\] sampled from the set of all unlabelled $\cF \circ \cG$-objects with probability proportional to its weight.   

We are going to study the asymptotic behaviour of the remainder $\mR_n$ obtained by deleting "the" largest component from $\mS_n$.  More specifically, we make a uniform choice of a component $Q_0 \in \pi_n$ having maximal size, and let $\mF_n'$ denote the $\cF'$-object obtained from the $\cF$-object $\mF_n$ by relabelling the $Q_0$ atom of $\mF_n$ to a $*$-placeholder.

This yields an unlabelled $\cF' \circ \cG$-object
\[
\mR_n := (\pi_n \setminus \{Q_0\}, \mF_n', (\mG_Q)_{Q \in \pi_n \setminus \{Q_0\}}) \in \mathscr{U}(\cF' \circ \cG).
\]
In the so called convergent case, the remainder $\mR_n$ is stochastically bounded and even converges in total variation toward a limit object.

\begin{theorem}[{\cite[Thm. 3.1]{2016arXiv161001401S}}]
	\label{te:gibbs}
	Suppose that the  ordinary generating series  $\tilde{\cG}^\nu(z)$ has positive radius of convergence $\rho$, and that the coefficients $g_i = [z^i] \tilde{\cG}^\nu(z)$ satisfy
	\begin{align*}
	\frac{g_n}{g_{n+d}} \sim \rho, \qquad \frac{1}{g_n}\sum_{i+j=n}g_ig_j \sim 2 \tilde{\cG}^\nu(\rho) < \infty.
	\end{align*}
	Suppose further that
	\begin{align*}
	Z_{\cF^\omega}( \tilde{\cG}^\nu(\rho) + \epsilon, \tilde{\cG}^{\nu^2}((\rho + \epsilon)^2), \tilde{\cG}^{\nu^3}((\rho + \epsilon)^3), \ldots) < \infty
	\end{align*}
	for some $\epsilon >0$. Then 
	\[
		[z^n] \widetilde{\cF^\omega \circ \cG^\nu}(z) \sim \widetilde{(\cF')^\omega \circ \cG^\nu}(\rho) [z^n] \cG^\nu(z).
	\]
	and 
	\begin{align*}
	d_{\textsc{TV}}(\mR_n, \mR) \to 0,  \qquad n\to \infty
	\end{align*}
	with $\mR$ denoting a random unlabelled $\cF' \circ \cG$-element  that follows a $\mathbb{P}_{ \widetilde{(\cF')^\omega \circ \cG^\nu}, \rho}$-Boltzmann distribution.	
\end{theorem}

\section{Probabilistic and combinatorial tools}
\label{sec:tools}



\subsection{The lattice case of the multivariate central local limit theorem}
We will make frequent use of the classical local limit theorem for random walks.



\begin{lemma}[Central local limit theorem for lattice distributions]
	\label{le:llt1dim}
	\label{le:llt2dim}
	Let $\mathbf{Y}$ be a random vector in $\ndR^d$ whose support is contained in the lattice $\mathbf{a} + \mathbf{D} \ndZ^d$ with $\mathbf{D} \in \textnormal{GL}_d(\ndR)$, $\mathbf{a} \in \ndR^d$, and in no proper sublattice. Suppose that $\mathbf{Y}$ has a finite non-zero covariance matrix $\mathbf{\Sigma}$, and let $(\mathbf{Y}_i)_{i \ge 1}$ be a family of independent copies of $\mathbf{Y}$.  For all $n$ and $\mathbf{y}$ set
	\[
		p_n(\mathbf{y}) := \Pr{ \sum_{i=1}^n\mathbf{Y_i} = \mathbf{y}}
	\]
	and, as our assumptions imply that $\Sigma$ is positive-definite, we may also set
	\[
		\bar{p}_n(\mathbf{y}) := \frac{|\det \mathbf{D}| }{(2\pi n)^{d/2} \sqrt{\det  \mathbf{\Sigma}}}  \exp \left(-\frac{1}{2n}(\mathbf{y} - n \Ex{\mathbf{Y}})^\intercal \mathbf{\Sigma}^{-1} (\mathbf{y} - n \Ex{\mathbf{Y}}) \right).
	\]
	Then  
	\[
	\sup_{\mathbf{y} \in \mathbf{a} + \mathbf{D} \ndZ^d} |p_n(\mathbf{y}) - \bar{p}_n(\mathbf{y})| =o (n^{-d/2}).
	\]
	In particular, $p_n(\mathbf{y}) \sim \bar{p}_n(\mathbf{y})$ uniformly for  $(\mathbf{y} - n \Ex{\mathbf{Y}})/\sqrt{n}$ bounded.	
\end{lemma}

\subsection{A large deviation inequality for functions on finite Markov chains}

The following Chernoff-type bound for finite Markov chains was established by Lezaud~\cite{MR1627795} and does not require the chain to be reversible. 

\begin{lemma}[{\cite[Thm. 3.3]{MR1627795}}]
	\label{le:nonrev}
	\label{le:markov}
	Let $(X_n)_{n \ge 1}$ denote an irreducible Markov chain on a finite state space $S$ with transition matrix $\mathbf{P}$ and stationary distribution $\pi$. Suppose that the multiplicative symmetrization $\mathbf{K} = \mathbf{P}^\intercal \mathbf{P}$ is irreducible, and let $\epsilon(\mathbf{K})$ denote its spectral gap. Let $f: S \to [-1,1]$ denote a function  whose expected value with respect to the distribution $\pi$ equals $\mathbb{E}_\pi[f] = 0$. Let $b>0$ be a constant such that $0 < \|f\|_2 \le b$.  Then for each initial distribution $q = \cL(X_1)$ and each  $0 < \delta \le 1$ and~$n \ge 1$ it holds that
	\[
	\mathbb{P}({|f(X_1) + \ldots + f(X_n)| \ge \delta n }) \le 2 N_q \exp\left(- \frac{n \delta^2 \epsilon(\mathbf{K})}{8 b^2(1 + h(5\delta/b^2)} \right)
	\]
	where $N_q = \| q / \pi \|_2$ and 
	\[
		h(x) = \frac{1}{2}\left(\sqrt{1+x} -1 + \frac{x}{2} \right).
	\]
\end{lemma}

\subsection{A deviation inequality for random walk}
\label{sec:deviation}
The following deviation inequality is found in most textbooks on the subject.

\begin{lemma}[Medium deviation inequality for one-dimensional random walk]
	\label{le:deviation}
	Let $(X_i)_{i \in \ndN}$ be an i.i.d. family of real-valued random variables with $\Ex{X_1} = 0$ and $\Ex{e^{t X_1}} < \infty$ for all $t$ in some interval around zero. Then there are constants $\delta, c>0$ such that for all $n\in \ndN$, $x \ge 0$ and $0 \le\lambda\le\delta$ it holds that \[\Pr{|X_1 + \ldots + X_n| \ge x} \le 2 \exp(c n \lambda^2 - \lambda x).\]
\end{lemma}
The proof is by observing that $\Ex{e^{\lambda |X_1|}} \le 1 + c\lambda^2$ for some constant $c$ and sufficiently small $\lambda$, and applying Markov's inequality to the random variable $\exp(\lambda(|X_1| + \ldots + |X_n|))$.

\subsection{The cycle lemma}
The following combinatorial result is given for example in Tak{\'a}cs \cite{MR0138139}.
\begin{lemma}[The cycle lemma]
	\label{le:cycle}
	For each sequence $k_1, \ldots, k_n \ge -1$ of integers with $\sum_i k_i = -r \le 0$ there exist precisely $r$ values of $0 \le j \le n-1$ such that the cyclic shift \[(k_{1,j}, \ldots, k_{n,j}) :=   (k_{1+j},\ldots, k_n, k_1, \ldots, k_{j})\] satisfies $\sum_{i=1}^{u} k_{i,j} > -r$ for all $1 \le u \le n-1$.
\end{lemma}

\section{Random unlabelled weighted enriched trees with applications}
\label{sec:limits}

We develop a framework for random enriched trees considered up to symmetry and present our main applications to different models of random unlabelled graphs. 

\vspace{10pt}

\paragraph{ \em Index of notation} 

The following list summarizes frequently used terminology.

\begin{center}
	\begin{tabular}{@{}lp{358pt}@{}}
		$\cR^\kappa$ & $\kappa$-weighted species of $\cR$-structures, page \pageref{eq:omw}\\[2pt]
		$\cA_\cR^\omega$ & $\omega$-weighted species of $\cR$-enriched trees, page \pageref{eq:omw}\\[2pt]
		$\rho$ & radius of convergence of the ordinary generating series $\tilde{\cA}_\cR^\omega(z)$, page \pageref{mark:rho}\\[2pt]
		$\tilde{\mA}_n^\cR$ & random $n$-sized unlabelled $\cR$-enriched tree, page \pageref{mark:an}\\[2pt]
		$\mC_n^\omega$ & random $n$-sized unlabelled rooted block-weighted connected graph, page \pageref{mark:cn}\\[2pt]
		$(\cC^\bullet)^\omega$ & block-weighted species of rooted connected graphs, page \pageref{mark:c}\\[2pt]
		$\cB^\gamma$ & weighted species of $2$-connected graphs, page \pageref{mark:b}\\[2pt]
		$\cK$ & species of $k$-dimensional front-rooted trees, page \pageref{mark:k}\\[2pt]
		$\mK_n$ &  random unlabelled front-rooted $k$-tree, page \pageref{mark:kn}\\[2pt]
		$\cK^\circ$ & subspecies of $\cK$ of $k$-trees where the root front lies in precisely one hedron, page \pageref{mark:kc}\\[2pt]
		$\mK_n^\circ$ & random unlabelled $k$-tree from the class $\cK^\circ$ , page \pageref{mark:knc}\\[2pt]
		$d_G(\cdot, \cdot)$ & graph-metric on a connected graph $G$, page \pageref{mark:dc}\\[2pt]
		$d_{\textsc{block}}(\cdot, \cdot)$ & block-metric, page \pageref{mark:db}\\[2pt]
		$V_k(\cdot)$ & graph-distance $k$-neighbourhood, page \pageref{mark:nv}\\[2pt]
		$U_k(\cdot)$ & block-distance $k$-neighbourhood, page \pageref{mark:nu}\\[2pt]
		$f(A,v)$ & enriched fringe subtree of an enriched tree $A$ at a vertex $v$, page \pageref{mark:fringe}\\[2pt]
		$(\cT, \beta)$ & random $\Sym(\cR)$-enriched plane tree, page \pageref{le:sampler}\\[2pt]
		$\cT^f$ & fixpoint subtree corresponding to $(\cT, \beta)$, page \pageref{mark:tf}\\[2pt]
		$(\cT_n, \beta_n)$ & the random $\Sym(\cR)$-enriched plane tree $(\cT_n, \beta_n)$ conditioned on having $n$ vertices, page \pageref{mark:tn}\\[2pt]
		$\cT^f_n$ & fixpoint subtree corresponding to $(\cT_n, \beta_n)$, page \pageref{mark:tnf}\\[2pt]
		$(\cT^{(\ell)}, \beta^{(\ell)})$ & size-biased $\Sym(\cR)$-enriched tree, page \pageref{le:sampler2}\\[2pt]
		$(\cT^{(\infty)}, \beta^{(\infty)})$ & local weak limit of the $\Sym(\cR)$-enriched tree $(\cT_n, \beta_n)$, page \pageref{mark:ti}\\[2pt]
		$\tau^{[k]}$ & plane tree trimmed at height $k$, page \pageref{mark:tauk}\\[2pt]
		$(\tau, \gamma)^{<k>}$ & $\cG$-enriched tree trimmed at height $k$, page \pageref{mark:taugk}\\[2pt]
		$\mG$ & random $\cG$-object, page \pageref{mark:go}\\[2pt]
		$\hat{\mG}$ & random $\cG$-object with a bias on the number of fixpoints, page \pageref{mark:goh}\\[2pt]
		$\bar{\mG}$ & random $\cG$-object with a bias on the number of non-fixpoints, page \pageref{mark:gob}\\[2pt]
		$\hat{\mH}^\bullet=(\hat{\mH}, u^*)$ & the local limit of $(\cT_n, \beta_n)$ near a random vertex, page \pageref{mark:hk}\\[2pt]
	\end{tabular}        
	
\end{center}

\vspace{1 \baselineskip}

\subsection{Random weighted $\cR$-enriched trees}
\label{sec:intro1}

The concept of $\cR$-enriched trees was introduced by Labelle~\cite{MR642392}, and facilitates the unified treatment of a large variety of tree-like combinatorial structures. 



Given a species of structures $\cR$, the corresponding species of {\em $\cR$-enriched trees} $\cA_\cR$ is constructed as follows. For each finite set $U$ let $\cA_\cR[U]$ be the set of all pairs $(A, \alpha)$ with $A \in \cA[U]$ a rooted unordered tree with labels in $U$, and $\alpha$ a function that assigns to each vertex $v$ of $A$ with offspring set $M_v$ an $\cR$-structure $\alpha(v) \in \cR[M_v]$. The transport along a bijection $\sigma: U \to V$ relabels the vertices of the tree and the $\cR$-structures on the offspring sets accordingly. That is, $\cA_\cR[\sigma]$ maps the enriched tree $(A, \alpha)$ to the tree $(B, \beta)$ with $B=\cA[\sigma](A)$ and $\beta(\sigma(v)) = \cR[\sigma|_{M_{v}}](\alpha(v))$ for each $v \in A$. The species of $\cR$-enriched trees admits the combinatorial specification
\begin{align}
\label{eq:renr}
\cA_\cR \simeq \cX \cdot \cR(\cA_\cR),
\end{align}
as any $\cR$-enriched tree consists of a root vertex (corresponding to the factor $\cX$) together with an $\cR$-structure, in which each atom is identified with the root of a further $\cR$-enriched tree.  By Theorem~\ref{te:implicitspecies} it holds that given any species $\cF$ with an isomorphism $\cF \simeq \cX \cdot \cR(\cF)$, there is a natural choice of an isomorphism
$
\cF \simeq \cA_\cR.
$ 
Hence a large variety of combinatorial structures have a natural interpretation as enriched trees. 


Given a weighting $\kappa$ on the species $\cR$,  we obtain a weighting $\omega$ on the species $\cA_\cR$ given by 
\begin{align}
\label{eq:omw}
\omega(A, \alpha) = \prod_{v \in A} \kappa(\alpha(v)).
\end{align}
This weighting is consistent with the isomorphism in \eqref{eq:renr}, that is,
\begin{align}
\label{eq:weighting}
\cA_\cR^\omega \simeq \cX \cdot \cR^\kappa(\cA_\cR^\omega).
\end{align}

We are going to study  the random unlabelled enriched tree $\tilde{\mA}_n^\cR$,\label{mark:an} drawn with probability  proportional to its weight among all unlabelled objects with size $n$. In the following, we  illustrate how this model of random enriched trees generalizes various models of random graphs. The list is of course non-exhaustive as a huge variety of other special cases of $\tilde{\mA}_n^\cR$ may be found in the literature. 

 
\subsubsection{Simply generated P\'olya trees}
\label{sec:sigepo}
For $\cR = \Set$, the species $\cA_\cR$ describes rooted unordered trees. The corresponding unlabelled objects are also called P\'olya trees. Given a weight sequence $\om = (\omega_k)_k$, we may assign weight $\omega_k$ to each $k$-sized $\cR$-structure. Then $\tilde{\mA}_n^\cR$ is the random unordered unlabelled tree such that any P\'olya tree $A$ with $n$ vertices gets drawn with probability proportional to $\prod_{v \in A} \omega_{d^+_A(v)}$, with $d^+_A(v)$ denoting the outdegree of a vertex $v$. Note that setting weights to zero allows us to impose arbitrary degree restrictions. It is known that, depending on the weight-sequence, simply generated plane trees may show a very different behaviour, and it is natural to ask the same questions for simply generated P\'olya trees.

\subsubsection{Random unlabelled connected rooted graphs with weights on the blocks}
\label{sec:bijblock}

\begin{figure}[t]
	\centering
	\begin{minipage}{1.0\textwidth}
		\centering
		\includegraphics[width=1.0\textwidth]{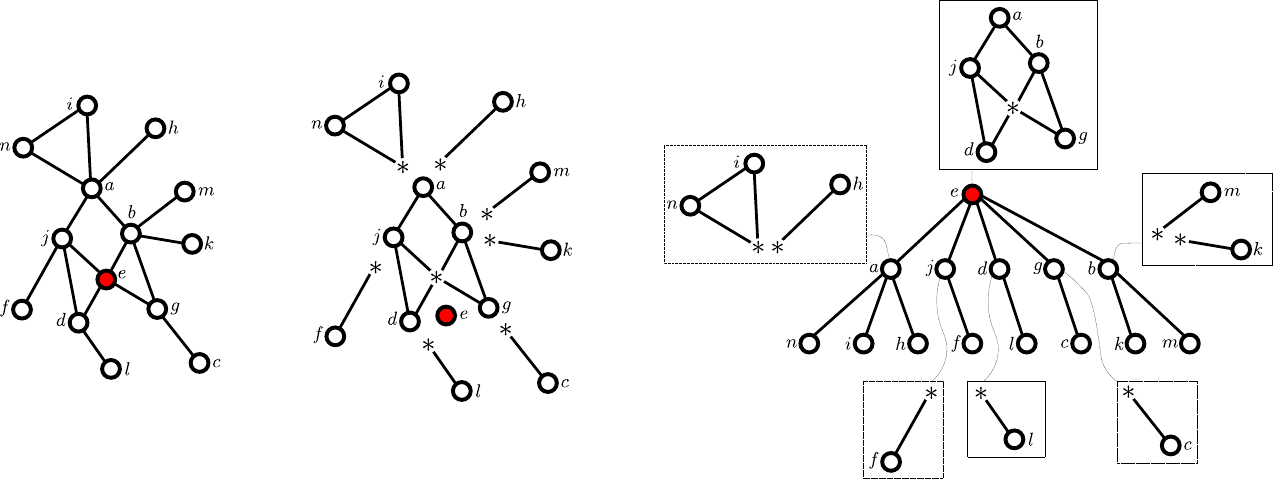}
		\caption{Correspondence of rooted connected graphs and enriched trees.}
		\label{fi:enrblockgraph}
	\end{minipage}
\end{figure}


The species $\cC$ connected graphs admits a decomposition in terms of the species $\cB$ of graphs that are $2$-connected  consist of two distinct vertices joint by an edge. The well-known combinatorial specification
\begin{align}
\label{eq:blode}
\cC^\bullet \simeq \cX \cdot \Set(\cB'(\cC^\bullet))
\end{align}
is illustrated in Figure~\ref{fi:enrblockgraph}. This  allows us to identify the species  $\cC^\bullet$  of rooted connected graphs with $\Set \circ \cB'$-enriched trees. That is, rooted trees, in which each offspring set gets partitioned, and each partition class $Q$ carries a $\cB'$-structure, that has $|Q| +1$ vertices, as the $*$-vertex receives no label. The isomorphism \eqref{eq:blode} can be found for example in Harary and Palmer \cite[1.3.3, 8.7.1]{MR0357214},  Robinson \cite[Thm. 4]{MR0284380}, and Labelle \cite[2.10]{MR699986}.

Let $\gamma$ be a weighting on the species $\cB$. \label{mark:b} We may consider the weighting $\omega$ on $\cC$ \label{mark:c} that assigns weight \[\omega(C) = \prod_B \gamma(B)\] to any graph $C$, with the index $B$ ranging over the blocks of $C$. The random graph $\mC_n^\omega$ \label{mark:cn} drawn from  the unlabelled $n$-sized  $\cC^\bullet$-objects with probability proportional to its $\omega$-weight is distributed like the random unlabelled enriched tree $\tilde{\mA}_n^\cR$ for the weighted species $\cR^\kappa=(\Set \circ \cB')^\kappa$, with $\kappa$ assigning the product of the $\gamma$-weights of the individual classes to any assembly of $\cB'$-structures. 

If all $\gamma$-weights are equal either to $0$ or $1$, we obtain random connected graphs from so called {\em block-classes} (or {\em block-stable} classes), that is, classes of graphs defined by placing constraints on the allowed blocks. For example, any class of graphs $\mathbf{Ex}(\cM)$ that may be defined by excluding a set $\cM$ of $2$-connected minors is also block-stable. Here a {\em minor} of a graph $G$ refers to any graph that may be obtained from $G$ by repeated deletion and contraction of edges. Prominent examples are outerplanar graphs $\mathbf{Ex}(K_4, K_{2,3})$, that may be drawn in the plane such that each vertex lies on the frontier of the infinite face, and series-parallel graphs $\mathbf{Ex}(K_4)$, that may be constructed similar to electric networks in terms of repeated serial and parallel composition. These two classes fall under the more general setting of random graphs from {\em subcritical} block-classes in the sense of  Drmota, Fusy, Kang, Kraus and Ru\'e \cite{MR2873207}, which also are special cases of the random graph $\mC_n^\omega$.

\begin{wrapfigure}{r}{0.3\textwidth}
	\begin{center}
		\includegraphics[width=0.15\textwidth]{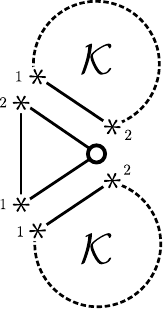}
	\end{center}
	\caption{Decomposition of the class $\cK^\circ$ for $k=2$.}
	\label{fi:ktree}
\end{wrapfigure}

\subsubsection{Random unlabelled front-rooted $k$-trees}
\label{sec:ktree}

Let $\cK$ \label{mark:k} denote the species of $k$-dimensional trees that are rooted at a front of distinguishable $*$-placeholder vertices. We let $\mK_n$ \label{mark:kn} denote the random unlabelled front-rooted $k$-tree that gets drawn uniformly at random among all unlabelled $\cK$-objects with $n$ hedra.

We consider the subspecies $\cK^\circ$ \label{mark:kc} front-rooted $k$-trees where the root-front is contained in precisely one hedron, and let $\mK_n^\circ$ \label{mark:knc} be sampled uniformly at random from the unlabelled $\cK$-objects with $n$ hedra.

Any element from $\cK$ may be obtained in a unique way by  glueing an arbitrary unordered collection of $\cK^\circ$-objects together at their root-fronts. Hence
\begin{align*}
\hspace{-100pt}
\cK \simeq \Set(\cK^\circ).
\end{align*}



As illustrated in Figure~\ref{fi:ktree}, any $\cK_{1}$-object may be constructed in a unique way by starting with a hedron $H$ consisting of the root-front and a vertex $v$, and then choosing, for each front $M$ of $H$ that contains $v$, a $k$-tree from $\cK$ whose root-front gets identified in a canonical way with $M$. Hence
\begin{align*}
\cK^\circ \simeq \cX \cdot \Seq_{\{k\}}(\cK).
\end{align*}
Combining the isomorphisms yields
\begin{align*}
\cK^\circ \simeq \cX \cdot (\Seq_{\{k\}}\circ \Set)(\cK^\circ).
\end{align*}
This identifies the species $\cK^\circ$ as $\Seq_{\{k\}}\circ \Set$-enriched trees, and the species $\cK$ as unordered forest of enriched trees. In particular, $\mK_n^\circ$ corresponds to the random unlabelled enriched tree $\tilde{\mA}_n^\omega$.

\subsection{Local convergence of random unlabelled enriched trees}
\label{sec:locunl}
{\em In the following, we let $\cR^\kappa$ denote an arbitrary  weighted species such that the inventory $|\cR[k]|_\kappa$ is positive for $k=0$ and for at least one $k \ge 2$. We let $\cA_\cR^\omega$ denote the corresponding species of weighted $\cR$-enriched trees. The radius of convergence of the ordinary generating series $\tilde{\cA}_\cR^\omega(z)$ will be denoted by $\rho$. \label{mark:rho}}

\subsubsection{A coupling with a random $\cG$-enriched (plane) tree}
\label{sec:unlcoupling}

Our first observation is that any symmetry $S=((T,\alpha), \sigma)$ of an $\cR$-enriched tree $A = (T,\alpha)$ admits a tree-like decomposition in form of a $\Sym(\cR)$-enriched tree $(T, \beta)$. Indeed, the automorphism $\sigma$ fixes the root $o$ of $T$ and permutes the roots of the $\cR$-enriched trees dangling from $o$ in such a way, that the induced permutation $\sigma(o)$ on the offspring of $o$ is an automorphism of the $\cR$-structure $\alpha(o)$. This yields an $\cR$-symmetry $\beta(o) := (\alpha(o), \sigma(o))$. For each fixpoint $v$ of the permutation $\sigma(o)$ it holds that the restriction of $\sigma$ to the $\cR$-enriched fringe subtree $f(A,v)$ \label{mark:fringe} (the maximum enriched subtree rooted at the vertex $v$) yields an $\cA_\cR$-enriched symmetry $(f(A,v), \sigma|_{f(A,v)})$ and we may proceed with the construction of $\beta$ in the same way. For each cycle $\tau=(v_1, \ldots, v_t)$ of $\sigma(o)$ having length $t \ge 2$ the situation is more complicated. We know that $\sigma$ permutes the $\cR$-enriched fringe subtrees $f(A, v_i)$ cyclically. Hence they are all structurally equivalent, and in fact, by the discussion in Section~\ref{sec:symmetry}, up to isomorphism composed out of isomorphic symmetries $(f(A, v_i), \sigma_i)$ with $\sigma_i = \sigma^t|_{f(A, v_i)}$. Hence we may proceed with the construction of $\beta$ as before, by considering the individual symmetries. This process is illustrated in Figure~\ref{fi:symdecomp}. Note that the $\Sym(\cR)$-enriched tree $(T, \beta)$ does not contain all information about the symmetry $S$, but we may reconstruct $S$ up to relabelling from $(T, \beta)$.

\begin{figure}[t]
	\centering
	\begin{minipage}{1.0\textwidth}
		\centering
		\includegraphics[width=0.5\textwidth]{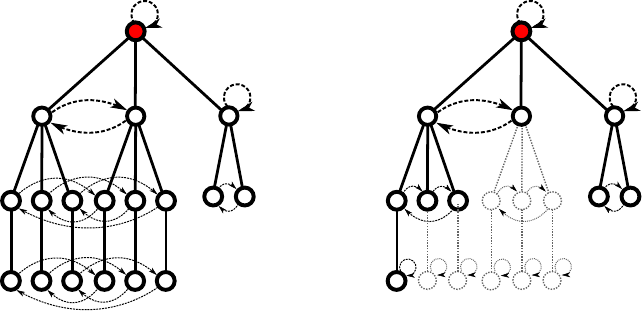}
		\caption{The encoding of $\cA_\cR$-symmetries as $\Sym(\cR)$-enriched trees.}
		\label{fi:symdecomp}
	\end{minipage}
\end{figure}

The fixpoints of the automorphism $\sigma$ form a subtree $T^f$ of $T$. Each fixpoint $v$ has a possibly empty set $f(v)$ of other fixpoints as offspring, and the remaining offspring correspond to a forest $F(v)$ of $\Sym(\cR)$-enriched fringe subtrees $f((T,\beta), v_i)$, which consist of non-fixpoints of $\sigma$. We are going to say the triple $G(v) := (\beta(v), f(v), F(v))$ is a $\cG$-object {\em on} the fixpoints $f(v)$ and define $|f(v)|$ to be its size. Formally, $\cG$-objects do not correspond to any species, but the analogy is clear, and we may call $(T^f, (G(v))_{v \in T^f})$ a $\cG$-enriched tree. 

Similarly, we may define the concept of a {\em $\cG$-enriched plane tree}, in which the label set of each occurring $\cR$-symmetry is required to belong to the collection $\{[k] \mid k \ge 0\}$. We are going to use the following recursive procedure illustrated in Figure~\ref{fi:sampler1} in order to sample random $\cA_\cR$-symmetries according to a weighted Boltzmann-distribution.

\begin{wrapfigure}{r}{0.32\textwidth}
	\begin{center}
		\includegraphics[width=0.22\textwidth]{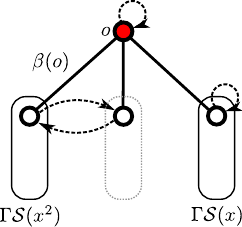}
	\end{center}
	\caption{The sampler $\Gamma \cS(x)$.}
	\label{fi:sampler1}
\end{wrapfigure}

\begin{lemma}[A coupling of random unlabelled $\cR$-enriched trees with random $\cG$-enriched trees]
	\label{le:sampler}
	For any parameter $x>0$ with $\tilde{\cA}^\omega_\cR(x) < \infty$ consider the following recursive procedure $\Gamma \cS(x)$ which draws a random $\Sym(\cR)$-enriched plane tree $(\cT, \beta)$.
	\begin{enumerate}[1.]
		\item Start with a root vertex $o$ and
		draw a random symmetry \[\beta(o)=(R(o), \sigma(o))\] from the set $\bigcup_{k \ge 0} \Sym(\cR)[k]$ such that $\beta(o)$ gets drawn with probability proportional to 
		\[
		\frac{\kappa(R(o))}{|R(o)|!}\tilde{\cA}^\omega_{\cR}(x)^{\sigma_1(o)} \tilde{\cA}^\omega_{\cR}(x^2)^{\sigma_2(o)} \cdots.
		\] Here $\sigma_i(o)$ denotes the number of $i$-cycles of the permutation $\sigma(o)$, with fixpoints counting as $1$-cycles. 
		\item  For each cycle $\tau$ of the permutation $\sigma(o)$ draw an independent copy $(\cT^\tau, \beta^\tau)$ of the recursively called sampler $\Gamma \cS(x^{|\tau|})$. Here $|\tau| \ge 1$ denotes the length of the cycle. For each atom $a$ of $\tau$ make an identical copy $(\cT^a, \beta^a)$ of $(\cT^\tau, \beta^\tau)$.
		\item Let $k$ denote the size of the $\Sym(\cR)$-structure $\beta(o)$. For each label $a \in [k]$ add an edge between the root vertex $o$ and the root of the plane tree $\cT^a$. The ordering of the offspring set is given by the order on the label set $[k]$. This defines a plane tree $\cT$ with root-vertex $o$. Moreover, for each $a \in [k]$ and each vertex $v \in \cT^a$ set $\beta(v) := \beta^a(v)$. This defines a $\Sym(\cR)$-enriched plane tree $(\cT, \beta)$.
	\end{enumerate}
	This procedure terminates almost surely and the resulting $\Sym(\cR)$-enriched plane tree $(\cT, \beta)$ corresponds to a symmetry on the vertex set of the plane tree $\cT$. Let $\Gamma Z_{\cA_\cR^\omega}(x)$ denote the result of relabelling this symmetry uniformly at random with labels from the set $[n]$, with $n$ denoting the number of vertices of the tree $\cT$. Then for any symmetry $(A, \sigma)$ from the set $\bigcup_{k \ge 0} \Sym(\cA_\cR)[k]$ it holds that
	\[
	\Pr{\Gamma Z_{\cA_\cR^\omega}(x) = (A, \sigma)} = \omega(A) \frac{x^{|A|}}{|A|!}  \tilde{\cA}_\cR^\omega(x)^{-1}
	.\]
\end{lemma}

If we condition the sampler $\Gamma Z_{\cA_\cR^\omega}(x)$ on producing a symmetry with size $n$, then any symmetry from $\Sym(\cA_\cR)[n]$ gets drawn with probability proportional to the $\omega$-weight of its $\cA_\cR$-object. By the discussion in Section~\ref{sec:preop} it follows that the isomorphism class of this $\cA_\cR$-object is distributed like the random unlabelled $\cR$-enriched tree $\tilde{\mA}^\cR_n$. 
	
Suppose that the radius of convergence $\rho$ of the ordinary generating series $\tilde{\cA}^\omega_\cR(z)$ is positive. As we state below, it holds that $\tilde{\cA}^\omega_\cR(\rho)$ is finite and hence we may consider the random $\Sym(\cR)$-enriched $(\cT, \beta)$ be drawn according to the sampler $\Gamma \cS(\rho)$. The vertices of $\cT$ that correspond to fixpoints of the symmetry $\Gamma Z_{\cA_\cR^\omega}(\rho)$ form a subtree $\cT^f \subset \cT$ containing the root. Note that by the discussion in Section~\ref{sec:symmetry} the fixpoints correspond precisely to the vertices in which the sampler $\Gamma \cS$ calls itself with parameter $\rho$ (as opposed to parameter $\rho^i$ for some $i \ge 2$). 
	
For each vertex $v$ of $\cT^f$ \label{mark:tf} let $G_{\cT^f}(v)=(\beta(v), f(v), F(v))$ denote the corresponding $\cG$-object. Moreover, let $(\cT_n, \beta_n)$, $\cT^f_n$ \label{mark:tn} \label{mark:tnf} and $G_{\cT^f_n}(\cdot) = (\beta_n(\cdot), f_n(\cdot), F_n(\cdot))$ denote the corresponding random variables conditioned on the event $|\cT| = n$. Let $\mG$ \label{mark:go} be a random variable that is identically distributed to the $\cG$-object $G_{\cT^f}(o)$ corresponding to the root $o$ of $\cT^f$. Moreover, let $\xi$ denote the number of the fixpoints of $\mG$ and $\zeta$ the size of the enriched forest corresponding to the non-fixpoints.

\begin{lemma}[Properties of the coupling with $\cG$-enriched trees] \label{le:ucoup} We make the following observations.
\begin{enumerate}
\item The radius of convergence $\rho$ of  $\tilde{\cA}^\omega_\cR(z)$ and the sum $\tilde{\cA}^\omega_\cR(\rho)$ are both finite.
\item The size of the tree $\cT$ satisfies
\begin{align*}
|\cT| = \sum_{v \in \cT^f} (1 + |F(v)|).
\end{align*}
\item For any $\Sym(\cR)$-enriched plane tree $(\cT', \beta')$ corresponding to $\cG$-objects $G_1, \ldots, G_\ell$ it holds that
\begin{align*}
\Pr{(\cT, \beta)  = (\cT', \beta')} = \prod_{i=1}^\ell \Pr{\mG = G_i}.
\end{align*}
\item
An arbitrary sequence of $\cG$-objects $G_i = (S_i, f_i, F_i)$, $i=1,\ldots ,\ell$ corresponds to a $\Sym(\cR)$-enriched tree if and only if
\begin{align*}
\sum_{i=1}^\ell |f_i| = \ell -1 \quad \text{and} \quad \sum_{i=1}^m |f_i| \ge m \quad \text{for all} \quad 1 \le m \le \ell-1.
\end{align*}
Let $\mG_i = (\mathsf{S}_i, \mathsf{f}_i, \mathsf{F}_i)$ denote independent identical copies of $\mG$. Let $L$ denote depth-first-search ordered list $L$ of the $\cG$-objects of $\cT^f_n$ and $|L|$ its length. Then $(L \mid |L|=\ell)$ is distributed like
\begin{align*}
((\mG_1, \ldots, \mG_{\ell}) \mid  \sum_{i=1}^\ell (1 + |\mathsf{F}_i|) = n, \sum_{i=1}^\ell |\mathsf{f}_i| = \ell -1,  \sum_{i=1}^m |\mathsf{f}_i| \ge m \text{ for all } 1 \le m \le \ell-1).
\end{align*}
\item 
The plane tree $\cT^f$ is distributed like a Galton--Watson tree with offspring distribution $\xi$ having probability generating function
\begin{align*}
\label{eq:pgf1} \Ex{z^\xi} = Z_{\cR^\kappa}(z \tilde{\cA}_\cR^\omega(\rho), \tilde{\cA}_\cR^\omega(\rho^2), \ldots) \rho / \tilde{\cA}_\cR^\omega(\rho).
\end{align*}
\item 
Given $\cT^f$, the forests $(F(v))_{v \in \cT^f}$ are conditionally independent. The conditional distribution of each forest depends only on the outdegree $d^+_{\cT^f}(v)$. The distribution of the forest size $\zeta$ is given by its probability generating function 
\begin{align*}
\Ex{z^\zeta} = Z_{\cR^\kappa}(\tilde{\cA}_\cR^\omega(\rho), \tilde{\cA}_\cR^\omega((\rho z)^2), \tilde{\cA}_\cR^\omega((\rho z)^3),\ldots) \rho / \tilde{\cA}_\cR^\omega(\rho).
\end{align*}
\end{enumerate}
\end{lemma}

In a more specific setting, where the random vector $(\xi,\zeta)$ has finite exponential moments, even more can be said. 
\begin{lemma}[Further properties of the coupling with $\cR$-enriched trees in a specific setting]
	\label{le:asymptotic}
	 Suppose that $\rho>0$ and that the function \[E(z, u) = z Z_{\cR^{\kappa}}(u, \tilde{\cA}^\omega_\cR(z^2), \tilde{\cA}^\omega_\cR(z^3), \ldots)\] satisfies $E(\rho + \epsilon,\tilde{\cA}^\omega_\cR(\rho) + \epsilon) < \infty$ for some $\epsilon>0$. 
	\begin{enumerate}
		\item Then the $n$th coefficient of $\tilde{\cA}^\omega_\cR(z)$ is asymptotically given by
	\[
	[z^n] \tilde{\cA}^\omega_\cR(z) \sim \spa(\mathbf{w}) \sqrt{\frac{\rho E_z(\rho, \tilde{\cA}^\omega_\cR(\rho))}{2\pi E_{uu}(\rho, \tilde{\cA}^\omega_\cR(\rho))}} \rho^{-n} n^{-3/2}
	\]
	as $n \equiv 1 \mod \spa(\mathbf{w})$ tends to infinity. 
	\item The series $\tilde{\cA}^\omega_\cR(z)$ has square root singularities at the points \[
	s_k = \rho \exp(2\pi i k/ \spa(\mathbf{w})), \quad k=0, \ldots, \spa(\mathbf{w})-1,
	\] with local expansions as analytic functions of $\sqrt{1- z/s_k}$.
	\item The offspring distribution $\xi$ of the Galton--Watson tree $\cT^f$ and the random variable $\zeta$ have finite exponential moments. Moreover,
	\begin{align*}
	\Ex{\xi}&=E_u(\rho, \tilde{\cA}_\cR^\omega(\rho)) = \rho \widetilde{(\cR')^\kappa \circ \cA_\cR^\omega}(\rho) = 1, \\ 	\Va{\xi}&=E_{uu} (\rho, \tilde{\cA}_\cR^\omega(\rho))\tilde{\cA}_\cR^\omega(\rho), \quad \text{and} \quad 
	\Ex{\zeta}= E_z(\rho, \tilde{\cA}_\cR^\omega(\rho)) \rho / \tilde{\cA}_\cR^\omega(\rho) -1.
	\end{align*}
	Consequently,
	\[
	\Pr{|\cT|=n} \sim \spa(\mathbf{w}) n^{-3/2}  \sqrt{\frac{1 + \Ex{\zeta}}{2\pi \Va{\xi}}}
	\]
	\item Suppose that at least one $\cR$-structure with positive $\kappa$-weight has a non-trivial automorphism. Then the lattice spanned by the support of $(\xi, \zeta)$ has a $2$-dimensional $\ndZ$-basis $\mathbf{B} \in \ndZ^{2\times2}$, and the covariance matrix $\mathbf{\Sigma}$ of $(\xi, \zeta)$ is positive-definite. Set
	\[
		\mu = \frac{1}{1 + \Ex{\zeta}},  \quad \sigma^2 = \frac{\det \mathbf{\Sigma}}{\Va{\zeta}(1 + \Ex{\zeta})^3}, \quad \text{and} \quad d = \frac{|\det \mathbf{B}|}{\spa(\mathbf{w})}.
	\]
	Then, as $n \equiv 1 \mod \spa(\mathbf{w})$ tends to infinity, 
	\[
		\sqrt{n} \Pr{|\cT_n^f| = \ell} \sim  \frac{d}{\sigma \sqrt{2 \pi }} \exp(- \frac{x^2}{2 \sigma^2})
	\]
	uniformly for all bounded $x$ satisfying
	\[
		\ell := \mu n + x \sqrt{n} \in n + d \ndZ.
	\]
	In particular,
	\[
		\frac{|\cT_n^f| - n \mu}{\sqrt{n}} \convdis \cN(0, \sigma^2).
	\]
	\end{enumerate}
\end{lemma} 

Properties {\em (1)} and {\em (2)} are an application of results by Bell, Burris and Yeats \cite{MR2240769}. The requirement in {\em (4)} that $\cR$ has at least one structure (with positive $\kappa$-weight) with non-trivial symmetries is not really a restriction. If it fails, then almost surely $|\cT_n^f|=n$ for all $n$, which makes the analysis of $\tilde{\mA}_n^\omega$ even easier.


\subsubsection{Local convergence around the fixed root}

\label{sec:localenrunl}

\begin{figure}[t]
	\centering
	\begin{minipage}{1.0\textwidth}
		\centering
		\includegraphics[width=0.5\textwidth]{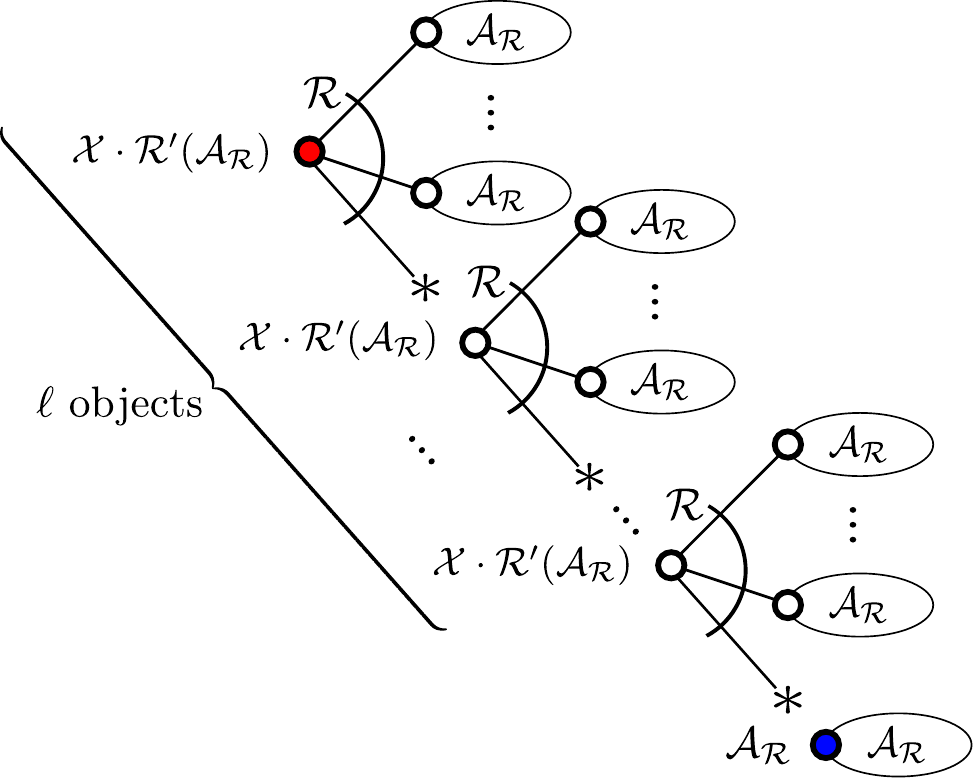}
		\caption{The decomposition of $\cA_\cR^{(\ell)}$.}
		\label{fi:pointenr}
	\end{minipage}
\end{figure}

\begin{wrapfigure}[9]{r}{0.4\textwidth}
	\begin{center}
		\includegraphics[width=0.25\textwidth]{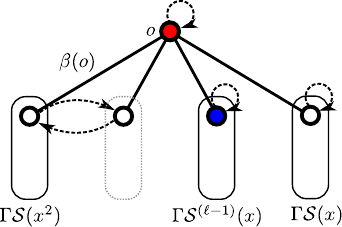}
	\end{center}
	\caption{The sampler $\Gamma \cS^{(\ell)}(x)$.}
	\label{fi:sampler2}
\end{wrapfigure}

Labelle established in \cite[Thm. A]{MR642392} the following decomposition of pointed $\cR$-enriched trees, which will aid us studying the behaviour of the $\cR$-structures along paths starting from the root in random enriched trees. The weighted species $\cA_\cR^\omega$ satisfies an isomorphism $\cA_\cR^\omega \simeq \cX \cdot \cR^\kappa(\cA_\cR^\omega)$. The derivative operator satisfies a product rule similar to the product rule for the derivative of smooth functions, see Proposition~\ref{pro:chainprod}. Hence applying the pointing operator yields a weight-compatible isomorphism 
\newline\begin{minipage}{0.60\textwidth}
\vspace{-3pt}
\begin{align*}
\cA_\cR^\bullet \simeq \cA_\cR + \cX \cdot \cR'(\cA_\cR) \cdot \cA_\cR^\bullet.
\end{align*}
\end{minipage}
\vspace{3pt}
\newline We may apply Joyal's implicit species theorem \cite[Th. 6]{MR633783} in order to unwind this recursion and obtain an isomorphism \[
\cA_\cR^\bullet \simeq \sum_{\ell=0}^\infty \cA_\cR^{(\ell)}, \quad \cA_\cR^{(\ell)} = (\cX \cdot \cR'(\cA_\cR))^\ell \cA_\cR
\] corresponding to the pointed enriched trees $((A,\alpha), v)$ in which the outer root $v$ has height $\ell$ in the rooted tree $A$. The correspondence is illustrated in Figure~\ref{fi:pointenr}. Again, this isomorphism is compatible with the weightings, and we may use it to construct the following sampler illustrated in Figure~\ref{fi:sampler2}.

\begin{lemma}[A modified random $\Sym(\cR)$-enriched tree]
	\label{le:sampler2}
	For any integer $\ell \ge 0$ and parameter $x > 0$ with $(\tilde{\cA}_\cR^{(1)})^\omega(x) < \infty$ consider the following recursive procedure $\Gamma \cS^{(\ell)}(x)$ that samples a random $\Sym(\cR)$-enriched plane tree $(\cT^{(\ell)}, \beta^{(\ell)})$ together with a distinguished vertex $r$ which we call the  outer root.
	\begin{enumerate}[1.]
		\item If $\ell = 0$ then return (an independent copy of) the random enriched plane tree $(\cT, \beta)$ from Lemma~\ref{le:sampler} with the outer root being the root-vertex of $\cT$. Otherwise, if $\ell \ge 1$, then proceed with the following steps.
		\item Start with a root vertex $o$ and draw a random $\cR'$-symmetry $(R, \sigma)$ from $\bigcup_{k \ge 0} \Sym(\cR')[k]$ with probability proportional to  \[\frac{\kappa(R)}{|R|!}\tilde{\cA}^\omega_{\cR}(x)^{\sigma_1} \tilde{\cA}^\omega_{\cR}(x^2)^{\sigma_2} \cdots.\] Set $k := |R|$ and make a uniformly at random choice of a bijection $f$ from the set $[k] \cup \{*_k\}$ of labels of the $\cR$-structure $R$ to the set of integers $[k+1]$. Relabel the symmetry via the transport function: \[\beta(o) := (R(o), \sigma(o)) := \Sym(\cR)[f](R,\sigma).\] Let $b := f(*_k)$ denote the vertex corresponding to $*_k$.
		\item Note that $b$ is a fixpoint of the permutation $\sigma(o)$. For each cycle $\tau \ne (b)$ of $\sigma(o)$ draw an independent copy $(\cT^\tau, \beta^\tau)$ of the sampler $\Gamma \cS(x^{|\tau|})$ with $|\tau| \ge 1$ denoting the length of the cycle. For each atom $a$ of the cycle $\tau$ make an identical copy $(\cT^a, \beta^a)$ of $(\cT^\tau, \beta^\tau)$.
		\item Draw an independent copy $(\cT^b, \beta^b)$ of the sampler $\Gamma \cS^{(\ell-1)}(x)$.
		\item For each label $a \in [k+1]$ add an edge between the root vertex $o$ and the root of the plane tree $\cT^a$. The ordering of the offspring set is given by the order on the label set $[k+1]$. This defines a plane tree $\cT$ with root-vertex $v$. Moreover, for each $a \in [k+1]$ and each vertex $u \in \cT^a$ set $\beta(u) := \beta^a(u)$. This defines an $\Sym(\cR)$-enriched plane tree $(\cT^{(\ell)}, \beta^{(\ell)})$.
	\end{enumerate}
	This procedure terminates almost surely. As described in Section~\ref{sec:unlcoupling}, the resulting $\Sym(\cR)$-enriched plane tree $(\cT^{(\ell)}, \beta^{(\ell)})$ corresponds to a symmetry on the vertex set of the tree $\cT^{(\ell)}$. Let $n$ denote the number of vertices of the tree $\cT^{(\ell)}$ and let $\Gamma Z_{\cA_\cR^{(\ell)}}(x)$ denote the result of relabelling this symmetry uniformly at random with labels from the set $[n]$. Then $\Gamma Z_{\cA_\cR^{(\ell)}}(x)$ satisfies a weighted Boltzmann distribution, that is, for any symmetry $((A,r), \sigma)$ from the set $\bigcup_{k=0}^\infty \Sym(\cA_\cR^{(\ell)})[k]$ we have that
	\begin{align}
	\label{eq:holygrail}
	\Pr{\Gamma Z_{\cA_\cR^{(\ell)}}(x) = ((A,r), \sigma)} = \omega(A) \frac{x^{|A|}}{|A|!} \tilde{\cA}_\cR^{(\ell)}(x)^{-1} 
	\end{align}
	In particular,
	\[
	\Pr{\Gamma Z_{\cA_\cR^{(\ell)}}(x) = ((A,r), \sigma)} = (x \widetilde{\cR' \circ \cA_\cR^\omega}(x))^\ell \Pr{\Gamma Z_{\cA_\cR^\omega}(x) = (A,\sigma)}.
	\]
\end{lemma}

Suppose that $\rho>0$ and consider the $\Sym(\cR)$-enriched tree $(\cT^{(\ell)}, \beta^{(\ell)})$ generated by the procedure $\Gamma \cS^{(\ell)}(\rho)$. The path from the root to the distinguished vertex in $\cT^{(\ell)}$ is its {\em spine}. If we set $\ell=\infty$, the above construction yields an infinite but locally finite $\Sym(\cR)$-enriched tree $(\cT^{(\infty)}, \beta^{(\infty)})$ \label{mark:ti} having an infinite spine. We will show that this object, is the local limit weak limit of the random graph $\mA_n^\omega$, if certain conditions are met.

In order to formalize our notion of local convergence, we require the concept of trimmed $\cG$-enriched trees. For any $\cG$-enriched plane tree $(\tau,\gamma)$ and any  non-negative integer $k$ let $(\tau, \gamma)^{<k>}$ \label{mark:taugk}  denote the result of trimming at height $k$. That is,
\[
(\tau, \gamma)^{<k>} = (\tau^{[k]}, (\gamma(v))_{v \in \tau^{[k-1]}}),
\]
with $\tau^{[k]}$ \label{mark:tauk} denoting the plane tree trimmed hat height $k$. That is, we delete all vertices from $\tau$ with height larger than $k$.
In order to simplify notation, we also set $(\cT', \beta')^{<k>} := (\tau, \gamma)^{<k>}$ for the $\Sym(\cR)$-enriched plane tree $(\cT', \beta')$ corresponding to $(\tau, \gamma)$.

\begin{theorem}[Local convergence of random unlabelled $\cR$-enriched trees]
	\label{te:localunlabelled}
	Suppose that the ordinary generating series generating series $\tilde{\cA}^\omega_\cR(z)$ has radius of convergence $\rho > 0$, and that the series \[E(z, u) = z Z_{\cR^{\kappa}}(u, \tilde{\cA}^\omega_\cR(z^2), \tilde{\cA}^\omega_\cR(z^3), \ldots)\] satisfies $E(\rho + \epsilon, \tilde{\cA}^\omega_\cR(\rho) + \epsilon) < \infty$ for some $\epsilon>0$. Then for any sequence $k_n = o(\sqrt{n})$ of non-negative integers it holds that
	\[
		d_{\textsc{TV}}((\cT_n, \beta_n)^{<k_n>}, (\cT^{(\infty)}, \beta^{(\infty)})^{<k_n>}) \to 0.
	\]
	as $n$ becomes large.
\end{theorem}

The limit object $(\cT^{(\infty)}, \beta^{(\infty)})$ admits a more accessible description in terms of $\cG$-enriched trees, that we are going to use in the proof of Theorem~\ref{te:localunlabelled}. Let $\hat{\mG}$ \label{mark:goh} denote a random variable that is distributed like the $\cG$-object corresponding to the root of $\cT^{(\infty)}$. Here we do not explicitly distinguish the vertex corresponding the $*$-vertex of the $\cR'$-symmetry.

\begin{lemma}
	\label{le:slow}
	Suppose that $\rho >0$, $\Ex{\xi}=1$ and that $(\xi, \zeta)$ has a finite covariance matrix.
	\begin{enumerate}
	\item The $\cG$-object corresponding to the root of $\cT^{(\infty)}$ together with the fixpoint corresponding to the second spine-vertex is distributed like $\mG$ with a uniformly at random selected marked fixpoint.
	\item The distribution of the limit enriched tree $(\cT^{(\infty)}, \beta^{(\infty)})$ may be described as follows. There are normal fixpoints and mutant fixpoints, and we start with a mutant root. Each normal fixpoint receives as $\cG$-object an independent copy of $\mG$, and each of the fixpoints of this $\cG$-object is declared normal. Ever mutant fixpoint receives an independent copy of $\hat{\mG}$, and one of the corresponding fixpoints is selected uniformly at random and declared mutant, whereas the remaining fixpoints are declared normal.
		\item Let $(\tau, \gamma)$ denote a $\cG$-enriched tree with height at least $k$. Let $G_1, \ldots, G_t$ denote the depth-first-search ordered list of the $\cG$-objects of $\tau^{[k-1]}$. Then
		\[
		\Pr{ (\cT^{(\infty)}, \beta^{(\infty)})^{<k>} = (\tau, \gamma)^{<k>}}= L_k(\tau) \prod_{i=1}^t \Pr{\mG = G_i}.
		\]
	\item 	For any $\cG$-enriched tree $(\tau, \gamma)$ with $\gamma(v) = (\beta_\tau(v), f_\tau(v), F_\tau(v))$ and any integer $k \ge 0$ we set
	\[
	\qquad L_k(\tau) = |\{v \in \tau \mid \he_\tau(v)  = k\}|, \quad L_k^\cG(\tau) = \sum_{\substack{v \in \tau\\ \he_\tau(v) = k}} |F_\tau(v)|, \quad H_k^\cG(\tau) = \sum_{i=0}^k L_i^\cG(\tau).
	\]
	Let $(\hat{\xi}, \hat{\zeta})$ denote the sizes of the fixpoints and non-fixpoints of $\hat{\mG}$. When $(\tau, \gamma)$ is the $\cG$-enriched tree corresponding to $(\cT^{(\infty)}, \beta^{(\infty)})$, it holds for all $k \ge 1$ that
	\begin{align*}
		\Ex{L_k(\tau)} &= 
		k(\Ex{\hat{\xi}} -1)+1 \\
		\Ex{|\tau^{[k]}|} &= 
		k(k+1)(\Ex{\hat{\xi}}-1)/2+ k+ 1 \\
		\Ex{L_k^\cG(\tau)} &= 
		k(\Ex{\hat{\xi}}-1)\Ex{\zeta} + \Ex{\hat{\zeta}} \\
		\Ex{H_k^\cG(\tau) -  |\tau^{[k]}| \Ex{\zeta}} &=  (k+1)(\Ex{\hat{\zeta}} - \Ex{\zeta}) \\
		\Va{H_k^\cG(\tau) -  |\tau^{[k]}| \Ex{\zeta}} &= k(k+1)(\Ex{\hat{\xi}}-1)\Va{\zeta}/2+ (k+ 1)(\Va{\hat{\zeta}} + \Va{\zeta})
	\end{align*}
	\end{enumerate}
\end{lemma}

\subsubsection{Local convergence around a random root}

Inspired by Aldous' approach \cite{MR1102319} on fringe subtrees of random trees, we may also treat local convergence with respect to a uniformly at random drawn root of the random unlabelled enriched tree $\tilde{\mA}_n^\omega$ in a similar manner. 
Let $v^*$ be an uniformly at random drawn vertex of the tree $\cT_n$, and let $v_0$ denote the unique nearest vertex of $\cT_n^f$. That is, $v_0 = v^*$ if $v^* \in \cT_n^f$, and otherwise $v_0$ is the unique vertex of the fixpoint tree whose $\cG$-object contains $v^*$. 
For any $i \ge 1$, let $v_{i}$ denote the $i$'th predecessor of $v_0$ in the fixpoint tree $\cT_n^f$, if this predecessor exists. If not, that is, if $v_0$ has height greater than $i$ in $\cT_n^f$, then set $v_i = \diamond$ and $f((\cT_n, \beta_n), v_i)  = \diamond$ for some symbol $\diamond$ not contained in the set of $\cG$-enriched trees. (For example, we could use the empty set.) For any $k \ge 0$ we consider the vector of increasing fringe subtrees \label{mark:hn}
\[
\mH^n_{[k]} := (\mH^n_i)_{0 \le i \le k} := (f((\cT_n, \beta_n), v_0), \ldots, f((\cT_n, \beta_n), v_{k})).
\]

We are going to establish convergence of these random vectors of enriched trees toward extended enriched fringe subtrees of a limit object, which we introduce in the following lemma.

\begin{lemma}
	\label{le:slow2}
	Suppose that $\rho >0$, $\Ex{\xi}=1$ and that $(\xi, \zeta)$ has a finite covariance matrix.
	\begin{enumerate}
		\item Let $\bar{\mG} = (\bar{\mathsf{S}}, \bar{\mathsf{f}}, \label{mark:gob} \bar{\mathsf{F}})$ denote the random $\cG$-object with distribution given by
		\[
			\Pr{\bar{\mG} = (S,f,F)} = (1+ |F|) \Pr{\mG = (S,f,F)} / (1 + \Ex{\zeta}).
		\]
		We define an infinite random $\cG$-enriched tree $\hat{\mathsf{H}}$ in terms of its sequence $(\hat{\mathsf{H}}_k)_{k \ge 0}$ of increasing extended enriched fringe subtrees. The distribution of the fringe subtree tree $\hat{\mathsf{H}}_k$ is given as follows. There are normal fixpoints and special fixpoints, and we start with a special root. Each normal fixpoint receives as $\cG$-object an independent copy of $\mG$, and all fixpoints of this $\cG$-object are declared normal. Every special fixpoint with height less than $\ell$ receives an independent copy of $\hat{\mG}$, and one of the corresponding fixpoints is selected uniformly at random and declared special, whereas the remaining fixpoints are declared normal. A special fixpoint with height $\ell$ receives $\bar{\mG}$ and all fixpoints of this $\cG$-object are declared normal.
		
		Then the special vertices of $\hat{\mathsf{H}}$ form an infinite spine $u_0, u_1, \ldots$ that grows backwards, with 
		\[
			\hat{\mathsf{H}}_k = f(\hat{\mathsf{H}}, u_k)
		\]
		for all $k$. We distinguish a point $u^*$ that is drawn uniformly at random from the set $\{u_0\} \cup \hat{\mathsf{F}}(u_0)$, with $\hat{\mathsf{F}}(u_0)$ denoting the set of non-fixpoints of $\cG$-object corresponding to the root of $\hat{\mH}_0$, and set $\hat{\mH}^\bullet = (\hat{\mH}, u^*)$. \label{mark:hk}

		\item 
		For any two $\cG$-enriched trees $A$ and $A'$, let $Q(A,A')$ denote the number of fixpoint sons $v$ of the root of $A$ with  $f(A,v) = A'$. 
		For any increasing finge subtree representation $\mathbf{H} = (H_i)_{0 \le i \le k}$ of a $\cG$-enriched tree we set
		\[
		p(\mathbf{H}) = \prod_{i=1}^{k} Q(H_i, H_{i-1}).
		\]
		Then  $p(\mathbf{H})$ counts the number of fixpoints $v$ at height $k$ in $H_k$ with the property, that the extended enriched fringe subtree respresentation with respect to $v$ is identical to $\mathbf{H}$.
		
		\item Let $u$ be either the root of $H_0$ or a non-fixpoint of $\cG$-object corresponding to the root of $H_0$, and let $G_1, \ldots, G_t$ be the $\cG$-objects corresponding to $H_k$. Then 
		\[
		\Pr{\hat{\mH}_{[k]} = \mathbf{H}} = p(\mathbf{H})(1 + \Ex{\zeta})^{-1} \prod_{i=1}^t \Pr{\mG = G_i}.
		\]
		\item For any $\cG$-enriched tree $(\tau, \gamma)$ with $\gamma(v) = (\beta_\tau(v), f_\tau(v), F_\tau(v))$  let $\#_f(\tau,\gamma)= |\tau|$ denote its number of fixpoints, and $\#(\tau, \gamma) = \sum_{v \in \tau}(1 + F_\tau(v))$ its total size. Then for any sequence $k_n = \sqrt{n} t_n$ of non-negative integers with $t_n = o(1)$ it holds with probability tending to one that
		\begin{align*}
		\qquad \quad \#_f \hat{\mH}_{k_n} \le  n t_n \quad \text{and} \quad |\#\hat{\mH}_{k_n} - \Ex{\zeta} \#_f\hat{\mH}_{k_n} | \le \sqrt{n t_n}. \\
		\end{align*}
	\end{enumerate}
\end{lemma}

We may now establish convergence of the extended enriched fringe subtrees, that will help us to apply our main theorems to specific examples of random discrete structures, in particular random graphs.

\begin{theorem}[Local convergence of random unlabelled $\cR$-enriched trees around a random root]
	\label{te:localhaupt}
	Suppose that the ordinary generating series generating series $\tilde{\cA}^\omega_\cR(z)$ has radius of convergence $\rho > 0$, and that the series \[E(z, u) = z Z_{\cR^{\kappa}}(u, \tilde{\cA}^\omega_\cR(z^2), \tilde{\cA}^\omega_\cR(z^3), \ldots)\] satisfies $E(\rho + \epsilon, \tilde{\cA}^\omega_\cR(\rho) + \epsilon) < \infty$ for some $\epsilon>0$. Then for any sequence $k_n = o(\sqrt{n})$ of non-negative integers the increasing fringe subtree sequence $\mH^n_{[k_n]}$ of length $k_n$, corresponding to the uniformly at random drawn vertex $v^*$ of $(\cT_n, \beta_n)$, converges in total variation to the fringe subtree sequence of $\hat{\mH}$ with the same length. That is,
	\begin{align}
		\label{le:toshow}
		d_{\textsc{TV}}( (\mH^n_{[k_n]}, v^*), (\hat{\mH}_{[k_n]}, u^*) ) \to 0.
\end{align}
\end{theorem}
Here $u^*$ denotes a random vertex of $\hat{\mH}$ that we defined in Lemma~\ref{le:slow2}. 


\subsection{Scaling limits of metric spaces based on $\cR$-enriched trees}
\label{sec:partA}


\begin{wrapfigure}{r}{0.29\textwidth}
	\begin{center}
		\includegraphics[width=0.24\textwidth]{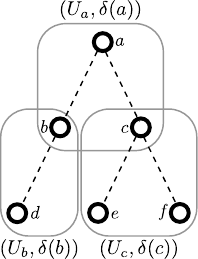}
	\end{center}
	\caption{Patching together discrete metric spaces.}
	\label{fi:patch}
\end{wrapfigure}

\subsubsection{Patching together discrete metric spaces}

We study  metric spaces patched together from metrics associated to the vertices of a tree. Let $A$ be a rooted tree. For each vertex $v \in A$ let $M_v$ denote its offspring set. Let $\delta$ be a map that assigns to each vertex $v$ of $A$ a metric $\delta(v)$ on the set $U_v := M_v \cup \{v\}$. This induces a metric $d$ on the vertex vertex set $V(A)$ that extends the metrics $\delta(v)$ by patching together as illustrated in Figure~\ref{fi:patch}. Formally, we define this  metric as follows. Consider the graph $G$ on the vertex set of $A$ obtained by connecting any two vertices $x \ne y$ if and only if there is some vertex $v$ of the tree $A$ with $x,y \in U_v$ and assigning the weight $\delta(v)(x,y)$ to the edge. The resulting graph is connected and the distance of any two vertices $a$ and $b$ is defined by the minimum of all sums of edge-weights along paths joining $a$ and $b$ in the graph~$G$.

Suppose that for each finite set $U$ and each $\cR$-structure $R \in \cR[U]$ we are given a random metric $\delta_R$ on the set $U \cup \{*_U\}$ with $*_U$ denoting an arbitrary fixed element not contained in $U$. For example, we could set $*_U := \{ U\}$.  Let $\tilde{\mA}_n^\cR = (\tilde{\mA}_n, \alpha_n)$ denote the random $n$-sized $\cR$-enriched tree drawn with probability proportional to its $\omega$-weight. We construct a random $n$-element metric space $\mY_n$ as follows. For each vertex $v$ of $\tilde{\mA}_n^\cR$ with offspring set $M_v$ let $\delta_n(v)$ be the metric on the set $M_v \cup \{v\}$ obtained by taking an independent copy of $\delta_{\alpha_n(v)}$ and identifying $*_{M_v}$ with $v$. Let $d_{\mY_n}$ denote the metric patched together from the family $(\delta_n(v))_v$ as described in the preceding paragraph.

In order for this to be a sensible model of a random tree-like structure we require the following two assumptions. 
\begin{enumerate}
\item We assume that there is a real-valued random variable $\chi\ge0$ such that for any $\cR$-structure $R$ the diameter of the metric $\delta_R$ is stochastically bounded by the sum of $|R|$ independent copies $\chi_1^R, \ldots, \chi_{|R|}^R$ of $\chi$.
\item For any bijection $\sigma: U \to V$ of finite sets and for any $\cR$-structure $R \in \cR[U]$ we require that the metric $\delta_{\cR[\sigma](R)}$ is identically distributed to the push-forward of the metric $\delta_R$ by the bijection $\bar{\sigma}: U \cup \{*_U\} \to V \cup \{*_V\}$ with $\bar{\sigma}|_U = \sigma$. 
\end{enumerate}

Recall that by Lemma~\ref{le:asymptotic} the radius of convergence $\rho$ of $\tilde{\cA}^\omega_\cR(z)$ is finite, and the sum $\tilde{\cA}^\omega_\cR(\rho)$ is finite as well. 

\begin{theorem}[Scaling limits in the unlabelled setting]
\label{te:main2}
Suppose that the ordinary generating series generating series $\tilde{\cA}^\omega_\cR(z)$ has radius of convergence $\rho > 0$ and that the series \[E(z, u) = z Z_{\cR^{\kappa}}(u, \tilde{\cA}^\omega_\cR(z^2), \tilde{\cA}^\omega_\cR(z^3), \ldots)\] satisfies \[E(\rho + \epsilon, \tilde{\cA}^\omega_\cR(\rho) + \epsilon) < \infty\] for some $\epsilon>0$. Then the rescaled space $(\mY_n, n^{-1/2}d_{\mY_n})$ converges weakly to a constant multiple of the (Brownian) continuum random tree $\CRT$ with respect to the Gromov--Hausdorff metric as $n \equiv 1 \mod \spa(\mathbf{w})$ tends to infinity.
\end{theorem}

An explicit expression of the scaling constant in  Theorem~\ref{te:main2} is given in the corresponding proof in Section~\ref{sec:propartA}.  It is interesting to note that  the local weak limit $(\cT^{(\infty)}, \beta^{(\infty)})$ contains some  information on the scaling limit, as it is responsible for one of the factors in the scaling factor. We also provide the following sharp tail-bound for the diameter. 
\begin{theorem}[tail bounds for the diameter in the unlabelled setting]
\label{te:tail2}
Under the same assumptions of Theorem~\ref{te:main2} there are constants $C,c>0$ such that for all $n$ and $x \ge 0$ it holds that
\[
\Pr{\Di(\mY_n) \ge x} \le C (\exp(-c x^2/n) + \exp(-c x)).
\]
\end{theorem}
Again it holds that if $\chi$ is bounded, then we have the tail-bound \[\Pr{\Di(\mY_n) \ge x} \le C \exp(-cx^2/n)\] for some constants $C,c>0$.

The main idea of the  proofs of Theorems~\ref{te:main2} and  Theorem~\ref{te:tail2} is that we may use the random $\Sym(\cR)$-enriched tree $(\cT^{(\ell)}, \beta^{(\ell)})$ of Lemma~\ref{le:sampler2} to relate for any two vertices $x,y \in \cT_n^f$ distances  $d_{\mY_n}(x,y)$ and  $d_{\cT_n^f}(x,y)$ by constant factor. The following basic observation then takes care of the rest.

\begin{lemma}
	\label{le:blobs}
	Suppose that $\rho>0$ and that the function \[E(z, u) = z Z_{\cR^{\kappa}}(u, \tilde{\cA}^\omega_\cR(z^2), \tilde{\cA}^\omega_\cR(z^3), \ldots)\] 
	satisfies $E(\rho + \epsilon,\tilde{\cA}^\omega_\cR(\rho) + \epsilon) < \infty$ for some $\epsilon > 0$.   Then the following assertions hold:
	\begin{enumerate}
		\item There are constants $C,c>0$ such that for all $n$ and $x \ge 0$ \[\Pr{\max_{v \in \cT^f_n}(|f_n(v)| + |F_n(v)|) \ge x} \le C n^{5/2} \exp(-cx).\] 
		\item For any vertex $v$ let $D_v$ denote the $d_{\mY_n}$-diameter of the subspace $\{v\} \cup F_n(v)\subset \mY_n$. Then there are constants $C,c>0$  such that for all $h \ge 0$  \[\Pr{\max_{v \in \cT^f_n} D_v \ge h} \le C n^{5/2} \exp(-ch).\] 
		\item We have that \[
		\frac{\sqrt{(1 + \Ex{\zeta})\Va{\xi}}}{2\sqrt{n}}  \cT^f_n \convdis \CRT\]  in the Gromov--Hausdorff sense as $n \equiv 1 \mod \spa(\mathbf{w})$  tends to infinity. 
	\end{enumerate}
\end{lemma}

A result similar to Lemma~\ref{le:blobs} was used in \cite{2015arXiv150207180P} to provide a combinatorial proof for the scaling limit of uniform random P\'olya trees (with possible degree restrictions). Theorem~\ref{te:main2} is more general, as it applies for example to scaling limits of models of random graphs with respect to the first-passage-percolation metric, and its proof is deeper and more involved, as it requires the interplay with the $\Sym(\cR)$-enriched tree $(\cT^{(\ell)}, \beta^{(\ell)})$ of Lemma~\ref{le:sampler2} that is related to the local weak limit.



\subsection{Applications to random unlabelled rooted connected graphs}
\label{sec:unlablocal}


Let $\cC$ denote the class of connected graphs and $\cB$ its subclass of graphs that are two-connected or a single edge with its ends. Recall that the rooted class $\cC^\bullet$ may be identified with the class of $\Set \circ \cB'$-enriched trees as discussed in Section~\ref{sec:bijblock}. Suppose that we have a weighting $\gamma$ on the class $\cB$, that is, for each $\cB$-graph $B$ we are given a weight $\gamma(B)\ge0$ such that the weights of isomorphic graphs agree. This induces a weighting $\kappa$ on the species $\Set \circ \cB'$ by setting the weight of a set of graphs to the product of the individual weights. Hence we also have a weighting $\omega$ on $\cC$ given by \[\omega(C) = \prod_B \gamma_B\] for all $\cC$-objects $C$, with the index $B$ ranging over all blocks of the graph $C$.  In the following, we study random unlabelled rooted graph $\mC_n^\omega$ drawn from the unlabelled $\cC^\bullet$-objects of size $n$ with probability proportional to its $\omega$-weight. This corresponds to the model of random unlabelled enriched trees $\tilde{\mA}_n^\cR$ in Section \ref{sec:intro1} for the special case $\cR^\kappa = (\Set \circ \cB')^\kappa$.

Under the premise that the cycle index sums related to the random graph $\mC_n^\omega$ satisfy Equation~\eqref{eq:cyceq}, we establish a local weak limit for the vicinity of the fixed root in Theorem~\ref{te:localunl} and  a Benjamini--Schramm limit in Theorem~\ref{te:localunl2}. In both cases we actually establish total variational convergence of arbitrary $o(\sqrt{n})$-neighbourhoods, which is best-possible in this setting.  We also consider the first-passage-percolation metric on the graph $\mC_n^\omega$, which is more general than the graph-metric, and establish sharp exponential tail-bounds for the diameter and a Gromov--Hausdorff scaling limit in Theorem~\ref{te:scaling}. As a byproduct, we also obtain a bound for the size of the largest $2$-connected component.

As an important special case, these Theorems apply to uniform random unlabelled rooted graph from a subcritical block class. This model was studied by Drmota, Fusy, Kang, Kraus and Rué in \cite[Def. 10]{MR2873207}, and includes uniform random rooted unlabelled cacti graphs, outerplanar graphs, and series-parallel graphs \cite[Thm. 15]{MR2873207}. The scaling limit Theorem~\ref{te:scaling} is a  strong result in this context and also establishes the correct order of the diameter of this type random graphs. 



\subsubsection{Local weak limit}

The infinite $\Sym(\Set \circ \cB')$-enriched tree $(\cT^{(\infty)}, \beta^{(\infty)})$ from Section~\ref{sec:localenrunl} is naturally also a $\Set \circ \cB'$-enriched tree, and may hence be interpreted as an infinite locally finite random graph $\hat{\mC}$ according to the bijection in Section~\ref{sec:bijblock}.  Theorem~\ref{te:localunlabelled}  yields local weak convergence of the random graph $\mC_n^\omega$ with respect to neighbourhoods around its fixed root vertex.

\begin{theorem}[Local convergence of random unlabelled graphs]
	\label{te:localunl}
	Suppose that the weighted ordinary generating series $(\tilde{\cC}^{\bullet})^\omega(z)$ has2 radius of convergence $\rho>0$, and that 
	\begin{align}
	\label{eq:cyceq}
	E(z, u) = z Z_{(\Set \circ \cB')^{\kappa}}(u, (\tilde{\cC}^{\bullet})^\omega(z^2), (\tilde{\cC}^{\bullet})^\omega(z^3), \ldots)
	\end{align}
	satisfies $E(\rho + \epsilon, (\tilde{\cC}^{\bullet})^\omega(\rho) + \epsilon) < \infty$ for some $\epsilon>0$. Then for any sequence $k_n = o(\sqrt{n})$ of non-negative integers it holds that
	\begin{align}
	\label{eq:loconv}
	d_{\textsc{TV}}(U_{k_n}(\mC_n^\omega), U_{k_n}(\hat{\mC})) \to 0,
	\end{align}
	and likewise for the graph-metric neighbourhoods $V_{k_n}(\cdot)$. Thus, the infinite random graph $\hat{\mC}$ is the local weak limit of the random graph $\mC_n^\omega$ as $n$ becomes large.
\end{theorem}
Note that this form of convergence is best-possible, as the diameter of $\mC_n^\omega$ has order $\sqrt{n}$ by Theorem~\ref{te:scaling}, and hence \eqref{eq:loconv} fails if the order of $k_n$ is comparable to $\sqrt{n}$. 

\subsubsection{Benjamini--Schramm limit and subgraph count asymptotics}



The infinite $\cG$-enriched tree $\hat{\mH}^\bullet$ from Section~\ref{sec:localenrunl} may be interpreted as an infinite locally finite random graph $\hat{\mC}^\bullet$ according to the bijection in Section~\ref{sec:bijblock}.  Theorem~\ref{te:localhaupt}  yields Benjamini--Schramm convergence of the random graph $\mC_n^\omega$.

\begin{theorem}[Benjamini--Schramm convergence of random unlabelled graphs]
	\label{te:localunl2}
	Suppose that the weighted ordinary generating series $(\tilde{\cC}^{\bullet})^\omega(z)$ has radius of convergence $\rho>0$, and that 
	\[
	E(z, u) = z Z_{(\Set \circ \cB')^{\kappa}}(u, (\tilde{\cC}^{\bullet})^\omega(z^2), (\tilde{\cC}^{\bullet})^\omega(z^3), \ldots)
	\] 
	satisfies $E(\rho + \epsilon, (\tilde{\cC}^{\bullet})^\omega(\rho) + \epsilon) < \infty$ for some $\epsilon>0$. Let $v^*$ be a uniformly at random drawn vertex of the random graph $\mC_n^\omega$. Then for any sequence $k_n = o(\sqrt{n})$ of non-negative integers it holds that
	\[
	d_{\textsc{TV}}(U_{k_n}(\mC_n^\omega, v^*), U_{k_n}(\hat{\mC}^\bullet)) \to 0,
	\]
	and likewise for the graph-metric neighbourhoods $V_{k_n}(\cdot)$. Thus, the infinite random graph $\hat{\mC}^\bullet$ is the Benjamini--Schramm limit of the random graph $\mC_n^\omega$ as $n$ tends to infinity.
\end{theorem}

Again this form of convergence is best-possible, as the diameter of $\mC_n^\omega$ has order $\sqrt{n}$ by Theorem~\ref{te:scaling}. Benjamini--Schramm convergent sequences have many nice properties, for example we may apply general results by Kurauskas \cite[Thm. 2.1]{2015arXiv150408103K}) and Lyons \cite[Thm. 3.2]{MR2160416} to deduce laws of large numbers for subgraph count statistics and spanning tree count statistics.

\subsubsection{Scaling limit and diameter tail-bounds}
\label{sec:app1} 

We apply our results to first-passage percolation on graphs. Let $\iota > 0$ denote a random variable which has finite exponential moments. 
Given a connected graph $G$ we may consider the first-passage percolation metric $d_{\textsc{FPP}}$ on $G$ by assigning an independent copy of $\iota$ to each edge of $G$, letting for any two vertices $x,y$ the distance $d_{\textsc{FPP}}(x,y)$ be given by the minimum of all sums of weights along paths joining $x$ and $y$. We let ${\Di}_{\textsc{FPP}}(\cdot)$ and ${\He}_{\textsc{FPP}}(\cdot)$ denote the diameter and height with respect to the $d_{\textsc{FPP}}$-distance  Theorems~\ref{te:main2} and ~\ref{te:tail2} and the fact, that the diameter and height of the CRT are related by \[\Ex{\Di(\CRT)} = \frac{4}{3}\Ex{\He(\CRT)} = \frac{4}{3} \sqrt{\frac{\pi}{2}},\] readily yield the following result.

\begin{theorem}[First passage percolation random unlabelled rooted  graphs]
	\label{te:scaling}
Suppose that the weighted ordinary generating series has radius of convergence $\rho>0$, and that the series \[E(z, u) = z Z_{(\Set \circ \cB')^{\kappa}}(u, \tilde{\cC}^{\bullet \omega}(z^2), \tilde{\cC}^{\bullet \omega}(z^3), \ldots)\] is finite at the point $(\rho+\epsilon, \tilde{\cC}^{\bullet \omega}(\rho)+\epsilon)$ for some $\epsilon>0$. 
Then there exists a constant $a>0$ such that \[(\mC_n^\omega, a n^{-1/2} d_{\textsc{FPP}}) \convdis (\CRT, d_{\CRT})\] in the Gromov--Hausdorff sense as $n \equiv 1 \mod \spa(\mathbf{w})$ becomes large. Furthermore, there are constants $C,c>0$ with \[\Pr{ \Di(\mC_n^\omega, d_{\textsc{FPP}}) \ge x} \le C \exp(-c x^2 /n)\] for all $n$ and $x \ge 0$. In particular, the rescaled height and diameter converge in the space $\mathbb{L}_p$ for all $p \ge 1$. We have asymptotically
\[
\Ex{{\Di}_{\textsc{FPP}}(\mC_n^\omega)} \sim \frac{4}{3} \Ex{{\He}_{\textsc{FPP}}(\mC_n^\omega)} \sim  \frac{4}{3a}\sqrt{\frac{\pi n}{2}}.
\]
\end{theorem} 


Lemma~\ref{le:blobs} also yields the following result for the size of the largest $2$-connected component of the random graph $\mC_n^\omega$.

\begin{corollary}
	There is a constant $C>0$ such that the largest block in the random graph $\mC_n^\omega$ has size at most $C \log n$ with probability tending to $1$ as $n$ becomes large. Likewise, the maximum degree admits an $O(\log(n))$ bound with high probability.
\end{corollary}

\subsection{Applications to random unlabelled front-rooted $k$-dimensional trees}
\label{sec:apktr}

We consider the species $\cK$ of front-rooted $k$-trees and the subclass $\cK^\circ$ where the root-front is required to lie in a single hedron. Let $\mK_n$ denote a uniform random unlabelled $\cK$-object with $n$ hedra and likewise $\mK_n^\circ$ a uniform random $\cK^\circ$-object with $n$ hedra. 
As discussed in Section~\ref{sec:ktree}, the two species are related by the equations
\[
	\cK \simeq \Set \circ \cK^\circ, \qquad \cK^\circ \simeq \cX \cdot (\Seq_{\{k\}} \circ \Set)(\cK^\circ).
\]
This identifies the random $k$-tree $\mK_n^\circ$ with the random enriched tree $\mA_n^\cR$ for the special case $\cR = \Seq_{\{k\}} \circ \Set$. The random front-rooted $k$-tree $\mK_n$ may be interpreted as a random unordered forest of $\cR$-enriched trees.  We let $\rho$ denote the radius of convergence of $\tilde{\cK}^\circ(z) = \tilde{\cA}_\cR^\omega(z)$.

\subsubsection{Local weak limit}

Theorem~\ref{te:localunlabelled} readily yields a local weak limit of the random graph $\mK_n^\circ$ with respect to neighbourhoods of the root-front, or any fixed vertex of the root-front. The limit object is the infinite random $k$-tree $\hat{\mK}^\circ$ that corresponds to the limit $\Sym(\cR)$-enriched tree $(\cT^{(\infty)}, \beta^{(\infty)})$ according to the bijection in Section~\ref{sec:ktree}.

The random unlabelled front-rooted $k$-tree $\mK_n$ may be viewed as a random unlabelled Gibbs-partition. By  Theorem~\ref{te:gibbs} it follows that $\mK_n$ exhibits a giant component of size $n + O_p(1)$, and the small fragments converge in total variation toward a Boltzmann limit $\mR$ that follows a $\mathbb{P}_{\widetilde{\Set \circ \cK^\circ}, \rho}$ distribution. We let $\hat{\mK}$ denote the infinite random $k$-tree obtained by identifying the root-front of $\hat{\mK}^\circ$ with the root-front of the front-rooted $k$-tree corresponding to $\mR$  according to the bijection in Section~\ref{sec:ktree}. 

\begin{theorem}[Local convergence of random unlabelled front-rooted $k$-trees] \label{te:lokt} For any sequence $k_n = o(\sqrt{n})$ it holds that
\[
	d_{\textsc{TV}}(V_{k_n}(\mK_n), V_{k_n}(\hat{\mK})) \to 0.
\]
Thus the infinite random graph $\hat{\mK}$ is the local weak limit of the random front-rooted $k$-tree $\mK_n$ as $n$ becomes large. Here we may interpret $V_{k_n}(\cdot)$ as the neighbourhood of the root-front, or of any fixed vertex of the root-front. By exchangeability, it does not matter which we choose.
\end{theorem}

Note that Theorem~\ref{te:lokt} does not follow directly from the above discussion, as we still need to relate the height of vertices in the $\cG$-enriched tree representation with the height in the corresponding $k$-tree. We provide a detailed justification in Section~\ref{sec:proapp}.

\subsubsection{Benjamini--Schramm limit}

The infinite $\cG$-enriched tree $\hat{\mH}^\bullet$ from Section~\ref{sec:localenrunl} may be interpreted as an infinite $k$-tree  $\hat{\mK}^\bullet$ according to the bijection in Section~\ref{sec:ktree}.   

As we are going to argue in detail in Section~\ref{sec:proapp}, a random vertex in $\mK_n$ lies with high probability in the largest $\cK^\circ$-component. Theorem~\ref{te:localhaupt} together with a large deviation estimate yield Benjamini--Schramm convergence of the random $k$-tree $\mK_n^\circ$ toward the random graph $\hat{\mK}^\bullet$, and hence also Benjamini--Schramm convergence for the largest $\cK^\circ$-component of $\mK_n$. Moreover, within this component, a random vertex is unlikely to lie anywhere near the root. Thus $\hat{\mK}$ is also the Benjamini--Schramm limit of the random front-rooted $k$-tree $\mK_n$.

\begin{theorem}[Benjamini--Schramm convergence of random unlabelled front-rooted $k$-trees]
	\label{te:lok2t}
	Let $v^*$ denote a uniformly at random selected vertex of the random unlabelled front-rooted  $k$-tree $\mK_n$. Then for any sequence of positive integers $k_n = o(\sqrt{n})$ it holds that
	\[
		d_{\textsc{TV}}(V_{k_n}(\mK_n, v^*), V_{k_n}(\hat{\mK}^\bullet)) \to 0.
	\]
	Thus $\hat{\mK}^\bullet$ is the Benjamini--Schramm limit of $\mK_n$ as $n$ becomes large.
\end{theorem}

\subsubsection{Scaling limit}

By  Theorem~\ref{te:gibbs} we know that $\mK_n$ exhibits a giant $\cK^\circ$-component of size $n + O_p(1)$. Hence in order to establish a scaling limit for $\mK_n$ it suffices to study the random $k$-tree $\mK^\circ_n$ where the root-front lies in a single hedron. 
In order to establish a scaling limit for $\mK^\circ_n$ we may not apply Theorem~\ref{te:main2} directly, as the metric of $\mK^\circ_n$ does not fit  in the general scheme of random metric spaces considered in Section~\ref{sec:partA}. Rather than that, we make direct use of the size-biased $\cG$-enriched tree of Lemma~\ref{le:sampler2} and the results of Lemma~\ref{le:blobs}.

\begin{theorem}
	\label{te:scalktree}
	There is a  constant $a_k>0$ such that 
	\[
	(\mK_n, a_k n^{-1/2} d_{\mK_n}) \convdis (\CRT, d_{\CRT})
	\] in the Gromov--Hausdorff sense as $n$  becomes large. 
\end{theorem}

As a byproduct, we obtain the following properties of the random $k$-tree $\mK_n^\circ$.

\begin{lemma}
	\label{le:calkcirc}
	It holds that 
		\[
	(\mK_n^\circ, a_k n^{-1/2} d_{\mK_n}) \convdis (\CRT, d_{\CRT}).
	\] Moreover, there are constants $C,c>0$ such that for all $x \ge 0 $ and $n \ge 1$
	\[
		\Pr{\Di(\mK_n^\circ) \ge x} \le C \exp(-cx^2/n).
	\]
\end{lemma}

\subsection{Applications to simply generated P\'olya trees}

\label{sec:polya}

Let $(\kappa_i)_{i \in \ndN}$ be a sequence of non-negative weights with $\kappa_0 >0$ and $\kappa_i >0$ for some $i \ge 2$. Hence $\kappa$ can be seen as a weighting on the species $\Set$. Let $d$ denote the greatest common divisor of the set of all indices  $i$ with $\kappa_i>0$. For $n \equiv 1 \mod d$ large enough we may draw a random P\'olya tree $\tau_n$ having $n$ vertices with probability $\Pr{\tau_n = \tau}$ proportional to  $\prod_{v \in \tau} \kappa_{d^+_\tau(v)}$ for any unlabelled unordered tree $\tau$ with size $n$. This corresponds to the random unlabelled enriched tree $\tilde{\cA}_n^\cR$ for $\cR^\kappa = \Set^\kappa$. We let $\rho$ denote the radius of convergence of the corresponding generating series $\tilde{\cA}(z) := \tilde{\cA}_\cR^\omega(z)$.

In the following, we only consider the case where
\begin{align}
	\label{eq:condition}
	\rho >0 \qquad \text{and} \qquad Z_{\Set^\kappa}(\tilde{\cA}(\rho) +\epsilon, \tilde{\cA}((\rho+\epsilon)^2), \ldots) < \infty
\end{align}
for some $\epsilon > 0$.

\subsubsection{Local weak limit}
Let $\hat{\tau}$ denote the tree $\cT^{(\infty)}$ for the special case $\cR^\kappa = \Set^\kappa$.  Theorem~\ref{te:localunlabelled} readily yields the following result.

\begin{theorem}
	\label{te:locpol}
	If Condition \eqref{eq:condition} is satisfied, then for any sequence $k_n = o(\sqrt{n})$ of positive integers it holds that
	\[
		d_{\textsc{TV}}(V_{k_n}(\tau_n), V_{k_n}(\hat{\tau})) \to 0.
	\]
	Thus, $\hat{\tau}$ is the local weak limit of $\tau_n$ as $n$ becomes large.
\end{theorem}

\subsubsection{Benjamini--Schramm limit}
Let $\hat{\tau}^\bullet$ denote the pointed plane tree corresponding to the tree $\hat{\mH}^\bullet$ for the case $\cR^\kappa = \Set^\kappa$. Theorem~\ref{te:localhaupt}  applies directly and yields a Benjamini--Schramm limit for the simply generated P\'olya tree $\tau_n$.

\begin{theorem}
	\label{te:bspol}
	Suppose that Condition \eqref{eq:condition} holds. Let $v^*$ denote a uniformly at random selected vertex of the tree $\tau_n$. Then for any sequence $k_n = o(\sqrt{n})$ of positive integers it holds that
	\[
	d_{\textsc{TV}}(V_{k_n}(\tau_n, v^*), V_{k_n}(\hat{\tau}^\bullet) ) \to 0.
	\]
	In particular,  $\hat{\tau}^\bullet$ is the Benjamini--Schramm limit of $\tau_n$ as $n$ becomes large.
\end{theorem}

\subsubsection{Scaling limit and diameter tail-bound}

Another application of Theorems~\ref{te:main2} and~\ref{te:tail2} is the following scaling limit with a sharp tail-bound for the diameter.

\begin{theorem}[Scaling limits of simply generated P\'olya trees]
	\label{te:scapol}
	If Condition~\eqref{eq:condition} is satisfied, then  there is a constant $a>0$ such that \[(\tau_n, a n^{-1/2} d_{\tau_n}) \convdis (\CRT, d_{\CRT})\] with respect to the Gromov--Hausdorff metric as  becomes large. Moreover, there are constants $c,C>0$ such that for all $x \ge 0$ and $n$ it holds that  \[\Pr{\Di(\tau_n) \ge x} \le C \exp(-c x^2/n).\]
\end{theorem}

This is a mild extension of results for uniformly drawn P\'olya trees with $n$ vertices and vertex degree restrictions, whose scaling limits  were studied by Marckert and Miermont \cite{MR2829313}, Haas and Miermont \cite{MR3050512} and Panagiotou and Stufler \cite{2015arXiv150207180P}. 

\section{Proofs of the main results}
\label{sec:proo}

\subsection{Proof of the local convergence of unlabelled enriched trees in Section~\ref{sec:locunl}}

\begin{proof}[Proof of Lemma~\ref{le:sampler}]
	As discussed in Section~\ref{sec:recur}, the rules of Lemma~\ref{le:pobole} concerning the interplay of Boltzmann distributions and operations on species may be used to construct recursive samplers, if the concerning isomorphism satisfies the conditions of Theorem~\ref{te:implicitspecies}. This is the case for isomorphism
	\[
	\cA_\cR^\omega \simeq \cX \cdot \cR^\kappa(\cA_\cR^\omega),
	\]
	which corresponds to the combinatorial specification $\cY = \cH(\cX, \cY)$ with $\cH(\cX, \cY) =  \cX \cdot \cR^\kappa(\cY)$. Indeed, it holds that $\cH(0,0) = 0 \cdot \cR^\kappa(0) = 0$ and $\partial_2 \cH(0, 0) = 0 \cdot (\cR')^\kappa(0) = 0$. 
	
	Thus we may apply the product rule and substitution rule of Lemma~\ref{le:pobole} to construct a recursive procedure that samples according to the $\mathbb{P}_{\Sym(\cA_\cR^\omega),(x^i)_i}$-Boltzmann distribution. The result is the procedure described in Lemma~\ref{le:sampler}, with one important difference. According to Lemma~\ref{le:pobole} we would have to apply Lemma~\ref{le:symconstr} for each recursive call to construct an $\cA_\cR$-symmetry out of the $\cR$-symmetry and the attached $\cA_\cR$-symmetries, and then relabel uniformly at random. Instead of doing this for each recursive call, we skip this step and keep track of all the $\cR$-symmetries, yielding the $\Sym(\cR)$-enriched tree $(\cT, \beta)$. As discussed at the beginning of the Section~\ref{sec:unlcoupling},  we may construct the symmetry that corresponds to $(\cT, \beta)$ in one step by applying Lemma~\ref{le:symconstr}  for each of its vertices, starting with the leaves and working our way upwards. Thus $(\cT, \beta)$ corresponds to a symmetry on the vertex set of the plane tree $\cT$. Moreover, instead of relabelling uniformly at random after each application of Lemma~\ref{le:symconstr}, we may postpone this step and just relabel the symmetry corresponding to $(\cT, \beta)$ once uniformly at random. The result $\Gamma Z_{\cA_\cR^\omega}(x)$ then follows $\mathbb{P}_{\Sym(\cA_\cR^\omega),(x^i)_i}$-Boltzmann distribution, meaning that for each $n \ge 0$ and each symmetry $(A, \sigma) \in \Sym(\cA_\cR)[n]$ it holds that
	\[
	\Pr{\Gamma Z_{\cA_\cR^\omega}(x) = (A, \sigma)} = \omega(A) \frac{x^{n}}{n!} Z_{\cA_\cR^\omega}(x,x^2, \ldots)^{-1}.
	\]
	By Equation~\eqref{eq:relcyc} we know that 
	\[
	Z_{\cA_\cR^\omega}(x,x^2, \ldots) = \tilde{\cA}_\cR^\omega(x).
	\]
	This completes the proof.
\end{proof}

\begin{proof}[Proof of Lemma~\ref{le:ucoup}]
	As for claim {\em (1)}, note that the isomorphism $\tilde{\cA}^\omega_\cR \simeq \cX \cdot \cR^\kappa(\tilde{\cA}^\omega_\cR)$ implies that \[\tilde{\cA}^\omega_\cR(z) = z Z_{\cR^\kappa}(\tilde{\cA}^\omega_\cR(z), \tilde{\cA}^\omega_\cR(z^2), \ldots).\] By assumption, there is an $\cR$-structure with size zero and one with size at least two such that both have a positive $\kappa$-weight. It follows that there are constants $a,b > 0$  and $k \ge 2$ with the property that for all $0 \le x < \rho$ it holds that \[\tilde{\cA}^\omega_\cR(x) \ge x(a + b\tilde{\cA}^\omega_\cR(x)^k).\] This implies that $\lim_{x \uparrow \rho}\tilde{\cA}^\omega_\cR(x) < \infty$ and hence, by non-negativity of coefficients, $\tilde{\cA}^\omega_\cR(\rho) < \infty$.

	Claims {\em (2)} - {\em (6)} all follow after a moments consideration from the explicit description of the sampler in Lemma~\ref{le:sampler}.
\end{proof}

\begin{proof}[Proof of Lemma~\ref{le:asymptotic}]
	Claims {\em (1)} and {\em (2)} follow from a general enumeration theorem by Bell, Burris and Yeats \cite[Thm. 28]{MR2240769} which  implies the asymptotic behaviour of the coefficients of the power series $\tilde{\cA}^\omega_\cR(z)$. 

	Claim { \em (3)}: The expressions for the moments of $\xi$ and $\zeta$ follow from the equations describing the corresponding probability generating functions in Lemma~\ref{le:ucoup}. In order to verify that $\Ex{\xi}=1$, note that by Pringsheim's theorem the function $\tilde{\cA}^\omega_\cR(z)$ cannot be analytically continued in a neighbourhood of $\rho$, and hence by the implicit function theorem it must hold that the function \[H(z,u) := u - E(z,u)\] satisfies \[H_u(\rho, \tilde{\cA}^\omega_\cR(\rho))=0.\] In other words, $\Ex{\xi} = E_u(\rho, \tilde{\cA}^\omega_\cR(\rho)) = 1$. (Compare with the proof of \cite[Cor. 12]{MR2240769}.)
	
	Claim {\em (4)}: Let $\Lambda$ denote the lattice spanned by all differences $\mathbf{x} - \mathbf{y}$ of vectors that $(\xi, \zeta)$ assumes with positive probability. We assumed that  there is an $\cR$-structure with size zero and positive $\kappa$-weight. Its automorphism is the empty map with no cycles at all. Hence $\Pr{(\xi, \zeta)=(0,0)}>0$, and $\Lambda$ is actually generated by the support of $(\xi, \zeta)$. 
	
	We assumed further that at least one $\cR$-structure $R$ with positive $\kappa$-weight has a non-trivial automorphism group. Hence there are positive numbers $a^*, b^*$ and $c^*$ such that $(\xi, \zeta)$ assumes $(a^*,0)$ and $(b^*,c^*)$ with positive probability, as these points correspond the trivial and a non-trivial automorphism of $R$. Hence the lattice $\Lambda$ contains three points that do not lie on any straight line. Consequently, it has rank $2$, that is, $\Lambda \simeq \ndZ^2$ as abelian group.
	
	Let $\mathbf{B} \in \ndZ^{2 \times 2}$ be a basis of the lattice $\Lambda$. As the support of $(\xi, \zeta)$ is not contained on a straight line, it follows that the covariance matrix $\Sigma$ is positive-definite. Let $(\xi_i, \zeta_i)_{i \in \ndN}$ denote a family of independent copies of $(\xi,\zeta)$. By Lemma~\ref{le:ucoup} it holds for any positive integer $n$ and $\ell$ that
	\[
	\Pr{|\cT| = n, |\cT^f| = \ell} = \Pr{\sum_{i=1}^\ell (\xi_i, \zeta_i) = (\ell -1, n - \ell), \sum_{i=1}^m \xi_i \ge m \text{ for all $m < \ell$}}.
	\]
	Using rotational symmetry and the cycle lemma \ref{le:cycle}, it follows that
	\begin{align}
	\label{eq:the}
	\Pr{|\cT| = n, |\cT^f| = \ell} = \frac{1}{\ell} \Pr{\sum_{i=1}^\ell (\xi_i, \zeta_i) = (\ell -1, n - \ell)}.
	\end{align}
	We know by Claim {\em (1)} that if $n$ is larger than some fixed constant, then $\Pr{|\cT|=n}>0$ if and only if $n-1$ is divisible by $\spa(\mathbf{w})$. We shall check below that for each such $n$ 
	\begin{align}
		\label{eq:madness}
			 \{ \ell \in \ndZ \mid (\ell-1, n - \ell) \in \Lambda\} = n + \ndZ d
	\end{align}
	with $d = |\det \mathbf{B}| / \spa(\mathbf{w})$. For now, let us assume that \eqref{eq:madness} holds. Let $M>0$ be a fixed constant. Then it holds uniformly for all $\ell = (1 + \Ex{\zeta}) n + x \sqrt{n}$ with $|x| \le M$ that
	\[
		\ell \sim n / \sqrt{1 + \Ex{\zeta}}
	\]
	and 
	\[
		\ell^{-1/2} \left((\ell -1, n-\ell) - \ell \Ex{(\xi, \zeta)} \right) \sim (0, -x (1 + \Ex{\zeta})^{3/2}).
	\]
	The central local limit theorem given in Lemma~\ref{le:llt2dim} yields that
	\begin{align}
		\label{eq:lltapp}
		\frac{1}{\ell} \Pr{\sum_{i=1}^\ell (\xi_i, \zeta_i) = (\ell -1, n - \ell)} \sim \frac{(1 + \Ex{\zeta})|\det \mathbf{B}| }{n^2\sqrt{2 \pi \det \mathbf{\Sigma}}} \exp(- \frac{x^2}{2 \sigma^2}).
	\end{align}
	with $\sigma^2 = \det \mathbf{\Sigma} / (\Va{\zeta}(1 + \Ex{\zeta})^3)$. By Claim~{\em (3)} we know that
	\[
		\Pr{|\cT|=n} \sim \spa(\mathbf{w}) n^{-3/2}  \sqrt{\frac{1 + \Ex{\zeta}}{2\pi \Va{\xi}}}.
	\]
	Using Equation~\eqref{eq:the} it follows that
	\[
		\Pr{|\cT_n^f| = \ell} \sim \frac{d }{\sigma \sqrt{2 \pi n}} \exp(- \frac{x^2}{2 \sigma^2}).
	\]
	The central limit theorem now follows from Equation~\eqref{eq:madness}, as for any fixed $a<b$
	\[
	\Pr{a \le \frac{|\cT_n^f| - n/(1 + \Ex{\zeta})}{\sqrt{n}} \le b} \sim \frac{d}{\sqrt{n}}\sum_{x}  \frac{1}{\sigma \sqrt{2 \pi}} \exp(- \frac{x^2}{2 \sigma^2}) \sim \int_a^b \frac{1}{\sigma \sqrt{2 \pi}} \exp(- \frac{x^2}{2 \sigma^2}) \,\text{d}x,
	\]
	with the sum index $x$ ranging over $[a,b] \cap ((1- \mu)\sqrt{n} + \ndZ d/\sqrt{n})$.
	
	It remains to verify Equation~\eqref{eq:madness}, which requires careful reasoning, as we have to relate $\spa(\mathbf{w})$ with the involved lattice. In order to simplify our calculations, we would like to pick a "nice" basis $\mathbf{B}$ of $\Lambda$. Note that it does not matter for \eqref{eq:madness} which basis of $\Lambda$ we choose, as for any two bases $\mathbf{B}_1$ and $\mathbf{B}_2$ there is a matrix $\mathbf{M} \in \text{GL}_2(\ndZ)$ with $\mathbf{B}_1 = \mathbf{M} \mathbf{B}_2$, and as $\det \mathbf{M} \in \{1, -1\}$ it follows that $|\det\mathbf{B}_1| = |\det \mathbf{B}_2|$. The inconvenient part is that, contrary to vector spaces, not every linear independent subset of a lattice may be extended to a basis.
	However, a classical algebraic result states that for any free $\ndZ$-module $M$ with rank $r(M)$ and for any submodule $N \subset M$ with rank $r(N)$ there is a $\ndZ$-basis $v_1, \ldots, v_{r(M)}$ of $M$ and integers $\lambda_1, \ldots, \lambda_{r(N)}$ such that $\lambda_1 v_1, \ldots, \lambda_k v_k$ is a basis of $N$. See for example Roman's book \cite[Thm. 6.7]{MR1169104}, which states this in the more general context of modules over principal ideal domains. If $\mathbf{B} = (\mathbf{b}_1, \mathbf{b_2})$ is such a basis of $\Lambda$ for the submodule $N:= \ndZ (a^*,0)^\intercal \subset \Lambda$ (recall that we defined $a^*$ at the beginning of the proof, when we showed that $\Lambda$ has rank $2$), then there is an integer $\lambda$ with $\lambda \mathbf{b}_1 \in \ndZ (a^*,0)^\intercal$. But this implies that the second coordinate of $\mathbf{b}_1$ must be zero, and hence
	\begin{align}
		\label{eq:basis}
		\mathbf{B} = \left ( \begin{matrix}
		a &  b \\
		0 &  c 
		\end{matrix} \right )
	\end{align}
	with $a,b,c>0$ is an upper-triangular matrix.  For any $\ell, n$ it holds that
	$(\ell-1, (n-1) - (\ell -1))^\intercal \in \Lambda$ if and only if $(\ell-1, n-1)^\intercal \in \Lambda' := \mathbf{C} \ndZ^2$ with
	\[
	\mathbf{C} = \left ( \begin{matrix}
	1 &  0 \\
	-1 &  1 
	\end{matrix} \right )^{-1} \left ( \begin{matrix}
	a &  b \\
	0 &  c 
	\end{matrix} \right ) = \left ( \begin{matrix}
	a &  b \\
	a &  b+ c 
	\end{matrix} \right ).
	\]
	We may easily calculate that 
	\[
	(\ell-1, n-1) \in \mathbf{C} \ndZ^2 \quad \text{if and only if} \quad n-1  \in \ndZ \gcd(a, b+c), \,\,  \ell \in n + \ndZ ac / \gcd(a,b+c).
	\]
	Indeed, it necessarily holds that $n-1 \in \ndZ a + \ndZ(b+c) = \ndZ \gcd(a, b+c)$. Conversely, if $n-1 = \lambda_0 a + \mu_0 b$, then any pair $(\lambda, \mu) \in \ndZ^2$ satisfies $\lambda a + \mu b = n-1$ precisely if there is an arbitrary integer $t$ with $\lambda - \lambda_0 = t(b+c)/\gcd(a,b+c)$ and $\mu - \mu_0 = ta/\gcd(a,b+c)$. Hence $(\ell-1, n-1)^\intercal = \mathbf{C} (\lambda, \mu)^\intercal$ reduces to $\ell = n + act/\gcd(a,b+c)$ with $t \in \ndZ$.

	Since $\det(\mathbf{B}) = ac$, it remains to check that $\spa(\mathbf{w}) = \gcd(a,b+c)$. That is, we show that $\spa(\mathbf{w})$ is the span of the projection of $\Lambda'$ on the $y$-axis. The support $S$ of $(|\cT^f| -1, |\cT| -1)$ is a subset of the lattice $\Lambda'$. Hence $\spa(\mathbf{w})$, which is the span of the projection of $S$ to the $y$-axis, is a multiple of $\gcd(a,b+c)$. In order to show equality we need to verify that indeed two successive multiples of $\gcd(a,b+c)$ exist that $|\cT|-1$ assumes with positive probability. We may do this probabilistically. Let $C>0$ be a given constant. By Equation~\ref{eq:the} and the central local limit theorem given in Lemma~\ref{le:llt2dim} it follows that there is a positive integer $L$ such that for all $\ell \ge L$ and all $n$ with $(l-1, n-1)^\intercal \in \Lambda'$ and
	\begin{align}
		\label{eq:qwer11}
		\ell^{-1/2} \left \| ((\ell -1, n-\ell) - \ell \Ex{(\xi, \zeta)} )\right\| \le C
	\end{align}
	it holds that $\Pr{|\cT|=n, |\cT^f|=\ell} >0$. There is a constant $c>0$ such that for every $n$ at least $c \sqrt{n}$ consecutive integers $\ell \ge L$ exist that satisfy Inequality~\eqref{eq:qwer11}. If $n-1$ is additionally a constant multiple of $\gcd(a,b+c)$, then the intersection $\Gamma_n$ of the lattice $\Lambda'$ and the affine space $(0,n-1)^\intercal + \ndZ(1,0)$ is non-empty, and hence an affine subspace of rank one, whose span does not depend on $n$. In particular, $\Gamma_n$ hits the cone \eqref{eq:qwer11} about a constant fraction of $\sqrt{n}$ many times. This proves that $\spa(\mathbf{w}) = \gcd(a,b+c)$ and concludes the proof of Claim {\em (4)}.

\end{proof}

\begin{proof}[Proof of Lemma~\ref{le:sampler2}]
	The task is to provide a sampler for the species $\cA_\cR^{(\ell)}$ using the isomorphism
	\[
	\cA_\cR^{(\ell)} \simeq (\cX \cdot (\cR')^\kappa(\cA_\cR^\omega))^\ell \cA_\cR^\omega.
	\]
	Note that for $\ell =0$ we have $\cA_\cR^{(0)} = \cA_\cR^\omega$ and for $\ell \ge 1$ it holds that
	\[
		\cA_\cR^{(\ell)} \simeq \cX \cdot (\cR')^\kappa(\cA_\cR^\omega) \cdot \cA_\cR^{(\ell -1)}.
	\]
	Thus for $\ell \ge 1$ we may apply the product rule and substitution rule of Lemma~\ref{le:pobole} to construct a  procedure that samples according to the $\mathbb{P}_{\Sym(\cA_\cR^{(\ell)}),(x^i)_i}$-Boltzmann distribution, employing a  $\mathbb{P}_{\Sym(\cA_\cR^{(\ell-1)}),(x^i)_i}$-distributed symmetry, which we may sample by a recursive call to the sampler for $\ell -1$. 
	
	The result is essentially the procedure described in Lemma~\ref{le:sampler2}, but we make some modifications. According to Lemma~\ref{le:pobole} we would have to apply Lemma~\ref{le:symconstr} to construct an $\cA_\cR$-symmetry out of the $\cR'$-symmetry, the attached $\cA_\cR$-symmetries, and the $\cA_\cR^{(\ell -1)}$-symmetry, and relabel uniformly at random afterwards.	
	Instead of doing this for each  call, we skip this step and keep track of all the $\cR'$-symmetries. Also, instead of taking $\cA_\cR$-symmetries directly, we use the $\Sym(\cR)$-enriched tree $(\cT, \beta)$ from Lemma~\ref{le:sampler2}. This yields the $\Sym(\cR)$-enriched tree $(\cT^{(\ell)}, \beta^{(\ell)})$.
	
	As discussed at the beginning of the Section~\ref{sec:unlcoupling},  we may construct the symmetry that corresponds to $(\cT^{(\ell)}, \beta^{(\ell)})$ in one step by applying Lemma~\ref{le:symconstr}  for each of its vertices, starting with the leaves and working our way upwards. If we additionally relabel uniformly at random, the resulting symmetry $\Gamma Z_{\cA_\cR^{(\ell)}}$ follows a $\mathbb{P}_{\Sym(\cA_\cR^{(\ell)}, (x^i)_i}$-distribution. 
	
	We defined $\cA_\cR^{(\ell)}$ as a subspecies of the species of pointed $\cR$-enriched trees. In particular, for all $n \ge 0$ and any symmetry $((A,r), \sigma) \in \Sym(\cA_\cR^{(\ell)})$ we have that $(A, \sigma)$ is an $\cA_\cR^\omega$-symmetry and $r$ is a fixpoint of the automorphism $\sigma$. It follows  from the definition of the Boltzmann distributions and
	\[
		\tilde{\cA}_\cR^{(\ell)}(x) = (x\widetilde{(\cR')^\kappa \circ \cA_\cR^\omega}(x))^{\ell} \tilde{\cA}_\cR^\omega(x)
	\]
	that
	\[
		\Pr{\Gamma Z_{\cA_\cR^{(\ell)}}(x) = ((A,r), \sigma)} = \omega(A) \frac{x^{|A|}}{|A|!} \tilde{\cA}_\cR^{(\ell)}(x)^{-1} = (x \widetilde{(\cR')^\kappa \circ \cA_\cR^\omega}(x))^{-\ell} \Pr{\Gamma Z_{\cA_\cR^\omega}(x) = (A, \sigma)}.
	\]
\end{proof}

\begin{proof}[Proof of Lemma~\ref{le:slow}]
	Claim {\em (1)}: By Lemma~\ref{le:sampler2}, the $\cR$-symmetry $(\hat{\mR}, \hat{\sigma})$ of $\hat{\mG}$ follows, up to relabeling, a weighted P\'olya-Boltzmann distribution for the species $(\cR')^\kappa$ with parameter $\tilde{\cA}_\cR^\omega(\rho)$. It makes no difference whether we distinguish the vertex corresponding to the $*$-fixpoint, or a uniformly at random drawn fixpoint, as the results are identically distributed. By the discussion in Section~\ref{sec:opspe} it holds that
	\[
		Z_{(\cR')^\kappa}(s_1,s_2, \ldots) = \frac{\partial}{\partial s_1} Z_{\cR^\kappa}(s_1,s_2, \ldots).
	\]
	Hence, for any $\cR$-symmetry $(R, \sigma)$ with a marked fixpoint $u$, the probability that $(\hat{\mR}, \hat{\sigma}) = (R, \sigma)$ and that precisely the fixpoint $u$ gets marked, is simply given by the probability that a P\'olya-Boltzmann distributed $\cR$-symmetry with parameter $\tilde{\cA}_\cR^\omega(\rho)$ equals $(R, \sigma)$. Consequently, for any $\cG$-object $G$ with a marked fixpoint $u$ the probability that $\hat{\mG}$ assumes $G$ and that precisely the vertex $u$ is marked is  given by $\Pr{\mG = G}$.
	
	Claim {\em (2)}: This is a reformulation of Lemma~\ref{le:sampler2} in terms of $\cG$-enriched trees.

	Claim {\em (3)}: The event $(\cT^{(\infty)},\beta^{(\infty)})^{<k>} = (\tau, \gamma)^{<k>}$ corresponds to precisely $L_\tau(k)$ different outcomes of the first $k-1$ levels, depending on which leaf of $\tau^{[k]}$ the spine of the fixpoint tree $(\cT^{(\infty)})^f$ is supposed to pass through. By Claim~{\em (1)}, each of these events is equally likely with probability $\prod_{i=1}^t \Pr{\mG = G_i}$.

	Claim {\em (4)}:  The  $k+1$-th level of $(\cT^{(\infty)},\beta^{(\infty)})$ interpreted as $\cG$-enriched tree is obtained by taking for each fixpoint of the $k$-th level an independent copy of $\mG$, except for the unique distinguished fixpoint, which receives an independent copy of $\hat{\mG}$. This yields
	\begin{align*}
	\Ex{L_k(\tau)} &= (\Ex{L_{k-1}(\tau)} - 1)\Ex{\xi} + \Ex{\hat{\xi}} = \ldots =
	k(\Ex{\hat{\xi}} -1)+1, \\
	\Ex{|\tau^{[k]}|} &= \Ex{L_0(\tau) + \ldots + L_k(\tau)} = k(k+1)(\Ex{\hat{\xi}}-1)/2+ k+ 1, \\
	\Ex{L_k^\cG(\tau)} &= (\Ex{L_k} -1) \Ex{\zeta} + \Ex{\hat{\zeta}} = \ldots =
	k(\Ex{\hat{\xi}}-1)\Ex{\zeta} + \Ex{\hat{\zeta}},\\
	\Ex{H_k^\cG(\tau) -  |\tau^{[k]}| \Ex{\zeta}} &= \sum_{i=0}^{k} (\Ex{L_i^\cG(\tau)} - \Ex{L_i(\tau)}\Ex{\zeta}) =  (k+1)(\Ex{\hat{\zeta}} - \Ex{\zeta}).
	\end{align*}
	For the variance, note that
	\[
	\Va{H_k^\cG(\tau) -  |\tau^{[k]}| \Ex{\zeta}} = \Ex{(D_0 + \ldots D_k)^2)}
	\]
	with
	\[
		D_i := L_i^\cG(\tau) - L_i(\tau)\Ex{\zeta} - \Ex{\hat{\zeta}} + \Ex{\zeta}
	\]
	satisfying $\Ex{D_i \mid L_i(\tau) = \ell} = 0$ for all $\ell$. In particular, for all $i<j$, 
	\[
		\Ex{D_i D_j \mid (\tau^{[k]}, (\gamma(v))_{v \in \tau^{[k-1]}})} = D_i \Ex{ D_j \mid L_j(\tau)} = 0.
	\]
	It also holds that
	\[
		\Ex{D_i^2 \mid L_i(\tau)} = L_i(\tau) \Va{\zeta} + \Va{\hat{\zeta}},
	\]
	and consequently
	\begin{align*}
	\Va{H_k^\cG(\tau) -  |\tau^{[k]}| \Ex{\zeta}} &= \sum_{i=0}^k \Ex{D_i^2} = \sum_{i=0}^k (\Ex{L_i(\tau)}\Va{\zeta} + \Va{\hat{\zeta}}) \\
	&=  k(k+1)(\Ex{\hat{\xi}}-1)\Va{\zeta}/2+ (k+ 1)(\Va{\hat{\zeta}} + \Va{\zeta}). 
	\end{align*}
\end{proof}

\begin{proof}[Proof of Theorem~\ref{te:localunlabelled}]
	Let $\cE$ denote the countably infinite set of all $\cG$-enriched plane trees and set
	\[
		\cE_k = \{A^{<k>} \mid A \in \cE\}.
	\]
	We have to show that
	\begin{align}
		\label{eq:toshow}
		\lim_{n \to \infty} \sup_{\cH \subset \cE_{k_n}} | \Pr{(\cT_n, \beta_n)^{<k_n>} \in \cH} - \Pr{(\cT^{(\infty)}, \beta^{(\infty)})^{<k_n>} \in \cH}| = 0.
	\end{align}
	Recall that for any $\cG$-enriched tree $(\tau, (G_\tau(v))_{v })$ with $G_\tau(v) = (\beta_\tau(v), f_\tau(v), F_\tau(v))$ and any integer $k \ge 0$ we set
	\[
		L_k(\tau) = |\{v \in \tau \mid \he_\tau(v)  = k\}|, \quad L_k^\cG(\tau) = \sum_{\substack{v \in \tau\\ \he_\tau(v) = k}} |F_\tau(v)|, \quad \text{and} \quad H_k^\cG(\tau) = \sum_{i=0}^k L_i^\cG(\tau).
	\]
	By assumption, there is a sequence $t_n \to 0$ with $k_n = \sqrt{n} t_n$. We may without loss of generality assume that $k_n$ tends to infinity.	For any $C>0$ and all $k$ and $n$ we define with foresight the set $\cE_{C,k,n}$  of all $(\tau, \gamma) \in \cE_k$ satisfying $\Pr{(\cT^{(\infty)}, \beta^{(\infty)})^{<k>} = (\tau, \gamma)} > 0$ and 
	\begin{align*}
	 1 \le L_k (\tau) \le  C (nt_n)^{1/2}, \quad 
	| H_{k-1}^\cG(\tau) - |\tau^{[k-1]}| \Ex{\zeta} | \le C (n t_n)^{1/2}, \quad |\tau^{[k-1]}| \le nt_n
	.
	\end{align*}
	Using Markov's and Chebyshev's inequalities and the expressions of the moments in Lemma~\ref{le:slow}, it follows that there is a constant $C$ such that 
	\begin{align}
	\label{eq:mark}
	\lim_{n \to \infty} \Pr{(\cT^{(\infty)}, \beta^{(\infty)})^{<k_n>} \in \cE_{C, k_n, n}} = 1.
	\end{align}
	Hence, if we verify \eqref{eq:toshow} with the index $\cH$ ranging only over all subsets of $\cE_{C, k_n, n}$, then \eqref{eq:mark} implies that
	\begin{align}
	\lim_{n \to \infty} \Pr{(\cT_n, \beta_n)^{<k_n>} \notin \cE_{C, k_n,n}} = 1
	\end{align}
	and consequently \eqref{eq:toshow} already holds with the index ranging over all subsets of $\cE_{k_n}$. So the only thing that is left to show is that \eqref{eq:toshow} holds for $\cH \subset \cE_{C, k_n, n}$. That is, we have to verify that for any $\epsilon >0 $ it holds for large enough $n$ that
	\begin{align}
		\label{eq:nextstep}
			\sup_{\cH \subset \cE_{C,k_n,n}} | \Pr{(\cT_n, \beta_n)^{<k_n>} \in \cH} - \Pr{(\cT^{(\infty)}, \beta^{(\infty)})^{<k_n>} \in \cH}| \le \epsilon.
	\end{align}
	Let $\mathbf{B}, \mathbf{\Sigma}, \mu, \sigma$ and $d = |\det \mathbf{B}| / \spa(\mathbf{w})$ be as in Lemma~\ref{le:asymptotic}. We have shown in this Lemma that, as $n \equiv 1 \mod \spa(\mathbf{w})$ tends to infinity, 
	\begin{align}
		\sqrt{n} \Pr{|\cT_n^f| = \ell} \sim  \frac{d}{\sigma \sqrt{2 \pi }} \exp(- \frac{x_\ell^2}{2 \sigma^2})
	\end{align}
	uniformly for all bounded $x_\ell$ with
	\[
		\ell := \mu n + x_\ell \sqrt{n} \in n + d \ndZ.
	\]
	Moreover,
	\[
		\frac{|\cT_n^f| - n \mu}{\sqrt{n}} \convdis \cN(0, \sigma^2)
	\]
	and for any $\epsilon_1>0$ there is a constant $M>0$ such that for all $n$
	\[
		\Pr{|\cT_n^f| \notin I_n} \le \epsilon_1 \quad \text{with} \quad I_n := (n + d\ndZ) \cap [n / (1 + \Ex{\zeta}) - M \sqrt{n}, n / (1 + \Ex{\zeta}) + M \sqrt{n}].
	\]
	Hence the expression in \eqref{eq:nextstep} may be bounded by
	\begin{align}
		\label{eq:bound}
		\epsilon_1 + \sup_{(\tau, \gamma) \in \cE_{C, k_n, n}} \left |\frac{\Pr{(\cT_n, \beta_n)^{<k_n>} = (\tau, \gamma), |\cT_n^f| \in I_n}}{\Pr{(\cT^{(\infty)}, \beta^{(\infty)})^{<k_n>} = (\tau, \gamma)}} - 1 \right|.
	\end{align}
	Let $(\tau, \gamma) \in \cE_{C, k_n, n}$ and let $G_i=(R_i, f_i, F_i)$, $1 \le i \le t$ be the depth-first-search ordered list of its $\cG$-objects. We set $(h,H) := \sum_{i=1}^t(|f_i|, |F_i|)$. Moreover, let $\mG_i = (\mathsf{R}_i, \mathsf{f}_i, \mathsf{F}_i)$ denote a family of independent copies of the random $\cG$-enriched tree $\mG$. For each $\cG$-enriched plane tree $G$ set $\pi_G = \Pr{\mG = G}$.  By Lemma~\ref{le:ucoup} it follows that for all $\ell$ the probability $\Pr{(\cT, \beta)^{<k_n>} = (\tau, \gamma), |\cT^f| = \ell}$ is given by 
	\begin{align*}
	\pi_{G_1} \cdots \pi_{G_t} \Pr{\sum_{i=t+1}^{\ell} (|\mathsf{f}_i|, |\mathsf{F_i}|) = (\ell-1-h, n - \ell - H), \sum_{i=t+1}^m \mathsf{f}_i \ge m - h \text{ for all $1 \le m < \ell$}  }.
	\end{align*}
	Using rotational symmetry and the Cycle Lemma \ref{le:cycle}, this may be simplified further to
	\[
		\pi_{G_1} \cdots \pi_{G_t} \frac{h-t+1}{\ell-t}\Pr{\sum_{i=t+1}^{\ell} (|\mathsf{f}_i|, |\mathsf{F_i}|) = (\ell-1-h, n - \ell - H) }.
	\]
	 The tree $\tau$ has precisely $h-t+1$ many leaves with height $k$ and thus Lemma~\ref{le:slow} implies that
	 \[
	 \Pr{(\cT^{(\infty)}, \beta^{(\infty)})^{<k_n>} = (\tau, \gamma)} = \pi_{G_1} \cdots \pi_{G_t} (h-t+1).
	 \]	 
	It holds uniformly for $(\tau, \gamma) \in \cE_{C, k_n, n}$ and  $\ell = \mu n + x_\ell \sqrt{n} \in I_n$ that 
	\[
		\ell^{-1/2}\left | (\ell -1 - h, n- \ell - H) - (\ell -t)\Ex{(\xi, \zeta)} \right | \sim (0, -x_\ell (1 + \Ex{\zeta})^{3/2})
	\]
	as $n$ becomes large. Hence, as $\ell -t \sim \ell$, we may apply the bivariate Central Local Limit Theorem~\ref{le:llt2dim} to obtain
	\[
	\frac{1}{\ell-t}\Pr{\sum_{i=t+1}^{\ell} (|\mathsf{f}_i|, |\mathsf{F_i}|) = (\ell-1-h, n - \ell - H) } \sim \frac{(1 + \Ex{\zeta})|\det \mathbf{B}| }{n^2\sqrt{2 \pi \det \mathbf{\Sigma}}} \exp(- \frac{x_\ell^2}{2 \sigma^2}).
	\] 
	By Lemma~\ref{le:asymptotic} we know that
	\[
	\Pr{|\cT|=n} \sim \spa(\mathbf{w}) n^{-3/2}  \sqrt{\frac{1 + \Ex{\zeta}}{2\pi \Va{\xi}}}.
	\]
	It follows that uniformly for all $(\tau, \gamma) \in \cE_{C, k_n, n}$ 
	\[
		\frac{\Pr{(\cT_n, \beta_n)^{<k_n>} = (\tau, \gamma), |\cT_n^f| \in I_n}}{\Pr{(\cT^{(\infty)}, \beta^{(\infty)})^{<k_n>} = (\tau, \gamma)}} \sim \frac{d}{\sqrt{n}}\sum_{\ell \in I_n}  \frac{1}{\sigma \sqrt{2 \pi}} \exp(- \frac{x_\ell^2}{2 \sigma^2})\sim \int_{-M}^M \frac{1}{\sigma \sqrt{2 \pi}} \exp(- \frac{x^2}{2 \sigma^2}) \,\text{d}x.
	\]
	Taking $\epsilon_1 = \epsilon/2$ and $M$ sufficiently large it follows that the bound \eqref{eq:bound} is smaller than $\epsilon$ for sufficiently large $n$. This completes the proof.
\end{proof}

\begin{proof}[Proof of Lemma~\ref{le:slow2}]
	Claims {\em (1)} and {\em (2)} are straight-forward. For Claim {\em (3)}, note that the event $\hat{\mH}_{[k]} = \mathbf{H}_{[k]}$ means that $\hat{\mH}_k = H_k$, and that $u_0$ lies at one of the $p(\mathbf{H})$ many locations, and that $u^*$ lies in precisely the location in $\{u_0\} \cup F(u_0)$ that corresponds to $u$, with $F(u_0)$ denoting the forest of non-fixpoints of the $\cG$-object corresponding to $u_0$. Let $G_1, \ldots, G_t$ denote the depth-first-search ordered list of the $\cG$-objects $G_i=(S_i, f_i, F_i)$ of $H_k$ and let $G_{\ell_1}, \ldots, G_{\ell_2}$ be the segment that corresponds to $H_0$. For any $\cG$-object $G=(S,f,F)$ with a marked fixpoint $v$ from $f$, the probability for $\hat{\mG}=(\hat{\mS}, \hat{f}, \hat{F})$ to assume $G$ and that a uniformly at random drawn fixpoint of $\hat{f}$ equals $v$ is given by
	\[
		\Pr{\hat{\mG} = G} |f|^{-1} = \Pr{\mG = G}.
	\]
	Likewise, if we distinguish a vertex $v$ from $\{*\} \sqcup F$, then the probability for $\bar{\mG} = (\bar{\mS}, \bar{\mathsf{f}}, \bar{\mF})$ to assume $G$ and that a uniformly at random drawn vertex from $\{*\} \sqcup \bar{\mF}$ equals $v$ is given by
	\[
		\Pr{\bar{\mG} = G} (1 + |F|)^{-1} = \Pr{\mG=G} (1 + \Ex{\zeta})^{-1}.
	\]
	Thus
	\[
		\Pr{(\hat{\mH}_{[k]}, u^*) = (\mathbf{H}, u)} = p(\mathbf{H}) \Pr{\mG = G_{\ell_1}}(1 + \Ex{\zeta})^{-1} \prod_{i \ne \ell_1} \Pr{\mG_i = G_i}.
	\]
	
	For Claim~{\em (4)}, let us start with the fixpoints. By assumption, there is a sequence $t_n \to 0$ such that $k_n = t_n \sqrt{n} \to \infty$. For any integer $m \ge 0$ let $S_m$ denote the sum of $m$ independent copies of the size of the $\xi$-Galton--Watson tree $\cT^f$. The number of fixpoints $\#_f \hat{\mH}_k$  is given by the sum of
	\[
		\#_f \hat{\mH}_0 \eqdist 1 + S_{|\bar{\mf}|}.
	\]
	and the independent differences
	\[
		\#_f\hat{\mH}_i - \#_f\hat{\mH}_{i-1} \eqdist S_{|\hat{\mf}| -1 }.
	\]
	Consequently, 
	\[
		\#_f \hat{\mH}_k \eqdist 1 + S_{M_k}
	\]
	with $M_k = |\bar{\mf}| + \sum_{i=1}^k (\hat{\xi}_i-1)$ and $(\hat{\xi}_i)_{i \ge 1}$ denoting a family of independent copies of $|\hat{\mf}|$. By a general result for the size of Galton--Watson forests, there is a constant $C>0$ such that 
	\[
	\Pr{S_m \ge x} \le C m x^{-1/2}
	\]
	for all $m$ and $x>0$. See Devroye and Janson \cite[Lem. 2.3]{MR2829308} and Janson \cite[Lem. 2.1]{MR2245498}. As $\hat{\xi}$ and $|\bar{\mf}|$ have finite first moments, it follows that
	\[
	\Pr{S_{M_k} \ge x} \le C \Ex{M_k} x^{-1/2} = C (\Ex{|\bar{\mf}|} + k( \Ex{\hat{\xi}}-1)) x^{-1/2}.
	\]
	Setting $x = nt_n-1$ and $k=k_n$, it follows that $\#_f \hat{\mH}_{k_n} \le nt_n$ with probability tending to one as $n$ becomes large.
	
	For the second statement, let $R_m$ denote the sum of $m$ independent copies of $|\cT| - (1 + \Ex{\zeta}) |\cT^f|$. The difference $\# \hat{\mH}_k - (1 + \Ex{\zeta}) \#_f \hat{\mH}_k$ is given by the sum of 
	\[
		\# \hat{\mH}_0 - (1 + \Ex{\zeta})\#_f \hat{\mH}_0  \eqdist  |\bar{\mF}| - \Ex{\zeta} + R_{|\bar{\mf}|} 
	\]
	 and the independent differences 
	 \[
	 \# \hat{\mH}_i - \# \hat{\mH}_{i-1} - (1 + \Ex{\zeta})(\#_f \hat{\mH}_i - \#_f \hat{\mH}_{i-1})  \eqdist  |\hat{\mF}| - \Ex{\zeta} + R_{|\hat{\mf}|-1}, \quad i=1\ldots k.
	 \]
	 Consequently, 
	\begin{align}
		\label{eq:step1}
		\# \hat{\mH}_k - (1 + \Ex{\zeta}) \#_f \hat{\mH}_k \eqdist |\bar{\mF}| - \Ex{\zeta} + \sum_{i=1}^k(\hat{\zeta}_i - \Ex{\zeta}) + R_{M_k}
	 \end{align}
	 with $\hat{\zeta}_i$ denoting independent copies of $|\hat{\mF}|$. Markov's inequality implies that for $k=k_n$
	 \begin{align}
		 \label{eq:step2}
		| |\bar{\mF}| - \Ex{\zeta} + \sum_{i=1}^{k_n}(\hat{\zeta}_i - \Ex{\zeta}) | \le \sqrt{nt_n} / 2
	 \end{align}
	 with probability tending to one as $n$ becomes large. 
	 
	 It remains to show that $|R_{M_{k_n}}| \le \sqrt{n t_n}/2$ with probability tending to one. If $\xi$ and $\zeta$ were independent, then this would be rather simple. But, as this is not necessarily the case, it requires a bit of effort. We may write
	 \begin{align}
		 \label{eq:interm}
	 |\cT| - (1 + \Ex{\zeta}) |\cT^f|  = \sum_{i=0}^{\infty}D_i\quad \text{with} \quad D_i = L_i^\cG(\cT^f) - L_i(\cT^f) \Ex{\zeta}.
	 \end{align}
	 The sum is finite, as $D_i=0$ for $i> \He(\cT^f)$. Since \[\Pr{\He(\cT^f) \ge h} \sim 2/(\Va{\xi}h)\] as $h$ becomes large, 
	 it follows that the probability for the maximum height of $m$ independent copies of $\cT^f$ to be less than $k_n / \sqrt{t_n}$ is given by
	 \begin{align}
		 \label{eq:expr}
		 \left(1 - \frac{(1+o(1)) 2 \sqrt{t_n}}{\Va{\xi} k_n} \right)^{m},
	 \end{align}
	 with the $o(1)$ term not depending on $m$. As $\Ex{M_{k_n}} \sim k_n(\Ex{\hat{\xi}} - 1)$ and $\Va{M_{k_n}} = k_n \Va{\hat{\xi}}$, it follows by Chebyshev's inequality that \[|M_{k_n} - k_n| \le k_n^{3/4}\] with probability tending to one. Hence Expression~\eqref{eq:expr} implies that the probability for the maximum height of $M_{k_n}$ independent copies of $\cT^f$ to be smaller than $k_n / \sqrt{t_n}$ tends to one as $n$ becomes large. Equation~\eqref{eq:interm} hence implies for all $x$ that
	 \begin{align}
		 \label{eq:last123}
		 \Pr{|R_{M_{k_n}}| \ge x} \le o(1) +  \Pr{|U_1 + \ldots + U_{k_n}| \ge x}
	 \end{align}
	 with the $U_j$ denoting independent copies of $\sum_{i= 0}^{\lfloor k_n/ \sqrt{t_n} \rfloor} D_i$. We aim to apply Chebyshev's inequality again, and hence compute the expected value and variance of the $U_j$.
	 It holds for all $i$ and $\ell$ that
	 \[
		 \Ex{D_i \mid L_i(\cT^f) = \ell} = \Ex{\sum_{j=1}^\ell \zeta_j - \ell \Ex{\zeta}} = 0 ,
	 \]
	 with $(\zeta_j)_j$ denoting a family of independent copies of $\zeta$. Consequently, for all $k$
	 \[
		 \Ex{ \sum_{i=0}^k D_i} = 0 \quad \text{and} \quad \Va{\sum_{i=0}^k D_i} = \sum_{i=0}^k \Ex{D_i^2} + 2 \sum_{0 \le i < j \le k} \Ex{D_i D_j}.
	 \]
	 For all $0\le i<j$ it holds that
	 \[
		\Ex{ D_i D_j \mid ((\cT^f)^{[j]}, (\beta(v))_{v \in (\cT^f)^{[j-1]}}} = D_i \Ex{D_j \mid L_j(\cT^f)} = 0,
	 \]
	 because for all $\ell$
	 \[
	 \Ex{D_j \mid L_j(\cT^f) = \ell} = \Ex{\sum_{r=1}^\ell \zeta_i - \ell \Ex{\zeta} }= 0.
	 \]
	 Clearly it also holds that
	 \[
		 \Ex{D_j^2 \mid L_j(\cT^f)= \ell} = \ell \Va{\zeta}.
	 \]
	 Hence,
	 \[
		 \Va{\sum_{i=0}^k D_i} = \Va{\zeta} \Ex{\sum_{i=0}^k L_i(\cT^f)} = (k+1) \Va{\zeta}.
	 \]
	 Setting $k= \lfloor k_n / \sqrt{t_n} \rfloor$, we may thus apply Chebyshev's inequality to  $\eqref{eq:last123}$ and obtain
	 \[
		 \Pr{|R_{M_{k_n}} \ge \sqrt{nt_n}/2} \le o(1) + \frac{k_n ( k_n / \sqrt{t_n} + 1)\Va{\zeta}}{n t_n / 4} = o(1).
	 \]
	 Together with \eqref{eq:step1} and \eqref{eq:step2} this implies that
	 \[
		 |\# \hat{\mH}_{k_n} - (1 + \Ex{\zeta}) \#_f \hat{\mH}_{k_n} | \le \sqrt{n t_n}
	 \]
	 with probability tending to one as $n$ becomes large.
\end{proof}

\begin{proof}[Proof of Theorem~\ref{te:localhaupt}]
	For any $k$, let $\cE_k$ denote the set of pairs $(\mathbf{H}, u)$ where $\mathbf{H} = (H_i)_{0 \le i \le k}$ is a representation of a $\cG$-enriched tree as sequence of increasing enriched fringe subtrees, and $u$ is either the root of $H_0$ or an element of the non-fixpoints of the $\cG$-object corresponding to the root of $H_0$. We need to show that
	\begin{align}
		\label{eq:showme1}
		\lim_{n \to \infty} \sup_{\cH \subset \cE_{k_n}} | \Pr{ (\mathsf{H}_{[k_n]}^n, v^*) \in \cH } - \Pr{ (\hat{\mathsf{H}}_{[k_n]}, u^*) \in \cH }| = 0.
	\end{align}
	By assumption, there is a sequence $t_n \to 0$ such that $k_n = \sqrt{n}t_n$. We define the subset
	\[
		\cE_{k,n} = \{ ((H_i)_{0 \le i \le k}, u) \in \cE_k \mid \#_f H_k \le n t_n, |\#H_k - (1 + \Ex{\zeta}) \#_f H_k| \le \sqrt{n t_n} \}.
	\]
	Lemma~\ref{le:slow2} implies that $(\hat{\mathsf{H}}_{[k_n]}, u^*)$ lies in $\cE_{k_n, n}$ with probability tending to one as $n$ becomes large. Hence, if we verify Equation~\eqref{eq:showme1} with the index $\cH$ only ranging over the subsets of $\cE_{k_n,n}$, then it follows that $(\mathsf{H}_{[k_n]}^n, v^*)$ also lies with probability tending to one in $\cE_{k_n, n}$. But this already verifies $\eqref{eq:showme1}$ when $\cH$ ranges over all subsets of $\cE_{k_n}$, and we are done.
	
	So it remains to show that for any $\epsilon >0$ it holds for sufficiently large $n$ that
	\begin{align}
		\label{eq:showme2}
		\sup_{\cH \subset \cE_{k_n,n}} | \Pr{ (\mathsf{H}_{[k_n]}^n, v^*) \in \cH } - \Pr{ (\hat{\mathsf{H}}_{[k_n]}, u^*) \in \cH }| \le \epsilon.
	\end{align}
	In order to show this, we first exert some control over the number of fixpoints in $(\cT_n, \beta_n)$. Let $\mathbf{B}, \mathbf{\Sigma}, \mu$ and  $\sigma$  be as in Lemma~\ref{le:asymptotic}. If follows from this Lemma that, as $n \equiv 1 \mod \spa(\mathbf{w})$ tends to infinity, 
		\[
		\sqrt{n} \Pr{|\cT_n^f| = \ell} \sim  \frac{d}{\sigma \sqrt{2 \pi }} \exp(- \frac{x_\ell^2}{2 \sigma^2})
		\]
		uniformly for all bounded $x_\ell \in \ndR$ with
		\[
		\ell := \mu n + x_\ell \sqrt{n} \in n + d \ndZ.
		\]
		Furthermore,
		\[
		\frac{|\cT_n^f| - n \mu}{\sqrt{n}} \convdis \cN(0, \sigma^2)
		\]
		and for any $\epsilon_1>0$ there is a constant $M>0$ such that for all $n$
		\[
		\Pr{|\cT_n^f| \notin I_n} \le \epsilon_1 \quad \text{with} \quad I_n := (n + d\ndZ) \cap [n / (1 + \Ex{\zeta}) - M \sqrt{n}, n / (1 + \Ex{\zeta}) + M \sqrt{n}].
		\]
		Hence we may bound the expression in \eqref{eq:showme2}  by
		\begin{align}
		\label{eq:showme3}
		\epsilon_1 + 
		\sup_{(\mathbf{H},u) \in \subset \cE_{k_n,n}} \left | \frac{\Pr{ (\mathsf{H}_{[k_n]}^n, v^*) =(\mathbf{H},u) , |\cT_n^f|\in I_n} } {\Pr{ (\hat{\mathsf{H}}_{[k_n]}, u^*) = (\mathbf{H},u) }} -1 \right |.
		\end{align}
	Throughout the rest of the proof, we set $k=k_n$. Let $(\mathbf{H}, u)  \in \cE_{k, n}$ with $\mathbf{H} = (H_i)_{0 \le i \le k}$ be given. We set $L = L_k(H_k) + L_k^\cG(H_k)$ and let $F_0$ denote the size of the forest of the $\cG$-object corresponding to the root of $H_0$. Let $A$ denote a $\cG$-enriched tree with a total number of vertices $\#A = n$. Given $(\cT_n, \beta_n) = A$, the vertex $v^*$ is drawn uniformly at random from $A$. Given $(\cT_n, \beta_n) = A$ and $\mH_k^n = H_k$, the vertex $v^*$ is drawn uniformly at random from $L$ possible locations. Out of these, exactly $p(\mathbf{H})(1 + F_0)$ many correspond to the event $ \mH_{[k]}^n = \mathbf{H}$, since the number  $p(\mathbf{H})$ defined in Lemma~\ref{le:slow2} counts the number of fixpoints $v$ at height $k$ in $H_k$ with the property, that the extended enriched fringe subtree representation with respect to $v$ is identical to $\mathbf{H}$. Moreover, given additionally $ \mH_{[k]}^n = \mathbf{H}$, there is precisely one out of $1 + F_0$ possible locations such that $v^* = u$. Hence
	\begin{align}
		\label{eq:tmp10}
			\Pr{ \mH_{[k]}^n = \mathbf{H} \mid (\cT_n, \beta_n) = A} = \Pr{\mH_n^k= H_k \mid (\cT_n, \beta_n) = A}p(\mathbf{H}) / L.
	\end{align}
	There is a $1$ to $L$ correspondence between the fixpoints $v$ of $A$  with $f( A, v) = H_k$, and the possible locations for $v^*$ such that $\mH_k^n = H_k$. Hence
	\begin{align}
		\label{eq:tmp11}
		\Pr{\mH_n^k= H_k \mid (\cT_n, \beta_n) = A} = \Ex{\sum_{v \in \cT_n^f} \one_{f( (\cT_n, \beta_n), v) = H_k} \mid (\cT_n, \beta_n) = A} L / n
	\end{align}
	Let $(G_1, \ldots, G_t)$ and $\mG_i^n$, $i=1, \ldots, |\cT_n^f|$ denote the depth-first-search ordered lists of $\cG$-objects of the enriched trees $H_k$ and $\cT_n^f$. The occurrences of $(G_1, \ldots, G_t)$ as substrings of $(\mG_i^n)_{i}$ correspond precisely to the vertices of $\cT_n^f$ where the fringe-subtree equals $H_k$.
	As Equations \eqref{eq:tmp10} and \eqref{eq:tmp11} hold uniformly for all $\cG$-enriched trees with $n$ vertices, it follows that
	\begin{align}
		\label{eq:anotherstep}
		\Pr{ \mH_{[k]}^n = \mathbf{H}, |\cT_n^f|\in I_n} = \sum_{\ell \in I_n}  \Ex{\sum_{j=1}^\ell J_j^n , |\cT_n^f| = \ell} p(\mathbf{H}) /n
	\end{align}
	with $J_j^n$ denoting the indicator variable for the event $(\mG_{j}^n, \ldots, \mG_{j+t-1}^n) = (G_1, \ldots, G_t)$. Here we set $\mG_{i}^n := \mG_{i -\ell}^n$ whenever $i > \ell$. Note that $J_j = 0$ whenever $j+t-1 > \ell$, as then the sequence $(\mG_{i}^n)_{j \le i \le j+ t-1}$ does not correspond to any $\cG$-enriched tree at all. Hence the sum of the $J_j$ really counts the number of occurrences of $(G_i)_{1 \le i \le t}$ in $(\mG_i)_{1 \le i \le \ell}$. Recall that by Lemma~\ref{le:ucoup} there is a natural coupling of the unconditioned $\cG$-enriched tree $(\cT, \beta)$ with a family $(\mG_i)_{i \in \ndN}$ of independent copies $\mG_i = (\mathsf{S}_i, \mathsf{f}_i, \mathsf{F}_i)$ of the random $\cG$-object $\mG$, conditioned on the event that there is an initial segment of $(\mG_i)_{i \in \ndN}$ that corresponds to an $\cG$-enriched tree.
	Let $(J_j)_{j = 1, \ldots, \ell}$ denote the unconditioned pendants of $J_j^n$ for the sequence $(\mG_i)_{1 \le i \le \ell}$. Lemma~\ref{le:ucoup} implies that
	\[
		\Ex{\sum_{j=1}^\ell J_j^n , |\cT_n^f| = \ell} = \Pr{|\cT_n^f| = \ell } \Ex{\sum_{j=1}^\ell J_j \mid \sum_{i=1}^\ell(|\mathsf{f}_i|, |\mathsf{F}_i|) = (\ell -1, n - \ell), \sum_{i=1}^m |\mathsf{f_i}| \ge m \text{ for all $m < \ell$} }.
	\]
	The sum $\sum_{j_1}^\ell J_j$ is invariant under cyclic permutations of the list $(\mG_1, \ldots, \mG_\ell)$. Hence the Cycle Lemma~\ref{le:cycle} yields
	\begin{align*}
		\Ex{\sum_{j=1}^\ell J_j^n , |\cT_n^f| = \ell} =  \Pr{|\cT_n^f| = \ell } \Ex{\sum_{j=1}^\ell J_j \mid \sum_{i=1}^\ell(|\mathsf{f}_i|, |\mathsf{F}_i|) = (\ell -1, n - \ell)}. 
	\end{align*}
	Conditioned on this simpler event, the $J_j$ are all identically distributed. Hence
	\begin{align*}
	\Ex{\sum_{j=1}^\ell J_j^n , |\cT_n^f| = \ell} = \ell \Pr{|\cT_n^f| = \ell }\Ex{J_1 \mid \sum_{i=1}^\ell(|\mathsf{f}_i|, |\mathsf{F}_i|) = (\ell -1, n - \ell) }.
	\end{align*}
	Setting $\pi_G = \Pr{\mG = G}$ for all $\cG$-objects $G$, it follows that
	\[
		\Ex{J_1 ,\sum_{i=1}^\ell(|\mathsf{f}_i|, |\mathsf{F}_i|) = (\ell -1, n - \ell) } = \pi_{G_1} \cdots \pi_{G_t} \Pr{\sum_{i=t+1}^{\ell} (|\mathsf{f}_i|, |\mathsf{F_i}|) = (\ell-1-t, n - \#H_k) }.
	\]
	 It holds uniformly for $(\mathbf{H},u) \in \cE_{k_n, n}$ and  $\ell = \mu n + x_\ell \sqrt{n} \in I_n$ that 
	 \[
	 \ell^{-1/2}\left | (\ell-1-t, n  - \#H_k) - (\ell -t)\Ex{(\xi, \zeta)} \right | \sim (0, -x_\ell (1 + \Ex{\zeta})^{3/2})
	 \]
	 as $n$ becomes large. Hence, as $\ell -t \sim \ell$, we may apply the multivariate central local limit theorem~\ref{le:llt2dim} to obtain
	 \[
	 \Pr{\sum_{i=t+1}^{\ell} (|\mathsf{f}_i|, |\mathsf{F_i}|) = (\ell-1-t, n - \#H_k) } \sim \frac{|\det \mathbf{B}| }{n\sqrt{2 \pi \det \mathbf{\Sigma}}} \exp(- \frac{x_\ell^2}{2 \sigma^2}) .
	 \]
	 Likewise, the probability $\Pr{\sum_{i=1}^{\ell} (|\mathsf{f}_i|, |\mathsf{F_i}|) = (\ell-1, n - \ell) }$ has precisely the same asymptotic order. It follows from Equation~\eqref{eq:anotherstep} that
	 \[
	 \Pr{ \mH_{[k]}^n = \mathbf{H}, |\cT_n^f|\in I_n} \sim p( \mathbf{H}) \pi_{G_1} \cdots \pi_{G_t} \sum_{\ell \in I_n} \frac{\ell}{n} \Pr{|\cT_n^f| = \ell}.
	 \]
	 Since $\ell/n \sim (1 + \Ex{\zeta})^{-1}$ and $\Pr{|\cT_n^f| \notin I_n} \le \epsilon_1$ it follows that uniformly for $(\mathbf{H},u) \in \cE_{k_n, n}$
	 \begin{align}
	 \label{eq:fin}
	 \left | \frac{ \Pr{ \mH_{[k]}^n = \mathbf{H}, |\cT_n^f|\in I_n}}{ p(\mathbf{H})(1 + \Ex{\zeta})^{-1} \pi_{G_1} \cdots \pi_{G_t}}  -1 \right | \le  \epsilon_1
	 \end{align}
	 for $n$ large enough. Setting $\epsilon_1 = \epsilon/2$, Inequality~\eqref{eq:fin} and Lemma~\ref{le:slow2} imply that the bound in \eqref{eq:showme3} is smaller than $\epsilon$ for large enough $n$. This verifies $\eqref{eq:showme2}$ and hence the proof is complete.
\end{proof}

\subsection{Proofs of the scaling limits and diameter tail bounds in Section~\ref{sec:partA}}
\label{sec:propartA}

\begin{proof}[Proof of Lemma~\ref{le:blobs}]
		Claim {\em (1)}:
		Let $x \ge 0$ be arbitrary and let $\cE$ denote the event that there is a vertex $v \in \cT^f$ with $|f(v)| + |F(v)|\ge x$. It follows from Claim {\em (1)} that
		\[
		\Pr{|\cT| = n} = [z^n]\tilde{\cA}^\omega_\cR(\rho z) / \tilde{\cA}^\omega_\cR(\rho) = \Theta(n^{-3/2})
		\]
		and hence
		\[
		\Pr{\cE \mid |\cT|=n} = O(n^{3/2}) \Pr{\cE, |\cT| = n}
		.\]
		By Lemma~\ref{le:ucoup}, the probability $\Pr{\cE, |\cT| = n}$ is given by
		\begin{align*}
		\Pr{ \exists \ell \le n: \max(\xi_1 + \zeta_1, \ldots, \xi_\ell + \zeta_\ell)\ge x, \sum_{i=1}^\ell \xi_i = \ell -1, \sum_{i=1}^\ell (1+\zeta_i) = n, \forall m < \ell: \sum_{i=1}^m \xi_i \ge m},
		\end{align*}
		with $(\xi_i, \zeta_i)_{i \in \ndN}$ denoting a list of independent copies of $(\xi, \zeta)$. We are not interested in precise asymptotics here and hence bound this very roughly by
		\[
		\Pr{\cE, |\cT| = n} \le \Pr{\max(\xi_1 + \zeta_1, \ldots, \xi_n + \zeta_n) \ge x} \le n \Pr{\xi + \zeta \ge x}.
		\]
		As $(\xi, \zeta)$ has finite exponential moments, it follows that
		\[
		\Pr{\cE, |\cT| = n} \le C n^{5/2} \exp(-cn)
		\]
		for some constants $C,c>0$ that do not depend on $n$.
	
	Claim {\em (2)}: We may form a random metric space $\mY$ by constructing a metric $d_\mY$ on the vertex set of $\Gamma \cS(\rho)$ by patching together independent copies of the metrics $\delta_R$ just as in the construction of the metric space $\mY_n$. Hence $\mY_n$ is distributed like the space $\mY$ conditioned on having size $n$. For any vertex $v$ of the fixpoint tree $\cT^f$ let $D_v$ denote the $d_{\mY}$-diameter of the subspace $\{v\} \cup F(v) \subset \mY$. Given $h \ge 0$ let $\cE'$ denote the event that $D_v \ge h$ for at least one vertex $v \in \cT^f$.

Using Claim~{\em (1)} it follows that 
\[
\tag{$*$}
\Pr{\cE' \mid |\mY|= n} = O(n^{3/2}) \Pr{\cE', |\mY| = n} = O(n^{5/2}) \Pr{D_o \ge h}
\]
with $o$ denoting the root of the fixpoint tree $\cT^f$.  

By assumption, for any vertex $u \in F(o)$, the distance $d_{\mY}(o,u)$ is bounded by the sum of $\sum_{e} d^+_{\cT}(e)$ many independent copies of a real-valued random variable $\chi \ge 0$ having finite exponential moments, with the sum index $e$ ranging over all ancestors of the vertex $u$ in the tree $\cT$. Clearly we have that
\[
\sum_{e} d^+_{\cT}(e) \le |F(o)|
\]
for all $u \in F(o)$. Since
$
D_o \le 2 \sup_{u \in F(o)} d_{\mY}(o,u),
$
it follows that 
\[
\Pr{D_o \ge h} \le \sum_{k=0}^\infty \Pr{|F(o)| = k} k \Pr{\chi_1 + \ldots + \chi_k \ge h/2}
\]
with $(\chi_i)_{i \in \ndN}$ a family independent copies of $\chi$. Moreover,
\[
\Pr{|F(o)| = k} = \Pr{\zeta = k} = O(\gamma_1^k)
\]
for some constant $0<\gamma_1<1$. By the deviation inequality given in Lemma~\ref{le:deviation} it follows that there are constants $a,b>0$ such that
\[
\gamma_1^k \Pr{\chi_1 + \ldots + \chi_k \ge h/2} \le 2 e^{-ak-bh}
\]
for all $k$ and $h$. Hence
\[
\Pr{D_o \ge h} = O(\gamma_2^h)
\]
for some constant $0<\gamma_2<1$. Hence  Equality $(*)$ implies that
\[
\Pr{\cE' \mid |\mY| = n} = O(n^{5/2}) \gamma_2^h
\]
and we are done.

Claim {\em (3)}: Set $t := \sqrt{(1 + \Ex{\zeta})} \sigma/2$ with $\sigma^2 = \Va{\xi}$. Let $g: \ndK^\bullet \to \ndR$ denote a Lipschitz-continuous, bounded function defined on the space of isometry classes of pointed compact metric spaces. Note that $\cT^f_n$ conditioned on having size $\ell$ is distributed like $\cT^f$ conditioned on having size $\ell$. Hence it is distributed like a $\xi$-Galton--Watson tree conditioned on having $\ell$ vertices, which we denote by $\cT'_\ell$. It follows from claim $ii)$ that 
\[\Ex{g(t \cT^f_n/\sqrt{n})} = o(1) + \sum _\ell \Ex{g(t \cT'_\ell / \sqrt{n})} \Pr{|\cT^f| = \ell}\]
with the sum index $\ell$ ranging over all integers $\ell \equiv 1 \mod \spa(\mathbf{w})$ contained in the interval $(1 \pm n^{-s})\frac{n}{1 + \Ex{\zeta}}$. Since $g$ is Lipschitz-continuous, we have that \[|\Ex{g(t \cT'_\ell / \sqrt{n})} - \Ex{g(\sigma \cT'_\ell/ (2 \sqrt{\ell}))}| \le a_{n,\ell} \Ex{\Di(\cT'_\ell)/\sqrt{\ell}}\] for some constants $a_{n, \ell}$ with $\sup_\ell(a_{n,\ell}) \to 0$ as $n \equiv 1 \mod \spa(\mathbf{w})$ tends to infinity. Moreover, the average rescaled diameter $\Ex{\Di(\cT'_\ell)/\sqrt{\ell}}$ converges as $\ell$ becomes large to a multiple of the expected diameter of the CRT $\CRT$. In particular, it is a bounded sequence. Since $\Ex{g(\frac{\sigma \cT'_\ell}{2 \sqrt{\ell}})}$ converges to $\Ex{g(\CRT)}$ as $\ell$ becomes large, it follows that $\Ex{g(\frac{t \cT^f_n}{ \sqrt{n}})}$ converges to $\Ex{g(\CRT)}$ as $n$ becomes large. This concludes the proof.
\end{proof}

\begin{proof}[Proof of Theorem~\ref{te:main2}]
By Lemma~\ref{le:blobs}  it follows that with high probability all vertices $v \in \cT^f_n$ have the property that the $d_{\mY_n}$-diameter of the subspace $\{v\} \cup F(v)$ is at most $O(\log(n))$. This implies that the Gromov--Hausdorff distance between the metric spaces $(\mY_n, d_{\mY_n} / \sqrt{n})$ and  $(\cT^f_n, d_{\mY_n} / \sqrt{n})$ converges in probability to zero. Moreover, by Lemma~\ref{le:blobs} we know that $(\cT_n^f,c d_{\cT_n^f}/\sqrt{n})$ with $c = \sqrt{(1 + \Ex{\zeta})\Va{\xi}}/2$ converges weakly to the CRT $\CRT$. It remains to show that there is a constant $c'$ such that the Gromov--Hausdorff distance between $(\cT^f_n, d_{\mY_n} / \sqrt{n})$ and $(\cT^f_n, c' d_{\cT^f_n} / \sqrt{n})$ converges in probability to zero.

We define the random number $\eta$ as follows. Choose a random $\cR'$-symmetry $(\mR, \sigma)$ from $\bigcup_{k \ge 0} \Sym(\cR')[k]$ with probability proportional to its weight \[\frac{\kappa(\mR)}{|\mR|!} \tilde{\cA}^\omega_{\cR}(\rho)^{\sigma_1} \tilde{\cA}^\omega_{\cR}(\rho^2)^{\sigma_2} \cdots\] and let $\eta$ denote the $\delta_\mR$-distance of the two distinct $*$-labels. Note that by our assumptions on the cycle index sum $Z_{\cR^\kappa}$ we have that $|\mR|$ has finite exponential moments. Moreover, the diameter of the metric $\delta_\mR$ is bounded by $|\mR|$ many independent copies of a real-valued random variable $\chi \ge 0$ with finite exponential moments. Hence $\eta$ has finite exponential moments. We are going to show that the Gromov--Hausdorff distance of $(\cT^f_n, d_{\mY_n} / \sqrt{n})$ and $(\cT^f_n, \Ex{\eta} d_{\cT^f_n} / \sqrt{n})$ converges in probability to zero. By the discussion in the preceding paragraph this implies that \[\frac{\sqrt{(1 + \Ex{\zeta})\Va{\xi}}}{2 \Ex{\eta} \sqrt{n}}\mY_n \convdis \CRT\] and we are done.

Let $s>1$ and $t>0$ be arbitrary constants and set $s_n = \log(n)^s$ and $t_n = n^t$. Let $\epsilon >0$ be given and let $\cE_1$ denote the event that there exists a fixpoint $v \in \cT_n^f$ and an ancestor $u$ of $v$ with the property that \[d_{\cT_n^f}(u,v) \ge s_n \quad \text{and} \quad d_{\mY_n}(u,v) \notin (1 \pm \epsilon) \Ex{\eta} d_{\mT_n}(u,v).\] Likewise, let $\cE_2$ denote the event that there exists a vertex $v$ and an ancestor $u$ of $v$ with \[d_{\cT_n}(u,v) \le s_n \quad \text{and} \quad d_{\mY_n}(u, v) \ge t_n.\] We are going to show that with high probability none of the events $\cE_1$ and $\cE_2$ takes place

This suffices to show the claim: Take $s=2$ and $t=1/4$ and suppose that the complementary events $\cE_1^{c}$ and $\cE_2^{c}$ hold. Given vertices $a \ne b$ in  the tree $\cT_n^f$ let $x$ denote their lowest common ancestor. If $x \in \{a,b\}$ then we have \[d_{\mY_n}(a,b) = d_{\mY_n}(a,x) + d_{\mY_n}(b,x).\] If $x \ne a,b$, then let $a'$ denote the offspring of $x$ that lies on the $\cT_n^f$-path joining $a$ and $x$ and likewise $b'$ the offspring of $x$ lying on the path joining $x$ and $b$. Hence we have that \[d_{\mY_n}(a,b) = d_{\mY_n}(a,x) + d_{\mY_n}(b,x) + R \quad \text{with} \quad R = d_{\mY_n}(a,a') - d_{\mY_n}(a',x) - d_{\mY_n}(b',x).\] By property $\cE_2^c$ and the triangle inequality it follows that $|R| = -R \le 2 n^{1/4}$. Thus, regardless whether $x \in \{a,b\}$, it holds that
\[
d_{\mY_n}(a,b) = d_{\mY_n}(a,x) + d_{\mY_n}(b,x) + O(n^{1/4}).
\]
Moreover, if $d_{\cT_n^f}(a,x) \ge \log(n)^2$, then it follows by property $\cE_1^c$ that \[d_{\mY_n}(a,x) \in (1 \pm \epsilon)\Ex{\eta} d_{\cT_n^f}(a,x).\] Otherwise, if $d_{\cT_n^f}(a,x) < \log(n)^2$ then it follows by property $\cE_2^c$ that $d_{\mY_n}(a,x) \le n^{1/4}$ and thus \[|d_{\mY_n}(a,x) - \Ex{\eta} d_{\cT_n^f}(a,x)| \le Cn^{1/4}\] for a fixed constant $C$ that does not depend on $n$ or the points $a$ and $x$. It follows that 
\[
	|d_{\mY_n}(a,b)/\sqrt{n} - \Ex{\eta}d_{\cT_n^f}(a,b)/\sqrt{n}| \le \epsilon \Di(\cT_n^f)/\sqrt{n} + o(1),
\]
with $\Di(\cT_n^f)$ denoting the diameter. 
Thus  \[d_{\text{GH}}(\mY_n, \Ex{\eta}\cT_n^f) \le \epsilon \Di(\cT_n^f)/\sqrt{n} + o(1)\] holds with high probability. Since we may choose $\epsilon$ arbitrarily small, and $\Di(\cT_n^f)/\sqrt{n}$ converges in distribution (to a multiple of the diameter of the CRT),  it follows that $d_{\text{GH}}(\mY_n, \Ex{\eta}\cT_n^f) \to 0$ in probability and we are done.

For each finite subset $U \in \ndN$ and each $\cR$-structure $R \in \cR[U]$ let $(\delta_R^i)_{i \in \ndN_0}$ be a family of independent copies of the metric $\delta_R$. Given a $\cA_\cR$-symmetry $S = ((T,\alpha), \sigma)$ with label set $[k]$ for some $k \ge 0$ we may form the family $(\delta^S(v))_{v \in T}$ of random metrics by traversing bijectively the vertices of $T$ in ascending order $1,2, \ldots k$ and assigning to each vertex $v$ the "leftmost" unused copy from the list $(\delta_{\alpha(v)}^1, \delta_{\alpha(v)}^2, \ldots)$. The metrics can be patched together to a metric $d^{S}$ on the vertex set $[k]$ of the tree $T$ just as described in Section~\ref{sec:partA}.

We may assume that all random variables considered so far are defined on the same probability space and that the metric $d_{\mY_n}$ of $\mY_n$ coincides with the metric $d^{\mZ_n}$ with $\mZ_n$ denoting the sampler $\Gamma Z_{\cA_\cR^\omega}(\rho)$ conditioned on having size $n$. Given the family $(\delta^i_R)_{R,i}$ let $\cH \subset \bigcup_{k=0}^\infty \Sym(\cA_\cR)[k]$ denote the finite set symmetries of size $n$ such that the event $\cE_1$ takes place if and only if $\mZ_n \in \cH$. By the definition of the event $\cE_1$ for any symmetry $S = ((T, \alpha), \sigma) \in \cH$ we may choose a fixpoint $v_S$ of $\sigma$ having the property that there exists an ancestor $u$ in the tree $T$ with \[d_T(u,v_S) \ge s_n\quad \text{and} \quad d^S(u, v_S) \notin (1 \pm \epsilon)\Ex{\eta}d_T(u,v).\] Let $\ell_S$ denote the height $h_T(v_S)$. Note that since $v_S$ is a fixpoint, the tupel $(T, \alpha, v_S, \sigma)$ is a $\cA_\cR^{(\ell_S)}$-symmetry. By Lemma~\ref{le:asymptotic} the probability for the sampler $\Gamma Z_{\cA_\cR^\omega}(\rho)$ to have size $n$ is $\Theta(n^{-3/2})$ and we have that $\rho \widetilde{\cR'\circ \cA_\cR^\omega}(\rho) = 1$. Hence Equation~(\ref{eq:holygrail}) implies that the conditional distribution of the event $\cE_1$ given $(\delta^i_R)_{R,i}$ equals
\[
\sum_{S \in \cH} \Pr{\mZ_n = S \mid (\delta^i_R)_{R,i}}  = \Theta(n^{3/2}) \sum_{S \in \cH} \Pr{\Gamma Z_{\cA_\cR^{(\ell_S)}}(\rho) = (S,v_S)\mid (\delta^i_R)_{R,i}}
.\] Let $v_0, \ldots, v_{\ell}$ denote the spine of $\Gamma Z_{\cA_\cR^{(\ell)}}(\rho)$, that is, $v_{\ell}$ is the outer root, $v_0$ is the inner root, and $(v_0, \ldots, v_{\ell})$ is the directed path connecting the roots. It follows that the probability for the event $\cE_1$ is bounded by
\[
\tag{$*$} \Theta(n^{3/2}) \sum_{\ell=1}^n \sum_{k=0}^{\ell - s_n} \Pr{d^{\left(\Gamma Z_{\cA_\cR^{(\ell)}}(\rho)\right)}(v_k, v_{\ell}) \notin (1 \pm \epsilon) \Ex{\eta} (\ell-k)}
\]
But the $d^{\left(\Gamma Z_{\cA_\cR^{(\ell)}}(\rho)\right)}$-distance between spine vertices $v_i$ and $v_j$ is distributed like the sum $\eta_1 + \ldots + \eta_{|i-j|}$ of independent copies $(\eta_i)_i$ of $\eta$. We know that $\eta$ has finite exponential moments and hence by the deviation inequality in Lemma~\ref{le:deviation} the bound $(*)$ converges to zero as $n \equiv 1 \mod \spa(\mathbf{w})$ tends to infinity. Thus with high probability $\cE_1$ does not hold. By the same arguments we may bound the probability for the event $\cE_2$ by \[\Theta(n^{3/2}) \sum_{\ell = 1}^{n} \sum_{k=1}^{\min(s_n,\ell)} \Pr{\eta_1 + \ldots + \eta_{k} \ge t_n}\]
which also converges to zero. This concludes the proof.
\end{proof}

\begin{proof}[Proof of Theorem~\ref{te:tail2}]
It suffices to show that there are constants $C,c,N>0$ such that for all $n \ge N$ and $h \ge \sqrt{n}$ we have that \[\Pr{\He(\mX_n) \ge h} \le C (\exp(-c h^2/n) + \exp(-c h)).\]

For any fixpoint $v \in \cT_n^f$ set $\ell(v) = \sum_u d^+_{\cT_n}(u)$ with the sum index $u$ ranging over all ancestors of the vertex $v$ in the plane tree $\cT^f_n$. Note that we are summing up the outdegrees in the tree $\cT_n$ and not in the tree $\cT_n^f$.
Moreover, for any vertex $y \in \cT_n$ let $v_y$ denote its closest fixpoint, that is, $v_y = y$ if $y$ is a fixpoint and otherwise $v_y$ is the unique vertex with $y \in F_n(v_y)$. If $y$ has height $\hei_{\mY_n}(y) \ge 2h$ then $\hei_{\mY_n}(v_y) \ge h$ or $d_{\mY_n}(u, v_y) \ge h$. Thus either $\He_{\mY_n}(\cT_n^f) \ge h$ or there exists a fixpoint $v \in \cT_n^f$ such that the $d_{\mY_n}$-diameter $D_v$ of the subspace $\{v\} \cup F(v)$ is greater than or equal to $h$.

Let $s>r>0$ be constants. Given $h \ge \sqrt{n}$ let $\cE^{2h}$ denote the event that $\He(\mX_n) \ge h$. It follows that $\cE^{2h} \subset \cE_0^h \cup \cE_1^h \cup \cE_2^h \cup \cE_3^h$ with the events $\cE_i^h$ given as follows. $\cE_0^h$ is the event that there exists a fixpoint $v \in \cT_n^f$ with $D_v \ge h$. $\cE_1^h$ is the event that $\He(\cT_n^f) \ge rh$. $\cE_2^h$ is the event that $\He(\cT_n^f) \le rh$ and $\ell(v) \ge sh$ for some fixpoint $v\in \cT_n^f$. $\cE_3^h$ is the event that $\ell(v) \le sh$ for all fixpoints $v\in \cT_n^f$ and $\He_{\mY_n}(\cT_n^f) \ge h$.

We are going to show that if we choose $r$ and $s$ sufficiently small, then each of these events is sufficiently unlikely. By Lemma~\ref{le:blobs} we have that \[\Pr{\cE_0^h} = O(n^{5/2}) \gamma^h\] for some $0< \gamma < 1$. Hence there are constants $C_0, c_0 >0$ such that 
$\Pr{\cE_0^h} \le C_0 \exp(-c_0 h^2/n)$ if $h \le n$ and $\Pr{\cE_0^h} \le C_0 \exp(-c_0 h)$ if $h \ge n$.

In order to bound the probability for the event $\cE_1^h$ note that the tree $\cT_n^f$ conditioned on having size $\ell$ is distributed like $\cT^f$ conditioned on having size $\ell$. That is, it is identically distributed to a $\xi$-Galton--Watson tree conditioned on having size $\ell$ which we denote by $\cT'_\ell$. Hence \[\Pr{\cE_1^h} \le \sum_{\ell=1}^n \Pr{|\cT_n^f| = \ell} \Pr{\He(\cT'_\ell) \ge rh}.\] By Inequality~(\ref{eq:gwttail}) there exist constants $C_1, c_1>0$ that do not depend on $n$ or $h$ such that for all $1 \le \ell \le n$ we have the tail bound \[\Pr{\He(\cT'_\ell) \ge rh} \le C_1 \exp(-c_1 r^2 h^2 / \ell) \le C_1 \exp(- c_1 r^2h^2 /n).\] In particular, it holds that $\Pr{\cE_1^h} \le C_1 \exp(-c_1 r^2 h^2 / n)$ for all $n$ and $h$.

We proceed to bound the probability for the event $\cE_2^h$. Let $\mZ_n$ denote the sampler $\Gamma Z_{\cA_\cR^\omega}(\rho)$ conditioned on having size $n$ and let $\cH \subset \bigcup_{k=0}^\infty \Sym(\cA_\cR)[k]$ denote the set of $\cA_\cR$-symmetries $S=((T, \alpha), \sigma)$ having the property that  there exists a fixpoint vertex $v_S$ in $T$ with the property that $\ell_S := \hei_T(v_S) \le rh$ and $\sum_u d_T^+(u) \ge sh$ with the sum-index $u$ ranging over all ancestors of the vertex $v$ in the tree $T$. By Equation~\ref{eq:holygrail} we may bound the probability for the event $\cE_2^h$ by
\[
\Pr{\cE_2^h} = \Pr{\mZ_n \in \cH} = O(n^{3/2})\sum_{S \in \cH} \Pr{\Gamma Z_{\cA_\cR^{(\ell_S)}}(\rho) = (S, v_S)}.
\]
Let $\eta$ denote the outdegree of the root in the sampler $\Gamma Z_{\cA_\cR^{(1)}}(\rho)$. By the assumptions on the cycle index sum $Z_{\cA_\cR^\omega}$ it follows that $\eta$ has finite exponential moments. Note that the outdegrees along the spine of $\Gamma Z_{\cA_\cR^{(\ell)}}$ are distributed like independent copies of $\eta$. It follows that 
\[
\Pr{\cE_2^h} = O(n^{3/2}) \sum_{\ell=1}^{\min(n, rh)}\Pr{\eta_1 + \ldots + \eta_\ell \ge sh} = O(n^{5/2}) \Pr{\eta_1 + \ldots + \eta_{\lfloor rh \rfloor} \ge sh}.
\]
By the deviation inequality in Lemma~\ref{le:deviation} it follows that there are constants $c,\lambda>0$ such that the above quantity is bounded by a constant multiple of $\exp(5/2 \log(n) + crh - \lambda sh)$. We assumed that $h \ge \sqrt{n}$, hence if we choose $r$ sufficiently small depending only on $s$, $c$ and $\lambda$, it follows that there are constants $C_2, c_2 > 0$ such that $\Pr{\cE_2^h} \le C_2 \exp(-c_2 h)$.

It remains to treat the event $\cE_3^h$. By assumption, for any fixpoint $v \in \cT_n^f$ we have that the height $\hei_{\mY_n}(v)$ is bounded by $\ell(v)$ many independent copies of a random variable $\chi$ having finite exponential moments. Thus \[\Pr{\cE_3^h} \le n \Pr{\chi_1 + \ldots + \chi_{\lfloor sh \rfloor} \ge h}\] with $(\chi_i)_{i \in \ndN}$  a family of independent copies of $\chi$. By the deviation inequality in Lemma~\ref{le:deviation} there are constants $c, \lambda>0$ such that this quantity is bounded by a constant multiple of
$\exp(\log(n)  + \lfloor sh \rfloor c - \lambda h)$. We assumed that $h\ge \sqrt{n}$, hence we may bound this by $\exp(h(\log(n)/\sqrt{n} + sc - \lambda))$. If $s$ is sufficiently small, then it follows that there are constants $C_3, c_3>0$ such that $\Pr{\cE_3^h} \le C_3 \exp(-c_3 h)$ for all $n$ and $h \ge \sqrt{n}$.

Thus there exist constants $C,c>0$ with \[\Pr{\cE^{2h}} \le \sum_i \Pr{\cE_i^h} \le C (\exp(-c h^2/n) + \exp(-ch))\] for all $n$ and $h \ge \sqrt{n}$. This concludes the proof.
\end{proof}

\subsection{Proofs concerning the applications}
\label{sec:proapp}

\subsubsection{Random weighted graphs}

As argued in Section~\ref{sec:unlablocal}, 
Theorems~\ref{te:localunl}, \ref{te:localunl2} and \ref{te:scaling} are special cases of the more general results we established in Sections~\ref{sec:locunl} and \ref{sec:partA}. 

\subsubsection{Random front-rooted $k$-dimensional trees}
We start with global geometric properties, as some intermediate results in there will also be useful in the study of the local properties.
\begin{proof}[Proof of Theorem~\ref{te:scalktree}]
By  Theorem~\ref{te:gibbs} it follows that the largest $\cK^\circ$-component of the random front-rooted $k$-tree $\mK_n$ has size $n + O_p(1)$. Hence Lemma~\ref{le:calkcirc} readily implies that $\mK_n$ converges toward the CRT after rescaling by the same factor as for $\mK_n^\circ$.
\end{proof}

\begin{proof}[Proof of Lemma~\ref{le:calkcirc}]
The random front-rooted $k$-tree $\mK^\circ_n$ corresponds to the random enriched tree $\tilde{\mA}_n^\cR$ for $\cR = \Seq_{\{k\}} \circ\Set$. Hence our framework applies.

We first show a tail-bound for the diameter. By the bijection discussed in Section~\ref{sec:ktree}, we know that the distance between any vertex in $(\cT_n, \beta_n)$ and its offspring is always $1$, as the two vertices are also joined by an edge in the corresponding $k$-tree.  Let $\sqrt{n} \le x \le n$ be given. If $\He(\mK_n^\circ) \ge x$, then it follows that $\He(\cT_n \ge x)$. (Here we define the height with respect to the vertex that corresponds to the root of $\cT_n$.) Hence $\He(\cT_n^f) \ge x/2$  or $|F(v)| \ge x/2$ for some $v \in \cT_n^f$. Let us denote these events by $\cE_1$ and $\cE_2$.  Lemma~\ref{le:blobs}  states  that there are constants $C_1,c_1>0$ such that uniformly for all $n$ and non-negative $x$
\begin{align}
\label{eq:blo}
\Pr{\max_{v \in \cT^f_n}(|f_n(v)| + |F_n(v)|) \ge x} \le C_1 n^{5/2} \exp(-c_1 x).
\end{align}
As we assumed that $\sqrt{n} \le x \le n$, it follows that
\[
	\Pr{\cE_2} \le  C_1 n^{5/2} \exp(-c_1 x/2) \le C_2 \exp(-c_2 x^2/n)
\]
with the constants $C_2, c_2>0$ not depending on $n$ or $x$. As for the event $\cE_1$,  Lemma~\ref{eq:gwttail} implies that there are constants $C_3, c_3>0$ that do not depend on $n$ or $x$ such that
\begin{align*}
	\Pr{\cE_1} &= \sum_{\ell=1}^n \Pr{|\cT_n^f|= \ell} \Pr{\He(\cT_n^f) \ge x/2 \mid |\cT_n^f|=\ell} \\
		&\le C_3 \sum_{\ell=1}^n \Pr{|\cT_n^f|=\ell} \exp(-c_3 x^2/\ell) \\
		&\le C_3 \exp(-c_3 x^2/\ell).
\end{align*}
Thus, there are constants $C_4, c_4>0$ such that 
\[
	\Pr{\He(\mK_n^\circ) \ge x}  \le \Pr{\cE_1} + \Pr{\cE_2} \le C_4 \exp(-c_4 x^2/n)
\]
for all $n$ and all $\sqrt{n} \le x \le n$. It is clear that, by possibly adjusting the constants involved, such an inequality also holds for all $x \ge 0$. This verifies the exponential tail-bound for the diameter of the $k$-tree $\mK_n^\circ$.

It remains to establish the scaling limit. Inequality~\eqref{eq:blo}  implies that with high probability all vertices $v \in \cT^f_n$ have the property that the number of vertices in the forest $F(v)$ is at most $O(\log(n))$. 
This implies that the (pointed) Gromov--Hausdorff distance between the $k$-tree $\mK_n^\circ$ and the subspace corresponding to the vertices of the fixpoint tree $\cT_n^f$  is with high probability at most $O(\log n)$. Consequently, it suffices to show that there is a constant $a_k >0$ such that this subspace rescaled by $a_k / \sqrt{n}$ converges toward the continuum random tree.

Let
\begin{align}
\label{eq:bk}
b_k = \left( k \sum_{i=1}^k \frac{1}{i} \right)^{-1}.
\end{align}
We are going to show that there are exponents $t>0$ and $1/2 < s < 1$  such that with high probability
\begin{align}
	\label{eq:showmee}
	|d_{\mK_n^\circ}(u,v) - b_k d_{\cT_n^f}(u,v)| \le d_{\cT_n^f}(u,v)^s + \log^t n
\end{align}
for all $u,v \in \cT_n^f$. This suffices to complete the proof. Indeed, it follows that with high probability
\begin{align}
\label{eq:tmpname}
d_{\textsc{GH}}( (\cT_n^f, d_{\mK_n^\circ}), (\cT_n^f, b_k d_{\cT_n^f})) \le \Di(\cT_n^f)^s + \log^t n
\end{align}
By Lemma~\ref{le:blobs} we know that
\begin{align}
\label{eq:aa}
\left(\cT_n^f, \frac{\sqrt{(1 + \Ex{\zeta})\Va{\xi}}}{2\sqrt{n}}  d_{\cT^f_n} \right) \convdis (\CRT, d_{\CRT}).
\end{align} in the (pointed) Gromov--Hausdorff sense.  In particular, $\Di(\cT_n^f) = O_p(\sqrt{n})$, and hence it follows from Equation~\eqref{eq:tmpname} that
\begin{align}
\label{eq:bb}
d_{\textsc{GH}}( (\cT_n^f, n^{-1/2} d_{\mK_n^\circ}), (\cT_n^f, b_k n^{-1/2} d_{\cT_n^f})) \convp 0.
\end{align}
Equations~\eqref{eq:aa} and \eqref{eq:bb} then readily imply that the subspace of the $k$-tree $\mK_n^\circ$, that corresponds to $\cT_n^f$, converges toward the CRT after rescaling the metric by $c_k / \sqrt{n}$ with
\[
	c_k = \frac{\sqrt{(1 + \Ex{\zeta})\Va{\xi}}}{2 b_k }.
\]
It follows that
\[
	(\mK_n^\circ, c_k n^{-1/2} d_{\mK_n^\circ}) \convdis (\CRT, d_{\CRT}).
\]
Hence  Inequality~\eqref{eq:showmee} is sufficient to complete the proof. 

Let $u,v \in \cT_n^f$ be arbitrary vertices and let $x \in \cT_n^f$ denote their youngest common ancestor. Let $o$ denote the root of $\cT_n^f$. Then any shortest path in $\mK_n^\circ$ from $o$ to $u$, or $o$ to $v$, or $u$ to $v$ contains at least one vertex with $d_{\mK_n^\circ}$-distance at most $1$ from $x$. Thus the expression
\[
	| d_{\mK_n^\circ}(u,v) - (d_{\mK_n^\circ}(o,u) + d_{\mK_n^\circ}(o,v) - 2 d_{\mK_n^\circ}(o,x))|
\]
is bounded by a fixed constant that does not depend on $u$, $v$ or $x$. Thus, in order to show Inequality~\eqref{eq:showmee}, it suffices to show that for a sufficiently small but fixed constant $c>0$ and  it holds with high probability that
\begin{align}
\label{eq:showm}
|d_{\mK_n^\circ}(o,v) - b_k d_{\cT_n^f}(o,v)| \le c(d_{\cT_n^f}(o,v)^s + \log^t n)
\end{align}
for all $v \in \cT_n^f$. 

To this end, let $\mZ_n$ denote the sampler $\Gamma Z_{\cA_\cR^\omega}(\rho)$ conditioned on having size $n$. That is, $\mZ_n$ is the symmetry corresponding to the $\Sym(\cR)$-enriched tree $(\cT_n, \beta_n)$.  Consider the set $\cH \subset \bigcup_{k=0}^\infty \Sym(\cA_\cR)[k]$ of $\cA_\cR$-symmetries $S=((T, \alpha), \sigma)$ having the property that the there exists a fixpoint vertex $v_S$ in $T$ with the property that $\ell_S := \hei_T(v_S) \ge \log^t(n)$ but the corresponding distance $\hei_*(v_S)$ in the $k$-tree corresponding to $(T, \alpha)$ satisfies
\[
	|\hei_*(v_S) - b_k \ell_S| > c\ell_S^s.
\]
By Lemma~\ref{le:asymptotic} the probability for the sampler $\Gamma Z_{\cA_\cR^\omega}(\rho)$ to have size $n$ is $\Theta(n^{-3/2})$.  It follows from Equation~\ref{eq:holygrail} that 
\begin{align}
\label{eq:mybound}
\Pr{\mZ_n \in \cH} = O(n^{3/2})\sum_{S \in \cH} \Pr{\Gamma Z_{\cA_\cR^{(\ell_S)}}(\rho) = (S, v_S)} \le O(n^{3/2}) \sum_{\log^t n \le \ell \le n} p_\ell,
\end{align}
with $p_\ell$ denoting the probability that the $k$-tree distance $d_\ell$ between the root vertex $v_0$ and the tip $v_\ell$ of the spine $v_0, \ldots, v_\ell$ in $\Gamma Z_{\cA_\cR^{(\ell_S)}}(\rho)$ satisfies
\[
	|d_\ell - b_k \ell| > c \ell^s.
\]

We are going to bound the $p_\ell$ to show that the bound in~\eqref{eq:mybound} converges to zero. 
In order to simplify the calculations, let $u$ denote any fixed vertex of the $*$-place-holder root-front of the $k$-tree corresponding to $\Gamma Z_{\cA_\cR^{(\ell_S)}}(\rho)$.  Let $d_\ell'$ denote the $k$-tree  distance from $u$ to $v_\ell$. Thus 
\[
	|d_\ell - d_\ell'| \le 1
\]
for all $\ell$ and it suffices to study the deviation of $d_\ell'$ from $b_k \ell$.  We are going to exploit properties of the bijection in Section~\ref{sec:ktree} to  take a Markov chain approach.  Consider the set $M_0$ consisting of all $k$ unlabelled vertices. The vertex $v_1$ is incident to $v_0$ and to a $k-1$-element subset $S_0 \subset M_0$. By the construction in Lemma~\ref{le:sampler2} each $(k-1)$-element subset of $M_0$ is equally likely. The distance of $v_1$ to $u$ is given by
\[
	d_1' = 1 + \min_{v \in S_0}d(u,v)
\]
with $d(\cdot, \cdot)$ denoting the $k$-tree distance. This follows from the fact that $M_0 \cup \{v_0\}$ is a $(k+1)$-clique. Setting $M_1 =  \{v_0\} \cup S_0$, the distance of $v_2$ to $u$ is again given by
\[
	d_2' = 1 + \min_{v \in S_1} d(u,v)
\]
with $S_1 \subset M_1$ denoting the $(k-1)$-element subset of $M_1$ that is incident with $v_2$. Here each $(k-1)$-element subset of $M_1$ is (conditionally) equally likely. We may continue this construction, yielding  sequences $S_0, \ldots, S_{\ell-1}$ and $M_0, \ldots, M_{\ell}$, such that for all $0 \le i \le \ell-1$ it holds that
\begin{align}
	\label{eq:dayum}
	d_{i+1}' = 1 + \min_{v \in S_{i}} d(u,v) \qquad \text{ and } \qquad M_{i+1} = \{v_{i}\} \cup S_{i},
\end{align}
and such that, conditioned on $M_{i}$, the subset $S_{i}$ gets drawn uniformly at random among the $(k-1)$-element subsets of $M_{i}$.  Note that for all $v \in M_{i}$ it holds that $d(u,v) = d_i'$ or $d(u,v) = d_i'-1$.  For all $0 \le i \le \ell$ we let $1 \le X_i \le k$ denote the number of vertices in $M_{i}$ with $d(u,v) = d_i'-1$. Equation~\eqref{eq:dayum} implies that for all $0 \le i \le \ell-1$
\begin{align*}
	d_{i+1}' = \begin{cases} d_i', &X_{i+1}<k \\ d_i' +1, &X_{i+1}=k \end{cases}.
\end{align*}
As $d_0'=1$ and $X_0 = 1$, it holds that
\begin{align}
	\label{eq:yuup}
	d_\ell' = 1 + \sum_{i=1}^\ell \one_{X_i = k}.
\end{align}
 Recall that given $M_i$, the set $S_i$ gets drawn uniformly at random from the $(k-1)$-element subsets of $M_i$. Thus, $(X_i)_i$ is Markov chain with the transition probabilities $p_{ij} = \Pr{X_{n+1} = j \mid X_n = i}$ given by the matrix
\[ \mathbf{P} = (p_{ij})_{i,j} =  \left(  \begin{array}{cccccc}
\frac{k-1}{k} &  &  & &   & \frac{1}{k} \\
\frac{2}{k}  & \frac{k-2}{k} & &   &  &  \\
 & \frac{3}{k} & \frac{k-3}{k}  &   &  &  \\ 
 &  & \ddots & \ddots &  &  \\
& &   & \frac{k-1}{k} & \frac{1}{k} &   \\
 &  &  &  & 1 & 
\end{array} 
\right) \in \ndR^{k \times k}.\]
Here we use the convention, that empty spaces in a matrix denote zero entries. 
The stationary distribution is given by
\[
\pi = (\pi_i)_{1 \le i \le k}  =  \frac{1}{\sum_{i=1}^k \frac{1}{i}} \left(1, \frac{1}{2}, \ldots, \frac{1}{k} \right).
\]
This chain is clearly irreducible. However,  unless $k=2$, it is not reversible. As the multiplicative symmetrization $\mathbf{P}^\intercal \mathbf{P}$ is irreducible, we may apply Lemma~\ref{le:markov} to obtain that there are constants $a,b>0$ such that for all sufficiently small $\epsilon>0$ and all $\ell \ge 1$ it holds that
\[
	\Pr{ |d_\ell' -1 - \pi_k \ell| \ge \epsilon \ell} \le a \exp(-b \epsilon^2 \ell).
\]
As $\pi_k = b_k$, follows that for any $1/2<s<1$ we may choose $t>0$ large enough such that the bound in~\eqref{eq:mybound} tends to zero as $n$ becomes large. 

Thus, it holds with high probability  that all vertices $v \in \cT_n^f$ with $d_{\cT_n^f}(o,v) \ge \log^t n$ satisfy
\begin{align}
\label{eq:refme}
|d_{\mK_n^\circ}(o,v) - b_k d_{\cT_n^f}(o,v)| \le c d_{\cT_n^f}(o,v)^s.
\end{align}
This readily verifies Equation~\eqref{eq:showm} and hence completes the proof.
\end{proof}

\begin{proof}[Proof of Theorem~\ref{te:lokt}]
	Let $k_n = o(\sqrt{n})$ be a given sequence.  
	The random front-rooted unlabelled $k$-tree $\mK_n$ may be viewed as a Gibbs partition. Theorem~\ref{te:gibbs} ensures  that $\mK_n$ exhibits a giant component, and that the small fragments converge in total variation toward a Boltzmann limit. Thus, it suffices to show that
	\begin{align}
		\label{eq:toshoow}
		d_{\textsc{TV}}(V_{k_n}(\mK^\circ_n), V_{k_n}(\hat{\mK}^\circ)) \to 0
	\end{align}
	as $n$ becomes large. Set $h_n = 2 b_k k_n + n^{1/4} = o(\sqrt{n})$ with $b_k$ defined in Equation~\eqref{eq:bk}. Theorem~\ref{te:localunlabelled} ensures that
	\[
	d_{\textsc{TV}}((\cT_n, \beta_n)^{<h_n>}, (\cT^{(\infty)}, \beta^{(\infty)})^{<h_n>}) \to 0.
	\]
	as $n$ becomes large. By Equation~\eqref{eq:refme} we know that with high probability the trimmed tree~$(\cT_n, \beta_n)^{<h_n>}$ already contains all information required to determine the $k_n$-neighbourhood $V_{k_n}(\mK^\circ_n)$. This verifies Equation~\eqref{eq:toshoow} and hence completes the proof.
\end{proof}

\begin{proof}[Proof of Theorem~\ref{te:lok2t}]
	Let $k_n = o(\sqrt{n})$ be a given sequence. Theorem~\ref{te:gibbs} ensures  that $\mK_n$ exhibits a giant component, which is distributed like $\mK_(r_n)^\circ$ for some random size $r_n = n + O_p(1)$. A uniformly at random selected vertex of $\mK_n$ lies with high probability in this giant component and outside of its root-front. Conditioned on this event, the random vertex is uniformly distributed among the non-root vertices of the large component. Note that every path from a vertex of the giant component to a vertex of a smaller component must pass through the root-front. Thus, it suffices to show that if $v^*$ is a random vertex of $(\cT_n, \beta_n)$, then
	\[
		d_{\textsc{TV}}(V_{k_n}(\mK^\circ_n, v^*), V_{k_n}(\hat{\mK}^\circ)) \to 0,
	\]
	and with high probability $V_{k_n}(\mK^\circ_n, v^*)$ contains no vertex of the root-front of $\mK_n^\circ$. Set $h_n = 2 b_k k_n + n^{1/4} = o(\sqrt{n})$ with $b_k$ defined in Equation~\eqref{eq:bk}. By  Theorem~\ref{te:localhaupt} it follows that
	\begin{align*}
		d_{\textsc{TV}}( (\mH^n_{[h_n]}, v^*), (\hat{\mH}_{[h_n]}, u^*) ) \to 0.
	\end{align*}
	and that with high probability the vertex $v^*$ has height $\he_{\cT_n}(v^*) > h_n$. By Equation~\eqref{eq:refme} it follows that with high probability it holds that $\he_{\mK_n^\circ}(v^*) > k_n$ and that $(\mH^n_{[k_n]}, v^*)$ already contains all information necessary to determine $V_{k_n}(\mK_n^\circ, v^*)$. This completes the proof.
\end{proof}

\subsubsection{Simply generated P\'olya trees}

We argued in Section~\ref{sec:polya} how Theorems~\ref{te:locpol}, \ref{te:bspol} and \ref{te:scapol} follow from the results on random $\cR$-enriched trees of Sections~\ref{sec:locunl} and \ref{sec:partA}.

\section*{Acknowledgement}
I thank Jetlir Duraj, Markus Heydenreich, Grégory Miermont, and Vitali Wachtel for related discussions.
\bibliographystyle{abbrv}
\bibliography{entrees}

\end{document}

%% file: macros.tex
\newcommand{\ndN}{\mathbb{N}}
\newcommand{\ndZ}{\mathbb{Z}}

\newcommand{\ndR}{\mathbb{R}}

\newcommand{\ndK}{\mathbb{K}}

\renewcommand{\Pr}[1]{\mathbb{P}(#1)}
\newcommand{\Ex}[1]{\mathbb{E}[#1]}
\newcommand{\Va}[1]{\mathbb{V}[#1]}
\newcommand{\one}{\mathbbm{1}}

\newcommand{\om}{\mathbf{w}}

\newcommand{\cC}{\mathcal{C}}
\newcommand{\cL}{\mathcal{L}}
\newcommand{\cM}{\mathcal{M}}
\newcommand{\cF}{\mathcal{F}}
\newcommand{\cB}{\mathcal{B}}

\newcommand{\cN}{\mathcal{N}}

\newcommand{\cS}{\mathcal{S}}
\newcommand{\cG}{\mathcal{G}}
\newcommand{\cT}{\mathcal{T}}
\newcommand{\cR}{\mathcal{R}}
\newcommand{\cA}{\mathcal{A}}
\newcommand{\cH}{\mathcal{H}}
\newcommand{\cK}{\mathcal{K}}

\newcommand{\cX}{\mathcal{X}}

\newcommand{\cY}{\mathcal{Y}}

\newcommand{\mA}{\mathsf{A}}
\newcommand{\mC}{\mathsf{C}}

\newcommand{\mF}{\mathsf{F}}

\newcommand{\mG}{\mathsf{G}}

\newcommand{\mH}{\mathsf{H}}
\newcommand{\mT}{\mathsf{T}}
\newcommand{\mS}{\mathsf{S}}
\newcommand{\mR}{\mathsf{R}}

\newcommand{\mK}{\mathsf{K}}

\newcommand{\mU}{\mathsf{U}}
\newcommand{\mX}{\mathsf{X}}
\newcommand{\mY}{\mathsf{Y}}
\newcommand{\mZ}{\mathsf{Z}}

\newcommand{\me}{\mathsf{e}}
\newcommand{\mf}{\mathsf{f}}

\newcommand{\CRT}{\mathcal{T}_\me}

\newcommand{\Sym}{\text{Sym}}

\newcommand{\cE}{\mathcal{E}}

\newcommand{\hei}{\text{h}}
\newcommand{\He}{\textnormal{H}}
\newcommand{\Di}{\textnormal{D}}
\newcommand{\emb}{\textnormal{emb}}

\newcommand{\he}{\text{h}}

\newcommand{\spa}{\mathsf{span}}

\newcommand{\eqdist}{\,{\buildrel (d) \over =}\,}

\newcommand{\convdis}{\,{\buildrel d \over \longrightarrow}\,}
\newcommand{\convp}{\,{\buildrel p \over \longrightarrow}\,}

\newcommand{\Set}{\textsc{SET}}
\newcommand{\Seq}{\textsc{SEQ}}

\newtheorem{theorem}{Theorem}[section]

\newtheorem{corollary}[theorem]{Corollary}
\newtheorem{proposition}[theorem]{Proposition}
\newtheorem{lemma}[theorem]{Lemma}

\numberwithin{equation}{section}
